\documentclass[12pt,a4paper]{article}
\usepackage{amsmath,amssymb}
\usepackage{amsthm}
\usepackage{txfontsb}
\usepackage{newpxtext, newpxmath}
\usepackage{mathrsfs}
\usepackage{ytableau}
\usepackage{caption}
\usepackage{tikz}
\usepackage{mathtools}
\usepackage{amsxtra}
\usepackage{textcomp}
\usepackage{color}
\usepackage{titlesec}
\usepackage{abstract}
\usepackage{enumitem}
\usepackage[normalem]{ulem}

\titleformat{\section}
  [block]
  {\normalfont\large\scshape\filcenter}
  {\thesection}
  {0.5em}
  {}
    
\titleformat{\part}
  [block]
  {\normalfont\large\scshape\filcenter}
  {\thepart}
  {0.5em}
  {}

\numberwithin{equation}{section}
\usetikzlibrary{arrows,patterns}
\usetikzlibrary{cd}
\usetikzlibrary{arrows.meta} 
\counterwithin*{part}{section}

\theoremstyle{definition}
\newtheorem{thm}{Theorem}[section]
\newtheorem{prop}[thm]{Proposition}
\newtheorem{dfn}[thm]{Definition}
\newtheorem{lem}[thm]{Lemma}

\newtheorem{ex}[thm]{Example}
\newtheorem{rmk}[thm]{Remark}

\renewenvironment{abstract}
  {
    \par\small
    \begin{list}{}{
      \setlength{\leftmargin}{2em}
      \setlength{\rightmargin}{2em}
      \setlength{\itemindent}{0pt}
    }
    \item\relax
    \textbf{Abstract.}\enspace\ignorespaces
  }
  {\end{list}\par\medskip}
\AtEndEnvironment{thm}{\qed}
\AtEndEnvironment{prop}{\qed}
\AtEndEnvironment{dfn}{\qed}
\AtEndEnvironment{lem}{\qed}
\AtEndEnvironment{cor}{\qed}
\AtEndEnvironment{ex}{\qed}
\AtEndEnvironment{rmk}{\qed}

\newcommand{\p}{\mathbb{P}}

\newcommand{\al}{\alpha}
\newcommand{\bp}{\bbackslash\hspace{0.1pt}'}

\newcommand{\ol}{\overline}
\newcommand{\w}{\widehat}
\newcommand{\xlr}{\xleftrightarrow}
\newcommand{\ws}{\xleftarrow[]{\!\!\!\!\!\text{WS}\!\!\!\!\!}}
\newcommand{\hm}{\xleftarrow[]{\!\!\!\!\!\text{HM}\!\!\!\!\!}}
\newcommand{\wsm}{\xleftarrow[]{\!\!\!\!\!\text{WSM}\!\!\!\!\!}}
\newcommand{\gm}{\xleftarrow[]{\!\!\!\!\!\text{GM}\!\!\!\!\!}}
\newcommand{\m}{\xleftarrow[]{\!\!\!\!\!\text{OM}\!\!\!\!\!}}
\newcommand{\dm}{\xleftarrow[]{\!\!\!\!\!\text{DM}\!\!\!\!\!}}
\newcommand{\wsmx}{\xleftarrow[X]{\!\!\!\!\!\text{WSM}\!\!\!\!\!}}
\newcommand{\gmx}{\xleftarrow[X]{\!\!\!\!\!\text{GM}\!\!\!\!\!}}
\newcommand{\x}{\xleftarrow[X]{}}
\newcommand{\ox}{\xleftarrow[\hspace{-6.5pt}\omega(X)\hspace{-6.5pt}]{}}

\newcommand{\red}{\textcolor{red}}
\newcommand{\blue}{\textcolor{blue}}
\usepackage{pict2e}

\DeclareMathOperator{\bbackslash}{
\setlength{\unitlength}{1.6ex}
  \begin{picture}(1,1)
  \polyline(0,1)(0,0)(1,0)(1,1)(0,1)(1,0)
  \end{picture}
  }

\makeatletter
\renewcommand{\@maketitle}{%
  \newpage
  \begingroup
    \centering
    \let\footnote\thanks
    {\Large \@title\par}
    \vspace{0.8em}
    {\normalsize \@author\par}
  \endgroup
  \vspace{1em}
}
\makeatother

\definecolor{linkblue}{rgb}{0,0,0.6}

\usepackage{hyperref}
\hypersetup{
  colorlinks=true,
  allcolors=linkblue,
  linktoc=page
}

\begin{document}

\title{A duality-preserving extension of the Worley--Sagan insertion and Haiman's mixed insertion for the hyperoctahedral group} 
\author{Masato Nakagiri}
\date{}
\addtolength{\skip\footins}{6pt}
\maketitle

\begingroup
\makeatletter
\renewcommand{\@makefntext}[1]{\noindent #1}
\makeatother
\footnotetext{%
\normalfont
\textit{\textup{2020} Mathematics Subject Classification.}\enspace
Primary 05A19; Secondary 05A05.\par
}
\endgroup

\begin{abstract}
The Worley--Sagan insertion and Haiman's mixed insertion are insertion algorithms for shifted Young tableaux, and each of them gives a Robinson--Schensted-type correspondence between the symmetric group of degree $n$ and a set consisting of certain pairs of same-shape shifted Young tableaux with $n$ cells. It is a known fact that these two insertions are dual to each other. Our purpose is to give an extension of these two insertions without losing the duality relationship. The extended ones will be insertions producing pairs of shifted tableaux from colored permutations. Our extension of the Worley--Sagan insertion is different from the restriction of Sagan's own ``Knuth version'' to colored permutations. In proving the duality between our extended insertions, we ``embed'' them into Shimozono and White's doubly mixed insertion for unshifted tableaux by ``doubling'' shifted tableaux and use the self-duality of the doubly mixed insertion shown by Shimozono and White. 
\end{abstract}

\begingroup
\let\bfseries\mdseries
\titleformat*{\section}{\normalfont\large\scshape\filcenter}
\tableofcontents
\endgroup

\section{Introduction}
The Robinson--Schensted correspondence is a bijective correspondence between $S_n$, the symmetric group of degree $n$, and the set of all pairs of same-shape standard tableaux with $n$ cells. This correspondence is defined by using the Schensted insertion \cite{sch}, a combinatorial algorithm which produces a new Young tableau from a tableau and a letter. There are many variations of the Schensted insertion and each of them gives a bijective correspondence between a set of some sort of words and a certain pairs of Young tableaux by suitably defining the recording tableaux. We call this type of correspondence a Robinson--Schensted-type correspondence. 

Our purpose in this paper is to give extensions of the two insertion schemes producing shifted tableaux, one developed by D. Worley \cite{wo} and B. Sagan \cite{sa}, and the other developed by M. Haiman \cite{h}, to the situation involving colored permutations (with $2$ colors, a terminology used by M. Shimozono and D. White \cite{sh-wh}), and to show that the Robinson--Schensted-type correspondence resulting from the two extended insertion schemes are dual to each other in the same sense that the ones resulting from the original Worley--Sagan insertion and Haiman's insertion are dual to each other. We are using the name Worley--Sagan insertion following \cite{h}.
Let us note in advance that our extension of the Worley--Sagan insertion is different from simply adding colors to the $P$-tableau reflecting the colors in the permutation, or from restricting Sagan's own ``Knuth version'' to colored permutation matrices \cite{sa} in a suitable sense. 
The definition of our extension of the Worley--Sagan insertion was inspired by M. Ishikawa's Ph.D. thesis \cite{i}. 

The classical Robinson--Schensted correspondence is a bijection sending $\pi\in S_n$, the symmetric group of degree $n$, to $(P(\pi), Q(\pi))$, a pair of same-shape standard tableaux with $n$ cells. The tableau $P(\pi)$ is given by inserting successively $\pi(1), \pi(2), \dots, \pi(n)$ into the empty tableau using the Schensted insertion, and is called the $P$-symbol or the insertion tableau. 
At each insertion, the shape of the insertion tableau gains one new cell. $Q(\pi)$ is the tableau of the same shape as $P(\pi)$ and in which the cell gained at the insertion of $\pi(k)$ is filled with $k$ ($1\leq k\leq n$). $Q(\pi)$ is called the $Q$-symbol or the recording tableau. 
We adopt the notation ``$\pi\xlr{}(T, U)$'' to indicate that $\pi$ corresponds to $(T, U)$ in the Robinson--Schensted correspondence.

The Robinson--Schensted correspondence (or the Schensted insertion) has an interesting property called the self-duality. 
This means that if $\pi\xlr{}(T, U)$, then $\pi^{-1}\xlr{}(U, T)$ holds for $\pi\in S_n$. 
Some variations of the Schensted insertion enjoy similar self-dualities, but in many cases dualities occur between two different variations. 
In particular, the duality between the Worley--Sagan insertion and Haiman's mixed insertion for shifted tableaux forms the basis of this paper. 
This duality means that if $\pi\xlr{\!\!\!\!\!\text{WS}\!\!\!\!\!}(P, Q)$, then $\pi^{-1}\xlr{\!\!\!\!\!\text{HM}\!\!\!\!\!}(Q, P)$ holds for $\pi\in S_n$. 
Here, ``$\,\xlr{\!\!\!\!\!\text{WS}\!\!\!\!\!}\,$'' and ``$\,\xlr{\!\!\!\!\!\text{HM}\!\!\!\!\!}\,$'' stand for the Robinson--Schensted-type correspondences given by the Worley--Sagan insertion and Haiman's mixed insertion for shifted tableaux, respectively.

Our extensions of the Worley--Sagan insertion and Haiman's mixed insertion are named the Worley--Sagan mixed insertion and the generalized mixed insertion in this paper. They are defined in Definition~\ref{WSM} and Definition~\ref{GM}, respectively. Our goal is to prove Theorem~\ref{WSMGMd}, which states the duality between these two extended insertions. 
This duality ``contains'' the original duality in it. 
That is, if we restrict our Robinson--Schensted-type correspondences for colored permutations to permutations, namely the elements of $S_n$, the newly obtained  duality coincides with the already-known duality between the Worley--Sagan insertion and Haiman's mixed insertion.
Our proof of Theorem~\ref{WSMGMd} follows the line of Haiman's proof \cite{h} of the duality between the Worley--Sagan insertion and his own mixed insertion in the sense that we ``double'' shifted tableaux and words, and ``embed'' the present duality into an already-known duality of insertions for unshifted tableaux. Haiman embedded the duality between the Worley--Sagan insertion and his own mixed insertion into the duality between the left-right insertion and his own mixed insertion for unshifted tableaux. 

Doubling a shifted tableau means producing an unshifted tableau with $2n$ cells from a shifted tableau with $n$ cells. Roughly speaking, the procedure attaches the ``tranpose'' of itself to a shifted tableau. Doubling a word refers to a certain procedure producing a word of length $2n$ from a word of length $n$. 
In our proof of Theorem~\ref{WSMGMd}, we use two types of doubling operators $\bbackslash$ and $\bp$ for shifted tableaux, and two types of doubling operators $d_1$ and $d_2$ for words. 
The definitions of $\bbackslash$\ , $\bp$, $d_1$ and $d_2$ are given in Definition~\ref{squareop}, Definition~\ref{d1} and Definition~\ref{d2}, respectively. 
By doubling shifted tableaux and words, the duality in Theorem~\ref{WSMGMd} will be embedded into the self-duality of the doubly mixed insertion by Shimozono and White \cite{sh-wh}, an insertion for unshifted tableaux which contains Haiman's left-right insertion and mixed insertion as special cases. 
This means that if we transport the Robinson--Schensted-type correspondences given by the Worley--Sagan mixed insertion and the generalized mixed insertion via the process of doubling tableaux and words, it coincides with a part of the Robinson--Schensted-type correspondence given by the doubly mixed insertion, and the duality in Theorem~\ref{WSMGMd}, after doubling tableaux and words, is obtained as a restriction of the self-duality of the doubly mixed insertion.

Let us explain the organization of this paper. 
In \S 2, we prepare some terminology for later use. 
In \S 3, we review the definition of unshifted tableaux and Shimozono and White's doubly mixed insertion. In \S 4, we review the definition of shifted tableaux, the Worley--Sagan insertion and Haiman's mixed insertion. We are using some space to review these ``historical'' procedures in \S\S3--4, since our extensions use the ideas from them extensively and in intricate combinations, and we hope that introducing notions such as the row/column nature of letters in the context of these prototypical procedures eases the understanding of our extensions. In \S 5, we define the Worly-Sagan mixed insertion and the generalized mixed insertion. In \S 6, we will show that the Worley--Sagan mixed insertion can be embedded into the doubly mixed insertion by doubling tableaux and words. \S 7 is the generalized mixed insertion version of \S 6. We embed the generalized mixed insertion into the doubly mixed insertion. In \S 8, using the results obtained in \S 6 and \S 7, we will show that the Worley--Sagan mixed insertion and the generalized mixed insertion are dual to each other. 

This paper grew out of the author's master's thesis \cite{n}.

\section{Letters, Words, and Cells in the Quadrant and the Octant}
A tableau is an object obtained by filling the boxes of a shifted or unshifted Young diagram with letters.
All insertion algorithms defined in this paper are operations by which we obtain a new tableau out of a tableau and a letter. We slide letters written in our original tableau in order to get an empty cell in which our new letter is added. The manner of sliding letters is determined by comparision of letters. 
In this paper, we adopt the set of letters $\p\cup\ol{\p}$, which consists of positive numbers and barred positive numbers $\ol{1}, \ol{2}, \ol{3}, \dots$. 
\begin{dfn}\label{letorder}
Let $\mathbb{P}$ be the set of positive integers. 
We define $\p\cup\ol{\p}$ to be the set $\{1, 2, \dots, \ol{1}, \ol{2}, \dots\}$ consisting of the positive integers and the letters $\ol{1}, \ol{2}, \ol{3}, \dots$ obtained by adding bars to the positive integers. The elements of $\p\cup\ol{\p}$ are called \textit{colored letters}. We endow $\p\cup\ol{\p}$ with a total order such that $\ol{1}<1<\ol{2}<2<\ol{3}<3<\dots$.
We call elements of $\ol{\p}$ \textit{barred letters} and elements of $\mathbb{P}$ \textit{unbarred letters}. In this paper, we use Greek characters $\alpha, \beta, \gamma, \dots$ to stand for elements of $\p\cup\ol{\p}$, and we use Latin characters $a, b, c, \dots$ to stand for elements of $\mathbb{P}$.
\end{dfn}

Next, we define two operators on letters which will appear repeatedly in our argument.
\begin{dfn}
We define
the operation $|\cdot| : \p\cup\ol{\p}\to\p$, which removes a bar from a barred letter and keeps an unbarred letter unchanged. That is, $|\bar{a}|=a$ and $|a|=a$ for $a\in\p$. 
We define the operator $\ol{\cdot} : \p\cup\ol{\p}\to\p\cup\ol{\p}$, which switches the colors of the letters: 
\begin{align*}
\ol{\alpha}=\begin{cases} \ol{a} & \text{if $\alpha=a$, unbarred}\\
a&
\text{if $\alpha=\ol{a}$, barred}\\
\end{cases}, 
\end{align*}
where $a\in\p$. 
\end{dfn}

A Robinson--Schensted-type correspondence is a bijection between a set of words and a set of certain pairs of tableaux.  
Words are defined as follows: 
\begin{dfn} Let $S$ be a set. 
A finite sequence $s_1\dots s_n$ of elements of $S$ is called an \textit{$S$-word of length n}. An \textit{$S$-biword of length n} is defined to be an arrangement $\begin{psmallmatrix}s_1\dots s_n \\ t_1\dots t_n\end{psmallmatrix}$ of two $S$-words of length $n$. 
Furthermore, for a $\left(\p\cup\ol{\p}\right)$-word $\al=\al_1\dots\al_n$, we define a $\p$-word $|\al|$ by 
$|\al|=|\al_1|\dots|\al_n|$. 
\end{dfn}
An element of $S_n$ is also called a permutation. 
As is commonly done, we identify $\pi\in S_n$ such that $\pi(k)=a_k$ ($1\leq k\leq n$) with
the $\mathbb{P}$-word $a_1\dots a_n$. 

\begin{ex}
Define $\sigma, \tau\in S_3$ by 
$\sigma(1)=2, \sigma(2)=1, \sigma(3)=3, \tau(1)=2, \tau(2)=3, \tau(3)=1$. 
In the identification mentioned above, $\sigma=213$ and $\tau=231$. 
\end{ex}

In dealing with tableaux and diagrams, we use the Cartesian coordinate system in the following manner: 
Set $\mathcal{Q}=\mathbb{P}^2, \mathcal{O}=\{\,(i, j)\in\mathbb{P}^2 \mid i\leq j\,\}$. We call $\mathcal{Q}$ and $\mathcal{O}$ the \textit{quadrant} and the \textit{octant}, respectively. 
As illustrated in Figure~\ref{fig1}, we draw a Cartesian plane with its $x$-axis in the downward direction and its $y$-axis in the rightward direction. 

We identify each element $(i, j)$ of $\mathcal{Q}$ with a square of side length $1$ centered at $(i, j)$, and each such square is called a \textit{cell}. In Figure~\ref{fig1}, the cells $(2, 6)$ and $(7, 2)$ are indicated by gray squares.
In this paper, a subset $\mathcal{A}$ of $\mathcal{Q}$ is represented by depicting all of the cells which correspond to the elements of $\mathcal{A}$. 
For a positive integer $i$, we call the set 
$\mathcal{R}_i=\{\,(i, k) \mid k\in\mathbb{P}\,\}$ the \textit{$i$-th row} of $\mathcal{Q}$ and the set $\mathcal{C}_i=\{\,(k, i) \mid k\in\mathbb{P}\,\}$ the \textit{$i$-th column} of $\mathcal{Q}$, and $\mathcal{A}\cap\mathcal{R}_i$ and $\mathcal{A}\cap\mathcal{C}_i$ the \textit{$i$-th row} and the \textit{$i$-th column} of $\mathcal{A}$, respectively. 
The direction in which the row numbers increase is called \textit{down} and the opposite direction is called \textit{up}. 
The direction in which the column numbers increase is called \textit{right} and the opposite direction is called \textit{left}. 
The left figure in Figure~\ref{fig2} illustrates $\mathcal{Q}$ and the right figure illustrates $\mathcal{O}$. 
We say a cell in $\mathcal{Q}$ or $\mathcal{O}$ is a \textit{diagonal} cell if it is represented as $(i, i)$ for some positive integer $i$, and an \textit{off-diagonal} cell if it is not diagonal. In Figure~\ref{fig1}, the diagonal cells up to the seventh row are indicated by shaded squares.
We denote the set of diagonal cells by $\Delta$. In other words, $\Delta=\{\,(i, i) \mid i\in\mathbb{P} \,\}$. 

\begin{figure}[htbp]\label{fig1}
\centering
\begin{tikzpicture}[x=1.0cm, y=1.0cm, yscale=-1]
\draw[step=1.0cm, help lines, color=gray!30, dashed] (-0.9,-0.9) grid (8.25,8.25);
\draw[->,ultra thick] (0,-1)--(0,8.5) node[below]{$x$};
\draw[->,ultra thick] (-1,0)--(8.5,0) node[right]{$y$};
\fill[black!30] (5.5,1.5) rectangle (6.5, 2.5);
\fill[black!30] (1.5,6.5) rectangle (2.5, 7.5);
\foreach \x in {1, ..., 7} \fill[pattern=north east lines] (\x-0.5,\x-0.5) rectangle (\x+0.5,\x+0.5);
\foreach \x in {1, ..., 8} \draw (0, \x)node[left]{$\x$};
\foreach \x in {1, ..., 8} \draw (\x,0) node[above]{$\x$};
\foreach \x in {1, ..., 7} \draw[line width=1.5pt] (\x-0.5,\x-0.5)--(\x+0.5,\x-0.5)(\x+0.5,\x-0.5)--(\x+0.5,\x+0.5)(\x+0.5,\x+0.5)--(\x-0.5,\x+0.5)(\x-0.5,\x+0.5)--(\x-0.5,\x-0.5);
\draw[line width=1.5pt] (5.5,1.5)--(6.5,1.5)(6.5,1.5)--(6.5,2.5)(6.5,2.5)--(5.5,2.5)(5.5,2.5)--(5.5,1.5);
\draw[line width=1.5pt] (1.5,6.5)--(2.5,6.5)(2.5,6.5)--(2.5,7.5)(2.5,7.5)--(1.5,7.5)(1.5,7.5)--(1.5,6.5);
\draw (0, 0) node[above left] {O}; 
\draw (6, 2) node{$(2, 6)$};
\draw (2, 7) node{$(7, 2)$};
\draw (7.9, 7.7) node{$\ddots$};
\end{tikzpicture}
\caption{\label{fig1}}
\end{figure}

\begin{figure}[htbp]\label{fig2}
\begin{tikzpicture}[x=0.8cm, y=0.8cm, yscale=-1]
\draw[step=0.8cm, help lines, color=gray!30, dashed] (-0.9,-0.9) grid (8.25,8.25);
\draw[->,ultra thick] (0,-1)--(0,8.5) node[below]{$x$};
\draw[->,ultra thick] (-1,0)--(8.5,0) node[right]{$y$};
\foreach \x in {1, ..., 8} \draw (0, \x)node[left]{$\x$};
\foreach \x in {1, ..., 8} \draw (\x,0) node[above]{$\x$};
\foreach \x in {1.5, ..., 7.5} \draw[line width=1.5pt] (0.5,\x)--(7.75,\x);
\foreach \x in {0.5, ..., 6.5} \draw[line width=1.5pt] (\x,7.75)--(\x, 0.5);
\draw[line width=1.5pt](0.5,0.5)--(7.75,0.5);
\draw[line width=1.5pt](7.5,0.5)--(7.5,7.75);
\draw (0, 0) node[above left] {O}; 
\foreach \x in {1, ..., 7} \draw (8, \x)node{$\dots$};
\foreach \x in {1, ..., 7} \draw (\x, 8)node{$\vdots$};
\draw(8,8)node{$\ddots$};
\end{tikzpicture}
\begin{tikzpicture}[x=0.8cm, y=0.8cm, yscale=-1]
\draw[step=0.8cm, help lines, color=gray!30, dashed] (-0.9,-0.9) grid (8.25,8.25);
\draw[->,ultra thick] (0,-1)--(0,8.5) node[below]{$x$};
\draw[->,ultra thick] (-1,0)--(8.5,0) node[right]{$y$};
\foreach \x in {1, ..., 8} \draw (0, \x)node[left]{$\x$};
\foreach \x in {1, ..., 8} \draw (\x,0) node[above]{$\x$};
\foreach \x in {1.5, ..., 7.5} \draw[line width=1.5pt] (\x-1,\x)--(7.75,\x);
\foreach \x in {0.5, ..., 6.5} \draw[line width=1.5pt] (\x,\x+1)--(\x, 0.5);
\draw[line width=1.5pt](0.5,0.5)--(7.75,0.5);
\draw[line width=1.5pt](7.5,0.5)--(7.5,7.75);
\draw (0, 0) node[above left] {O}; 
\foreach \x in {1, ..., 7} \draw (8, \x)node{$\dots$};
\draw(8,8)node{$\ddots$};
\end{tikzpicture}
\caption{\label{fig2}}
\end{figure}

\section{Unshifted Tableaux and the Doubly Mixed Insertion}
In this section, we define unshifted tableaux and review the doubly mixed insertion by Shimozono and White. The doubly mixed insertion is defined in Definition~\ref{DM} based on the ordinary mixed insertion and the dual mixed insertion defined in Definition~\ref{OMDM}. The main objects of study in our current work are the two insertions for shifted tableaux and the duality between them. However, our argument depends on the theory of unshifted tableaux through doubling shifted tableaux in order to show our duality. 
The concrete argument about doubling shifted tableaux will be given in \S 6 and \S 7. This section is a preparation for these sections.

Roughly speaking, an unshifted tableau is an object obtained by filling letters in a certain arrangement of cells in $\mathcal{Q}$.
The arrangements of cells used here are in one-to-one correspondence with the integer partitions. 
First of all, we begin with the definition of integer partitions. 

\begin{dfn}\label{ptn}
Let $n$ be a non-negative integer. 
An \textit{integer partition}, or simply a \textit{partition}, of $n$ is a finite sequence of positive integers $\lambda=(\lambda_1,\dots, \lambda_\ell)$ satisfying $\lambda_1+\cdots+\lambda_\ell=n$ and $\lambda_1\geq\cdots\geq\lambda_\ell$. The terms $\lambda_k$ ($1\le k\le \ell$) of the sequence are called the \textit{parts} of $\lambda$, and the number of parts $\ell$ is called the \textit{length} of $\lambda$ and denoted by $\ell(\lambda)$. 
For $n=0$, we regard the empty sequence, denoted by $\emptyset$, as the unique partition of $0$ with $\ell(\emptyset)=0$. 
We write $\mathcal{P}(n)$ for the set of partitions of $n$, and let $\mathcal{P} = \coprod_{n\geq 0} \mathcal{P}(n)$ be the set of all partitions.
\end{dfn}

Using partitions, unshifted diagrams and unshifted tableaux are defined as follows. 
\begin{dfn}\label{unshifted}
Let $\lambda$ be a partition. The \textit{unshifted diagram} of shape $\lambda$ is
\[D(\lambda)=\{\,(i, j)\in\mathcal{Q} \mid 1\leq i\leq\ell(\lambda), 1\leq j\leq\lambda_i\,\}.\]
An \textit{unshifted tableau} of shape $\lambda$ is a map $T:D(\lambda)\to \p\cup\ol{\p}$ which satisfies the following two conditions: 
\begin{enumerate}
\item\, If $(i, j), (i+1, j)\in D(\lambda)$, then $T(i, j)<T(i+1, j)$,
\item\, If $(i, j), (i, j+1)\in D(\lambda)$, then $T(i, j)\leq T(i, j+1)$.
\end{enumerate}

An unshifted tableau $T:D(\lambda)\to \p\cup\ol{\p}$ is illustrated by filling the cell $(i, j)$ with $T(i, j)$ for all $(i, j)\in D(\lambda)$.
For example, the unshifted diagram of shape $\lambda=(5, 3, 3, 1)$ and an example of an unshifted tableau of shape $\lambda$ are shown in the following figure: 

\[
\ytableausetup{centertableaux}
\hspace{10pt}
\ydiagram{5, 3, 3, 1}
\hspace{60pt}
\begin{ytableau}
\ol{1} & \ol{1} & 1 & 4 & 5 \\
1 & \ol{2} & 5 \\
5 & 7 & 7 \\
6
\end{ytableau}
\]

For an unshifted tableau $T$ of shape $\lambda$, its domain as a map (that is, the unshifted diagram $D(\lambda)$) is also denoted by $D(T)$. For any subset $D\subseteq D(T)$, $T|_D$ is the restriction of the map $T$ to $D$, which is illustrated as the collection of the cells in $D$, with the letters in those cells unchanged from $T$. 
\end{dfn}

\begin{rmk}
 In literature, the term ``unshifted tableau'' often means an arbitrary map $T:D(\lambda)\to\mathscr{L}$, where $\mathscr{L}$ means a set of relevant letters. 
In general, a map $T : D(\lambda)\to\mathscr{L}$ which satisfies (1) and (2) is called ``an unshifted semistandard tableau''. 
\cite{sh-wh} employs the terminology ``semistandard colored tableau'' to indicate an unshifted tableau defined in Definition~\ref{unshifted}. 
The terminology in which we impose (1) and (2) on the definition of an unshifted tableau is based on Macdonald's book \cite{m}. 
\end{rmk}

Let $T$ be an unshifted tableau. We sometimes identify $T$ with a map $\tilde{T}:\mathcal{Q}\to \p\cup\ol{\p}\cup\{\infty\}$ such that 
\begin{align*}
\tilde{T}(c)=\begin{cases}T(c) & \text{if $c\in D(T)$}, \\
\infty & \text{if $c\in\mathcal{Q}\setminus D(T)$}. 
\end{cases}
\end{align*}
When we treat $\tilde{T}$, we extend the total order on $\p\cup\ol{\p}$ (Definition~\ref{letorder} ) to $\p\cup\ol{\p}\cup\{\infty\}$ by setting $\alpha<\infty$ for all $\alpha\in\p\cup\ol{\p}$. 
Note that the conditions (1) and (2) in Definition~\ref{unshifted} are imposed on the cells in $D(T)$
only.
 
Next, we define the ordinary mixed insertion and the dual mixed insertion, on which the doubly mixed insertion is based. 
\begin{dfn}(\cite[Definition~3.1]{h}, \cite[\S 3.3, \S 3.4]{sh-wh})\label{OMDM}
The doubly mixed insertion is a process to produce a new unshifted tableau from an unshifted tableau $T$ and 
$\begin{psmallmatrix}\alpha\\\beta\end{psmallmatrix}$, a column of a certain $\left(\p\cup\ol{\p}\right)$-biword (see Definition~\ref{dcb} below), by one of the two rules: namely either by an \textit{ordinary mixed insertion} if $\alpha\in\p$ or by a \textit{dual mixed insertion} if $\alpha\in\ol{\p}$. $T$ will be identified with $\tilde{T}$ containing $\infty$ in $\mathcal{Q}\setminus D(T)$ as described above. 
It is not difficult to verify that the modified tableaux actually qualifies as a tableau after each step.

In the ordinary mixed insertion, among both $\beta$ and those already written in $T$, unbarred letters are given the \textit{row nature} and barred letters are given the \textit{column nature}. These terms are used as a mnemonic for the way these letters move, as explained in the following description.

\begin{itemize}
\item\textbf{Ordinary Mixed Insertion}

This rule is applied if \uline{$\alpha\in\p$}.

\begin{enumerate} 
\item[Step 1. ]
If \uline{$\beta\in\p$}, read the first \uline{row} of $T$ from left to right, searching for the first cell $c$ containing a letter $\beta'$ with $\beta'>\beta$; if \uline{$\beta\in\ol{\p}$}, read the first \uline{column} of $T$ from top to bottom, searching for the first cell $c$  containing a letter $\beta'$ with $\beta'\geq\beta$. In either case, remember $\beta'$, modify $T$ by replacing the content of $c$ by $\beta$, and call the modified tableau $T'$. If $\beta'=\infty$, then $T'$ is the result of inserting 
$\begin{psmallmatrix}\alpha\\\beta\end{psmallmatrix}$
into $T$; otherwise proceed to Step 2. 

\item[Step 2. ]
If \uline{$\beta'\in\p$}, read the \uline{row} of $T'$ immediately below $c$ from left to right, searching for the first cell $c'$ containing a letter $\beta''$ with $\beta''>\beta'$; 
if \uline{$\beta'\in\ol{\p}$}, read the \uline{column} of $T'$ immediately to the right of $c$ from top to bottom, searching for the first cell $c'$ containing a letter $\beta''$ with $\beta''\geq\beta'$. In either case, remember $\beta''$, modify $T'$ by replacing the content of $c'$ by $\beta'$, and call the modified tableau $T''$. If $\beta''=\infty$, then $T''$ is the result of inserting 
$\begin{psmallmatrix}\alpha\\\beta\end{psmallmatrix}$
into $T$; otherwise rename $T''$, $\beta''$ and $c'$ as $T'$, $\beta'$ and $c$, respectively, and repeat Step 2.  
\end{enumerate}
This process will terminate since $c'$ will eventually be a cell in $\mathcal{Q}\setminus D(T)$, whose content is $\infty$. This is because if $c'=(i', j')$, at least one of $i'$ or $j'$ increases by undergoing a process of Step 2. The resulting tableau of this process is denoted by $T\m\beta$.

In the dual mixed insertion, we interchange the role of rows and columns, assigning the column nature to unbarred letters and the row nature to barred letters.

\item\textbf{Dual Mixed Insertion}

This rule is applied if \uline{$\alpha\in\ol{\p}$}.

\begin{enumerate}
\item[Step 1. ]
If \uline{$\beta\in\p$}, read the first \uline{column} of $T$ from top to bottom, searching for the first cell $c$ containing a letter $\beta'$ with $\beta'\geq\beta$; if \uline{$\beta\in\ol{\p}$}, read the first \uline{row} of $T$ from left to right, searching for the first cell $c$ containing a letter $\beta'$ with $\beta'>\beta$. The rest of Step 1 is the same as Step 1 in the description of the ordinary mixed insertion case. 
\item[Step 2. ]
If \uline{$\beta'\in\p$}, read the \uline{column} of $T'$ immediately to the right of $c$ from top to bottom, searching for the first cell $c'$ containing a letter $\beta''$ with $\beta''\geq\beta'$; if \uline{$\beta'\in\ol{\p}$}, read the \uline{row} of $T'$ immediately below $c$ from left to right, searching for the first cell $c'$ containing a letter $\beta''$ with $\beta''>\beta'$. The rest of Step 2 is the same as Step 2 in the description of the ordinary mixed insertion case. 
\end{enumerate}
For the same reason as the ordinary mixed insertion case, this process will terminate. 
The resulting tableau of this process is denoted by $T\dm\beta$. 
\end{itemize}
\end{dfn}

The doubly mixed insertion is an insertion algorithm defined for a certain type of $\left(\p\cup\ol{\p}\right)$-biword called doubly colored biword. 
We need an operator $(\cdot)^\text{neg}:\p\cup\ol{\p}\to\p\cup(-\p)$ defined by
\begin{align*}
\alpha^\text{neg}=\begin{cases} \al & \text{if $\alpha$ is unbarred, }\\
-|\al|&
\text{if $\alpha$ is barred} \\
\end{cases}
\end{align*}
where $\p\cup(-\p)$ has the ordering induced from that of $\mathbb{Z}$.
\begin{dfn}\label{dcb}
A \textit{doubly colored biword} is a $\left(\p\cup\ol{\p}\right)$-biword
$\pi=\begin{psmallmatrix}
\alpha_1& \dots & \alpha_n \\
\beta_1& \dots & \beta_n \\
\end{psmallmatrix}
$
which satisfies the following conditions: 
\begin{itemize}
\item $\alpha_1\dots\alpha_n$ is weakly increasing. 
\item If $\alpha_k=\alpha_{k+1}$, both unbarred, then $\beta_k^\text{neg}\leq\beta_{k+1}^\text{neg}$ holds ($1\leq k\leq n-1$). 
\item If $\alpha_k=\alpha_{k+1}$, both barred, then $\beta_k^\text{neg}\geq\beta_{k+1}^\text{neg}$ holds ($1\leq k\leq n-1$). 
\end{itemize}
The set of doubly colored biwords of length $n$ is denoted by $ DC_n$. 
\end{dfn}

\begin{ex}
$\begin{psmallmatrix}
\ol{1}&\ol{1}&\ol{2}&\ol{2}&\ol{2}&2&2&2&2\\
3&1&2&\ol{1}&\ol{3}&\ol{1}&2&2&4\\
\end{psmallmatrix}
\in DC_9$ and 
$\begin{psmallmatrix}
\ol{1}&\ol{1}&\ol{1}&\ol{1}&1&1&1&\ol{2}&\ol{2}&\ol{2}\\
1&\ol{1}&\ol{2}&\ol{2}&\ol{4}&1&1&3&2&\ol{2}\\
\end{psmallmatrix}
\in DC_{10}$. 
\end{ex}

\begin{dfn}(the doubly mixed insertion \cite[\S 3.4]{sh-wh})\label{DM}
For $\pi=\begin{psmallmatrix}
\alpha_1& \dots & \alpha_n \\
\beta_1& \dots & \beta_n \\
\end{psmallmatrix}
\in DC_n$, we construct two sequences of unshifted tableaux: $P_0, P_1, \dots, P_n$, which leads to the insertion tableau $P_{M*}(\pi)$, and $Q_0, Q_1, \dots, Q_n$, which leads to the recording tableau $Q_{M*}(\pi)$. Let $P_0=\emptyset$ and $Q_0=\emptyset$. For each integer $k$ with $1\leq k\leq n$, $(P_k, Q_k)$ is defined inductively based on $(P_{k-1}, Q_{k-1})$ as follows: 
We define
\begin{align*}
P_k=\begin{cases}P_{k-1}\m\beta_k & \text{if $\alpha_k$ is unbarred, }\\
P_{k-1}\dm\beta_k& \text{if $\alpha_k$ is barred. }\end{cases}
\end{align*}
$Q_k$ is defined to be the unshifted tableau obtained by adding to $Q_{k-1}$ the cell $D(P_k)\setminus D(P_{k-1})$ filled with $\alpha_k$. 
We call $P_n$ the \textit{insertion tableau} of $\pi$ by the doubly mixed insertion and write it as $P_\text{M*}(\pi)$. Similarly, $Q_n$ is called the \textit{recording tableau} of $\pi$ by the doubly mixed insertion and written as $Q_\text{M*}(\pi)$. 
\end{dfn}
\begin{ex}\label{dmeg}
For
$\pi=\begin{psmallmatrix}
\ol{1}&\ol{1}&1&2&\ol{3}&3 \\
3&2&\ol{2}&1&\ol{1}&\ol{1} \\
\end{psmallmatrix}
\in DC_6$, the doubly mixed insertion proceeds as follows: 
\begin{align*}
&&(P_0, Q_0)&=(\emptyset, \emptyset) \\
\emptyset\dm3&=\begin{ytableau}3\end{ytableau}& (P_1, Q_1)&=\left(\,
\begin{ytableau}3\end{ytableau}\,, 
\begin{ytableau}\ol{1}\end{ytableau}\,
\right)\\
\begin{ytableau}\blue{3}\end{ytableau}\dm2&=\begin{ytableau}2&3\end{ytableau}& (P_2, Q_2)&=\left(\,
\begin{ytableau}2&3\end{ytableau}\,, 
\begin{ytableau}\ol{1}&\ol{1}\end{ytableau}\,
\right)\\
\begin{ytableau}\red{2}&3\end{ytableau}\m\ol{2}&=\begin{ytableau}\ol{2}&3\\2\end{ytableau}& (P_3, Q_3)&=\left(\,
\begin{ytableau}\ol{2}&3\\2\end{ytableau}\,, 
\begin{ytableau}\ol{1}&\ol{1}\\1\end{ytableau}\,
\right)\\
\begin{ytableau}\blue{\ol{2}}&\red{3}\\2\end{ytableau}\m1&=\begin{ytableau}1&\ol{2}\\2&3\end{ytableau}& (P_4, Q_4)&=\left(\,
\begin{ytableau}1&\ol{2}\\2&3\end{ytableau}\,, 
\begin{ytableau}\ol{1}&\ol{1}\\1&2\end{ytableau}\,
\right)\\
\begin{ytableau}\blue{1}&\red{\ol{2}}\\\blue{2}&\blue{3}\end{ytableau}\dm\ol{1}&=\begin{ytableau}\ol{1}&1&3\\\ol{2}&2\end{ytableau}& (P_5, Q_5)&=\left(\,
\begin{ytableau}\ol{1}&1&3\\\ol{2}&2\end{ytableau}\,, 
\begin{ytableau}\ol{1}&\ol{1}&\ol{3}\\1&2\end{ytableau}\,
\right)\\
\begin{ytableau}\blue{\ol{1}}&\red{1}&3\\\blue{\ol{2}}&\red{2}\end{ytableau}\m\ol{1}&=\begin{ytableau}\ol{1}&\ol{1}&3\\1&\ol{2}\\2\end{ytableau}& (P_6, Q_6)&=\left(\,
\begin{ytableau}\ol{1}&\ol{1}&3\\1&\ol{2}\\2\end{ytableau}\,, 
\begin{ytableau}\ol{1}&\ol{1}&\ol{3}\\1&2\\3\end{ytableau}\,
\right)\\
\end{align*}
Thus
\begin{align}
(P_\text{M*}(\pi), Q_\text{M*}(\pi))=\left(\,
\begin{ytableau}\ol{1}&\ol{1}&3\\1&\ol{2}\\2\end{ytableau}\,, 
\begin{ytableau}\ol{1}&\ol{1}&\ol{3}\\1&2\\3\end{ytableau}\,
\right). \label{dmeq1}
\end{align}
Here, after being bumped, the letters that are written in red enter the next row, and the letters that are written in blue enter the next column. 

\end{ex}
For $\lambda\in\mathcal{P}(n)$, let $\text{UNSHTab}(\lambda)$ be the set of unshifted tableaux of shape $\lambda$. 
The doubly mixed insertion gives a Robinson--Schensted-type correspondence as follows: 
\begin{thm}(\cite[Theorem~22]{sh-wh})\label{DMRS}
Let $n$ be a positive integer. The map which sends $\pi\in DC_n$ to a pair of same-shape unshifted tableaux with $n$ cells defined by
$$
\begin{array}{ccc} DC_n & \stackrel{}{\longrightarrow} & \displaystyle \coprod_{\lambda\in\mathcal{P}(n)}\text{UNSHTab}(\lambda)\times\text{UNSHTab}(\lambda) \\
\rotatebox{90}{$\in$} & & \rotatebox{90}{$\in$} \\
\pi & \longmapsto & \left(P_\text{M*}(\pi), Q_\text{M*}(\pi)\right)
\end{array}
$$
is bijective. 
We use the notation $\pi\xlr{\!\!\!\!\!\text{M*}\!\!\!\!\!}(T, U)$ to indicate that $\pi$ corresponds to a pair of tableaux $(T, U)$ in this correspondence. 
\end{thm}

Next, we recall the notion of the self-duality of the doubly mixed insertion by Shimozono and White. 
\begin{dfn}(\cite[\S 3.1]{sh-wh})
Let $\pi=\begin{psmallmatrix}
\alpha_1& \dots & \alpha_n \\
\beta_1& \dots & \beta_n \\
\end{psmallmatrix}$ be a doubly colored biword. 
The doubly colored biword obtained by rearranging
$\begin{psmallmatrix}
\beta_1 \\
\alpha_1 \\
\end{psmallmatrix}, \dots, \begin{psmallmatrix}
\beta_n \\
\alpha_n \\
\end{psmallmatrix}$ in the order that matches the definition of $DC_n$ is called the inverse of $\pi$, written $\pi^{-1}$. 
\end{dfn}

The self-duality of the doubly mixed insertion is stated as follows:  

\begin{thm}(\cite[Theorem~21]{sh-wh})\label{DMd}
Let $\pi$ be a doubly colored biword. If $\pi\xlr{\!\!\!\!\!\text{M*}\!\!\!\!\!}(T, U)$, then $\pi^{-1}\xlr{\!\!\!\!\!\text{M*}\!\!\!\!\!}(U, T)$ holds.  
\end{thm}
\begin{ex}
We have $\pi^{-1}=\begin{psmallmatrix}
\ol{1}&\ol{1}&1&\ol{2}&2&3 \\
3&\ol{3}&2&1&\ol{1}&\ol{1} \\
\end{psmallmatrix}$ for $\pi$ in Example~\ref{dmeg}. The doubly mixed insertion for $\pi^{-1}$ proceeds as follows: 
\begin{align*}
&&(P_0, Q_0)&=(\emptyset, \emptyset) \\
\emptyset\dm3&=\begin{ytableau}3\end{ytableau}& (P_1, Q_1)&=\left(\,
\begin{ytableau}3\end{ytableau}\,, 
\begin{ytableau}\ol{1}\end{ytableau}\,
\right)\\
\begin{ytableau}\blue{3}\end{ytableau}\dm\ol{3}&=\begin{ytableau}\ol{3}&3\end{ytableau}& (P_2, Q_2)&=\left(\,
\begin{ytableau}\ol{3}&3\end{ytableau}\,, 
\begin{ytableau}\ol{1}&\ol{1}\end{ytableau}\,
\right)\\
\begin{ytableau}\blue{\ol{3}}&\red{3}\end{ytableau}\m2&=\begin{ytableau}2&\ol{3}\\3\end{ytableau}& (P_3, Q_3)&=\left(\,
\begin{ytableau}2&\ol{3}\\3\end{ytableau}\,, 
\begin{ytableau}\ol{1}&\ol{1}\\1\end{ytableau}\,
\right)\\
\begin{ytableau}\blue{2}&\red{\ol{3}}\\\blue{3}\end{ytableau}\dm1&=\begin{ytableau}1&2\\\ol{3}&3\end{ytableau}& (P_4, Q_4)&=\left(\,
\begin{ytableau}1&2\\\ol{3}&3\end{ytableau}\,, 
\begin{ytableau}\ol{1}&\ol{1}\\1&\ol{2}\end{ytableau}\,
\right)\\
\begin{ytableau}\red{1}&2\\\blue{\ol{3}}&\red{3}\end{ytableau}\m\ol{1}&=\begin{ytableau}\ol{1}&2\\1&\ol{3}\\3\end{ytableau}& (P_5, Q_5)&=\left(\,
\begin{ytableau}\ol{1}&2\\1&\ol{3}\\3\end{ytableau}\,, 
\begin{ytableau}\ol{1}&\ol{1}\\1&\ol{2}\\2\end{ytableau}\,
\right)\\
\begin{ytableau}\blue{\ol{1}}&\red{2}\\1&\blue{\ol{3}}\\3\end{ytableau}\m\ol{1}&=\begin{ytableau}\ol{1}&\ol{1}&\ol{3}\\1&2\\3\end{ytableau}& (P_6, Q_6)&=\left(\,
\begin{ytableau}\ol{1}&\ol{1}&\ol{3}\\1&2\\3\end{ytableau}\,, 
\begin{ytableau}\ol{1}&\ol{1}&3\\1&\ol{2}\\2\end{ytableau}\,
\right)\\
\end{align*}
Thus
\begin{align}
\left(P_\text{M*}\left(\pi^{-1}\right), Q_\text{M*}\left(\pi^{-1}\right)\right)=\left(\,
\begin{ytableau}\ol{1}&\ol{1}&\ol{3}\\1&2\\3\end{ytableau}\,, 
\begin{ytableau}\ol{1}&\ol{1}&3\\1&\ol{2}\\2\end{ytableau}\,
\right). \label{dmeq2}
\end{align}
Here, the red or blue letters are used in the same sense as in Example~\ref{dmeg}. Comparing \eqref{dmeq2} with \eqref{dmeq1}, the order of the tableaux is swapped, as Theorem~\ref{DMd} asserts.
\end{ex}

\section{Shifted Tableaux, the Worley--Sagan Insertion, and Haiman's Mixed Insertion}
This section reviews the Worley--Sagan insertion, Haiman's mixed insertion, and the duality between them, which form the foundation of this paper. We define the Worley--Sagan insertion in Definition~\ref{WS} and Haiman's mixed insertion in Definition~\ref{HM}.
Our goal is to extend these insertion algorithms; however, since our extensions incorporate features from both, it is useful to examine them closely.

We begin by introducing shifted tableaux, the type of tableaux into which both insertions are defined. The arrangement of cells in a shifted tableau corresponds to a distinct partition.
\begin{dfn}[distinct partition]
Let $n$ be a non-negative integer.
A partition $\lambda=(\lambda_1,\dots,\lambda_\ell)$ of $n$ is called a \textit{distinct partition} of $n$ if its parts are distinct, that is, $\lambda_1>\cdots>\lambda_\ell$.
Note that $\emptyset$ is the unique distinct partition of $0$.
We write $\mathcal{P}_d(n)$ for the set of distinct partitions of $n$, and let $\mathcal{P}_d=\coprod_{n\geq 0}\mathcal{P}_d(n)$ be the set of all distinct partitions.
\end{dfn}
Shifted diagrams and shifted tableaux can be defined as follows.
\begin{dfn}[shifted tableaux]\label{shifted}
Let $\lambda$ be a distinct partition. 
The \textit{shifted diagram of shape $\lambda$} is 
\[D'(\lambda)=\{\,(i, j)\in\mathbb{P}^2 \mid 1\le i\le \ell(\lambda), i\leq j\leq \lambda_i+i-1\,\}.\] 
A \textit{shifted tableau of shape $\lambda$} is a map $P:D'(\lambda)\to \p\cup\ol{\p}$ which satisfies
the following four conditions: 
\begin{itemize}
\item If $(i, j), (i+1, j)\in D'(\lambda)$, then $P(i, j)\leq P(i+1, j)$.
\item If $(i, j), (i, j+1)\in D'(\lambda)$, then $P(i, j)\leq P(i, j+1)$.
\item For $j, a\in\p$, there exists at most one $i\in\p$ which satisfies $(i, j)\in D'(\lambda)$ and $P(i, j)=a$. 
\item For $i, a\in\p$, there exists at most one $j\in\p$ which satisfies $(i, j)\in D'(\lambda)$ and $P(i, j)=\ol{a}$.
\end{itemize}
We also use the symbol $D'(P)$ instead of $D'(\lambda)$ to stand for the set consisting of the cells of $P$. 
Shifted diagrams and shifted tableaux are illustrated in the same manner as unshifted diagrams and unshifted tableaux. For example, the shifted diagram of shape $\lambda=(6, 4, 2, 1)$ and an example of a shifted tableau of shape $\lambda$ are shown in the following figure: 
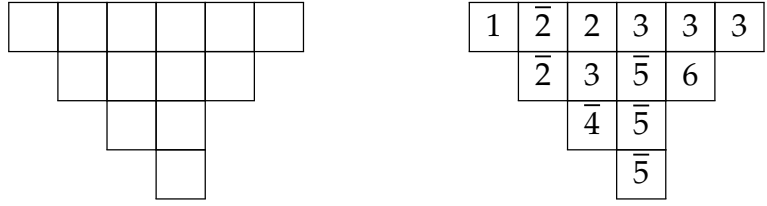
\[
\ytableausetup{centertableaux}
\hspace{10pt}
\ydiagram{6, 1+4, 2+2, 3+1}
\hspace{60pt}
\mathop{\begin{ytableau}
1&\ol{2}&2&3&3&3\\
\none &\ol{2}&3&\ol{5}&6\\
\none &\none &\ol{4} & \ol{5} \\
\none &\none &\none &\ol{5}
\end{ytableau}}
\]
We define shifted $\p$-tableaux to be shifted tableaux with no barred letters written in it. 
For shifted tableaux, the letters written in its diagonal cells are called \textit{diagonal} letters. 
As an agreement, the empty map $\emptyset\to\p\cup\ol{\p}$ is regarded as a unique shifted tableau of shape $\emptyset$, and is denoted by $\emptyset$. 
\end{dfn}

Let $P$ be a shifted tableau. We define a total order on $\p\cup\ol{\p}\cup\{\infty\}$ which is an extension of the total order on $\p\cup\ol{\p}$ and satisfies $\alpha<\infty$ for all $\alpha\in \p\cup\ol{\p}$. We sometimes identify $P$ with a map 
$\tilde{P}:\mathcal{O}\to \p\cup\ol{\p}\cup\{\infty\}$ such that 
\begin{align*}
\tilde{P}(c)=\begin{cases}P(c) & \text{if $c\in D(P)$, }\\
\infty & \text{if $c\in\mathcal{O}\setminus D(P)$. }
\end{cases}
\end{align*}

For later use, we define a (disjoint) union of maps as follows: 

\begin{dfn}
Let $\mathcal{A}_1, \dots, \mathcal{A}_n$ be a collection of pairwise disjoint subsets of $\mathcal{Q}$ and $P_i: \mathcal{A}_i\to\p\cup\ol{\p}$ be a map defined on $\mathcal{A}_i$. 
The union $P$ of $P_1, \dots, P_n$ is defined to be a map
\[ P: \mathcal{A}_1\cup\dots\cup\mathcal{A}_n\to\p\cup\ol{\p}
\]
which satisfies $P|_{\mathcal{A}_i}=P_i$ ($1\leq i\leq n$). We denote this map by $P_1\cup\dots\cup P_n$. 

For a map $P$ from a subset $\mathcal{A}$ of $\mathcal{Q}$ to $\p\cup\ol{\p}$, we define a map $|P|$ by $|P|(c)=|P(c)|$ for all $c\in\mathcal{A}$. 
\end{dfn}

\begin{rmk}
In this paper, we only treat shifted tableaux $P$ such that each positive integer appears at most once in $|P|$. 
\end{rmk}

Next, we recall the definition of the Worley--Sagan insertion. This is a process in which we produce a new shifted $\p$-tableau by adding a letter to a shifted $\p$-tableau. A distinctive feature of this insertion is that we change the way of bumping letters when a diagonal letter is bumped. The bumping of a diagonal letter occurs at most once in this insertion, and until that happens bumped letters are sent to the next rows. Upon the diagonal bumping and afterwards, the bumped letters are sent to the next columns.

\begin{dfn}[the Worley--Sagan insertion \cite{sa}, \cite{wo}]\label{WS}

Let $P$ be a shifted $\p$-tableau in which each element of $\p$ appears at most once, and let $a\in\p$ be a letter that does not appear in $P$.
The Worley--Sagan insertion is an operation that produces a new $\p$-tableau from $P$ and $a$, by applying the \textit{row bumping} and the \textit{column bumping} together with the \textit{rule change} from the former to the latter, following the procedure below. Here the row bumping moves a letter to the next row, while the column bumping moves a letter to the next column; in the terminology of Definition~\ref{OMDM}, every letter has the \uline{row nature} in the row bumping and the \uline{column nature} in the column bumping, all the letters being unbarred in the present situation.

The procedure starts by setting the rule to be the \uline{row} bumping. Then, just as if $a$ came from a row above the first row, we search the first \uline{row} of $P$ from left to right for an entry larger than $a$, and letting $a'$ be the first encountered such entry and $c'$ the cell containing it, we form a tableau $P'$ from $P$ by replacing the entry of the cell $c'$ by $a$; namely by letting $a$ bump out its old entry $a'$. If $a'=\infty$, the procedure terminates here with $P'$ the result of the insertion.

Otherwise we proceed, keeping the \uline{row} bumping in force, and we search the next \uline{row} of $P'$ below $c'$ from left to right for an entry larger than $a'$. Letting $a''$ be the first encountered such entry in this search and $c''$ the cell containing it, we form a tableau $P''$ from $P'$ by replacing the entry of the cell $c''$ by $a'$. If $a''=\infty$, the procedure terminates here with $P''$ the result of the insertion.

Otherwise we still keep the \uline{row} bumping and proceed to produce $P'''$, $P''''$, and so on, until either the procedure terminates as above, or the special situation develops where bumping occurs at a diagonal cell $e$. Since the \uline{row} bumping is in force, the letter bumped out of $e$ has the \uline{row nature}; this is an offending situation, since such a letter sent to anywhere in the $\mathcal{O}$-part of the next row would violate the tableau condition: either the resulting shape would not be a shifted diagram, or the transported letter would be less than the entry directly above its new position.
Now if this occurs, we call for the \textit{rule change} here and switch from the \uline{row} bumping to the \uline{column} bumping, so that the letter bumped out of $e$ will have the \uline{column nature} in the new rule.

Now we proceed to modify the tableau under the new rule: we search the next \uline{column} to the right of $e$ from top to bottom for an entry larger than the letter bumped out of $e$, and go on in the same manner, sending each letter bumped out of a cell to the next \uline{column} to the right of that cell, until the procedure terminates as above. Note that after switching to \uline{column} bumping, a bumping at a diagonal cell will never occur again, since in the search for a place for a letter bumped out of the previous column, the right-hand neighbor of its old position contains a larger letter (or is empty); so the new position for the bumped letter will be either this position or some higher position, which is never diagonal.

We denote the result of this insertion process by $P\ws a$.

For a permutation $\pi=a_1\dots a_n\in S_n$, we construct two sequences of shifted tableaux: 
$P_0$, $P_1$, $\dots$, $P_n$ which leads to the insertion tableau $P_\text{WS}(\pi)$, and $Q_0$, $Q_1$, $\dots$, $Q_n$ which leads to the recording tableau $Q_\text{WS}(\pi)$. 
Let $P_0=\emptyset$ and $Q_0=\emptyset$. For each integer $k$ with $1\leq k\leq n$, a tableaux pair $(P_k, Q_k)$ is constructed inductively based on $(P_{k-1}, Q_{k-1})$ through the following process: 
Let $P_k=P_{k-1}\xleftarrow[]{\!\!\!\!\!\text{WS}\!\!\!\!\!} a_k$, and let $c_k$ be the cell added to $P_{k-1}$ by inserting $a_k$. That is, $c_k=D'(P_k)\setminus D'(P_{k-1})$. 
Define an element $\tilde{k}$ of $\p\cup\ol{\p}$ by
\begin{align*}
\tilde{k}=\begin{cases} k & \text{if the insertion of $a_k$ into $P_{k-1}$ ends up with a row bumping, }\\
\ol{k}&
\text{if the insertion of $a_k$ into $P_{k-1}$ ends up with a column bumping. } \\
\end{cases}
\end{align*}
Let $Q_k$ be the tableau obtained from $Q_{k-1}$ by adding the cell $c_k$ filled with $\tilde{k}$. 
We call $P_n$ the \textit{insertion tableau} of $\pi$ by the Worley--Sagan insertion and write it as $P_\text{WS}(\pi)$. Similarly, $Q_n$ is called the \textit{recording tableau} of $\pi$ by the Worley--Sagan insertion and written as $Q_\text{WS}(\pi)$. 
\end{dfn}
\begin{ex}\label{wseg}
For $\pi=25413\in S_5$, the Worley--Sagan insertion proceeds as follows: 
\begin{align*}
&&(P_0, Q_0)&=(\emptyset, \emptyset) \\
\emptyset\ws2&=\begin{ytableau}2\end{ytableau}& (P_1, Q_1)&=\left(\,
\begin{ytableau}2\end{ytableau}\,, 
\begin{ytableau}1\end{ytableau}\,
\right)\\
\begin{ytableau}2\end{ytableau}\ws5&=\begin{ytableau}2&5\end{ytableau}& (P_2, Q_2)&=\left(\,
\begin{ytableau}2&5\end{ytableau}\,, 
\begin{ytableau}1&2\end{ytableau}\,
\right)\\
\begin{ytableau}2&\red{5}\end{ytableau}\ws4&=\begin{ytableau}2&4\\\none&5\end{ytableau}& (P_3, Q_3)&=\left(\,
\begin{ytableau}2&4\\\none&5\end{ytableau}\,, 
\begin{ytableau}1&2\\\none&3\end{ytableau}\,
\right)\\
\begin{ytableau}\blue{2}&\blue{4}\\\none&5\end{ytableau}\ws1&=\begin{ytableau}1&2&4 \\ \none&5\end{ytableau}& (P_4, Q_4)&=\left(\,
\begin{ytableau}1&2&4 \\ \none&5\end{ytableau}\,, 
\begin{ytableau}1&2&\ol{4} \\ \none&3\end{ytableau}\,
\right)\\
\begin{ytableau}1&2&\red{4} \\ \none&\blue{5}\end{ytableau}\ws3&=\begin{ytableau}1&2&3 \\ \none&4&5\end{ytableau}& (P_5, Q_5)&=\left(\,
\begin{ytableau}1&2&3 \\ \none&4&5\end{ytableau}\,, 
\begin{ytableau}1&2&\ol{4} \\ \none&3&\ol{5}\end{ytableau}\,
\right)\\
\end{align*}

Thus 
\begin{align}
(P_\text{WS}(\pi), Q_\text{WS}(\pi))=\left(\,
\begin{ytableau}1&2&3 \\ \none&4&5\end{ytableau}\,, 
\begin{ytableau}1&2&\ol{4} \\ \none&3&\ol{5}\end{ytableau}\,
\right). \label{wseq}
\end{align}
Here, the red or blue letters are used in the same sense as in Example~\ref{dmeg}. Red or blue letters are bumped in the order from the smallest to the largest among them.
\end{ex}

A shifted $\p$-tableau with $n$ cells is called \textit{standard} if every integer $k$ such that $1\leq k\leq n$ appears exactly once in it. 
For a distinct partition $\lambda$ of $n$, we denote the set of standard shifted $\p$-tableaux of shape $\lambda$
by $\text{SSHTab}(\lambda)$. 
The set of tableaux obtained from standard shifted $\p$-tableaux of shape $\lambda$ by adding bars to some of its \uline{off-diagonal} letters is denoted $\text{SSHTab}'(\lambda)$. Note that adding no bars is allowed here, so $\text{SSHTab}(\lambda)\subset\text{SSHTab}'(\lambda)$. 
The Worley--Sagan insertion gives a Robinson--Schensted-type correspondence as follows. 
\begin{thm}[\cite{sa}, \cite{wo}]\label{WSRS}
Let $n$ be a positive integer. The map which sends $\pi\in S_n$ to a pair of same-shape shifted tableaux with $n$ cells
$$
\begin{array}{ccc} S_n & \stackrel{}{\longrightarrow} & \displaystyle \coprod_{\lambda\in\mathcal{P}_d(n)}\text{SSHTab}(\lambda)\times\text{SSHTab}'(\lambda) \\
\rotatebox{90}{$\in$} & & \rotatebox{90}{$\in$} \\
\pi & \longmapsto & \left(P_\text{WS}(\pi),\  Q_\text{WS}(\pi)\right)
\end{array}
$$
is bijective. We use the notation $\pi\xlr{\!\!\!\!\!\text{WS}\!\!\!\!\!}(P, Q)$ to indicate that $\pi$ corresponds to a pair of tableaux $(P, Q)$ in this correspondence. 
\end{thm}
Next, we recall Haiman's mixed insertion. 
This inserts an unbarred letter into a shifted tableau in which some off-diagonal letters may be barred. Unbarred letters are given the row nature and barred letters the column nature. In this sense the rule is the \uline{ordinary} mixed buming as defined in Definition~\ref{OMDM}. However, a diagonal letter, which is unbarred, may be bumped. In this case one adds a bar to that letter, so that it assumes the column nature from that point, without changing the rule. We call this process the \textit{color change}. The associated Robinson--Schensted-type correspondence starts with an empty tableau and only unbarred letters are inserted, but due to this color change process the resulting insertion tableau can contain some barred off-diagonal letters.

\begin{dfn}[Haiman's mixed insertion \cite{h}]\label{HM}

Let $P$ be a shifted tableau in which each element of $\p$ appears at most once in $|P|$ and in which every diagonal entry is unbarred, and let $a\in\p$ be an unbarred letter such that $a$ does not appear in $|P|$.
Following the procedure below, Haiman's mixed insertion produces a new tableau from $P$ and $a$ by applying \textit{ordinary mixed bumping} as described in Definition~\ref{OMDM} together with color changes of letters.

The procedure starts just as if $a$ came from a row above the first row: we search the first \uline{row} of $P$ from left to right for an entry larger than $a$, and letting $\alpha'$ be the first encountered such entry and $c'$ the cell containing it, we form a tableau $P'$ from $P$ by replacing its entry of the cell $c'$ by $a$; namely by letting $a$ bump out its old entry $\alpha'$. If $\alpha'=\infty$, the procedure terminates here with $P'$ the result of the insertion.

Otherwise we proceed, and if $\alpha'$ has the \uline{row nature}, that is, if $\alpha'$ is \uline{unbarred}, we search the next \uline{row} of $P'$ below $c'$ from left to right for an entry larger than $\alpha'$; while if $\alpha'$ has the \uline{column nature}, that is, if $\alpha'$ is \uline{barred}, we search the next \uline{column} of $P'$ to the right of $c'$ from top to bottom for such an entry. Letting $\alpha''$ be the first encountered such entry in this search and $c''$ the cell containing it, we form a tableau $P''$ from $P'$ by replacing its entry of the cell $c''$ by $\alpha'$. If $\alpha''=\infty$, the procedure terminates here with $P''$ the result of the insertion.

Otherwise we proceed to produce $P'''$, $P''''$, and so on, until either the procedure terminates as above, or the special situation develops where bumping occurs at a diagonal cell $e$, say a letter $\beta$ bumping out the entry $\gamma$ of $e$. Note that this cannot happen if $\beta$ has the \uline{column} nature for the reason explained in Definition~\ref{WS}, so $\beta$ must have the \uline{row nature}, that is, $\beta$ is unbarred, so that the new entry of the diagonal cell $e$ is \uline{unbarred}. Note that $\gamma$ is also \uline{unbarred} due to the assumption on $P$, and under the rule it has the \uline{row nature}, but sending it to the next row will cause a problem for the reason also explained in Definition~\ref{WS}. When this happens, a \textit{color change} of the letter $\gamma$ is called for. This means that it is changed to a barred letter $\ol{\gamma}$ having the \uline{column nature} and, without changing the rule, is sent to the next \uline{column} to the right of $e$, which is searched from top to bottom as before.

Now we proceed to modify the tableau following the ordinary mixed bumping. After modification of the tableau continues for a certain number of steps, the diagonal bumping may again occur; in such a situation we call for the color change once more at that point and proceed. The color change may occur in this manner as many times as necessary before the procedure terminates. Since every letter newly entering a diagonal cell is unbarred, the resulting tableau again has all of its diagonal cells filled with unbarred letters, and hence the insertion can be iterated.

We denote the result of this insertion process by $P\hm a$.

For a permutation $\pi=a_1\dots a_n\in S_n$, we construct two sequences of shifted tableaux: $P_0$, $P_1$, $\dots$, $P_n$ which leads to the insertion tableau $P_\text{HM}(\pi)$, and $Q_0$, $Q_1$, $\dots$, $Q_n$ which leads to the recording tableau $Q_\text{HM}(\pi)$. 
Let $P_0=\emptyset$ and $Q_0=\emptyset$. For each integer $k$ with $1\leq k\leq n$, a tableaux pair $(P_k, Q_k)$ is constructed inductively based on $(P_{k-1}, Q_{k-1})$ through the following process: 
Let $P_k=P_{k-1}\hm a_k$, and let $c_k$ be the cell added to $P_{k-1}$ by inserting $a_k$. That is, $c_k=D'(P_k)\setminus D'(P_{k-1})$. 
Let $Q_k$ be the tableau obtained by adding to $Q_{k-1}$ the cell $c_k$ containing the letter $k$.
We call $P_n$ the \textit{insertion tableau} of $\pi$ by Haiman's mixed insertion and write it as $P_\text{HM}(\pi)$, and $Q_n$ is called the \textit{recording tableau} of $\pi$ by Haiman's mixed insertion and written as $Q_\text{HM}(\pi)$. 
\end{dfn}
\begin{ex}\label{hmeg}
For $\sigma=41532\in S_5$, Haiman's mixed insertion proceeds as follows: 
\begin{align*}
&&(P_0, Q_0)&=(\emptyset, \emptyset) \\
\emptyset\hm4&=\begin{ytableau}4\end{ytableau}& (P_1, Q_1)&=\left(\,
\begin{ytableau}4\end{ytableau}\,, 
\begin{ytableau}1\end{ytableau}\,
\right)\\
\begin{ytableau}\blue{4}\end{ytableau}\hm1&=\begin{ytableau}1&\ol{4}\end{ytableau}& (P_2, Q_2)&=\left(\,
\begin{ytableau}1&\ol{4}\end{ytableau}\,, 
\begin{ytableau}1&2\end{ytableau}\,
\right)\\
\begin{ytableau}1&\ol{4}\end{ytableau}\hm5&=\begin{ytableau}1&\ol{4}&5\end{ytableau}& (P_3, Q_3)&=\left(\,
\begin{ytableau}1&\ol{4}&5\end{ytableau}\,, 
\begin{ytableau}1&2&3\end{ytableau}\,
\right)\\
\begin{ytableau}1&\blue{\ol{4}}&\red{5}\end{ytableau}\hm3&=\begin{ytableau}1&3&\ol{4} \\ \none&5\end{ytableau}& (P_4, Q_4)&=\left(\,
\begin{ytableau}1&3&\ol{4} \\ \none&5\end{ytableau}\,, 
\begin{ytableau}1&2&3 \\ \none&4\end{ytableau}\,
\right)\\
\begin{ytableau}1&\red{3}&\ol{4} \\ \none&\blue{5}\end{ytableau}\hm2&=\begin{ytableau}1&2&\ol{4} \\ \none&3&\ol{5}\end{ytableau}& (P_5, Q_5)&=\left(\,
\begin{ytableau}1&2&\ol{4} \\ \none&3&\ol{5}\end{ytableau}\,, 
\begin{ytableau}1&2&3 \\ \none&4&5\end{ytableau}\,
\right)\\
\end{align*}
Thus 
\begin{align}
(P_\text{HM}(\sigma), Q_\text{HM}(\sigma))=\left(\,
\begin{ytableau}1&2&\ol{4} \\ \none&3&\ol{5}\end{ytableau}\,, 
\begin{ytableau}1&2&3 \\ \none&4&5\end{ytableau}\,
\right). \label{hmeq}
\end{align}
Here, the red or blue letters are used in the same sense as in Example~\ref{dmeg}. They are bumped in the order from the smallest to the largest among them. The blue letters on the diagonal get barred after being bumped. 
\end{ex}

Haiman's mixed insertion also gives a Robinson--Schensted-type correspondence. We should note that the order of the set of tableaux is swapped in this correspondence compared to the correspondence in Theorem~\ref{WSRS}.
\begin{thm}[\cite{h}]\label{HMRS}
Let $n$ be a positive integer. The map which sends $\pi\in S_n$ to a pair of same-shape shifted tableaux with $n$ cells
$$
\begin{array}{ccc} S_n & \stackrel{}{\longrightarrow} & \displaystyle \coprod_{\lambda\in\mathcal{P}_d(n)}\text{SSHTab}'(\lambda)\times\text{SSHTab}(\lambda) \\
\rotatebox{90}{$\in$} & & \rotatebox{90}{$\in$} \\
\pi & \longmapsto & \left(P_\text{HM}(\pi),\ Q_\text{HM}(\pi)\right)
\end{array}
$$
is bijective. We use the notation $\pi\xlr{\!\!\!\!\!\text{HM}\!\!\!\!\!}(P, Q)$ to indicate that $\pi$ corresponds to a pair of tableaux $(P, Q)$ in this correspondence. 
\end{thm}
The Worley--Sagan insertion and Haiman's mixed insertion enjoy the following duality relation. 
In \S8, we state and prove that this duality still holds between our extended versions of these two insertions. 
\begin{thm}[\cite{h}]\label{WSHMduality}
Let $\pi\in S_n$. If $\pi\xlr{\!\!\!\!\!\text{WS}\!\!\!\!\!}(P, Q)$, then $\pi^{-1}\xlr{\!\!\!\!\!\text{HM}\!\!\!\!\!}(Q, P)$ holds. 
\end{thm}
\begin{ex}
$\pi$ that appears in Example~\ref{wseg} and $\sigma$ that appears in Example~\ref{hmeg} are inverses of each other. 
Comparing \eqref{wseq} and \eqref{hmeq}, it can be seen that the order of tableaux is swapped, as Theorem~\ref{WSHMduality} asserts. 
\end{ex}

\section{The Worley--Sagan Mixed Insertion and the Generalized Mixed Insertion}
  
In this section, we will define the main algorithms of this paper, the Worley--Sagan mixed insertion and the generalized mixed insertion, and we will prove that each of them gives a Robinson--Schensted-type correspondence.

With the extension of the insertion algorithms, the set of words appearing in the Robinson--Schensted-type correspondences must also be extended; these extended correspondences will be maps defined on the set of colored permutations. We begin with the definition of them. 

\begin{dfn}
A \textit{colored permutation} is a $\left(\p\cup\ol{\p}\right)$-word obtained from the $\p$-word representing a permutation by adding bars to some of its letters. Adding no bars is also allowed. 
Let $B_n$ be the set of the colored permutations of length $n$. 
This notation is derived from the fact that these words can represent the elements of the Weyl group of type $B$, or the hyperoctahedral group.
\end{dfn}
For example, $\ol{5}3\ol{4}12, 24351, \ol{4}\ \ol{1}\ \ol{2}\ \ol{5}\ \ol{3}$ are elements of $B_5$. 
The cardinality of $B_n$ is $2^n\cdot n!$. 

Next, we define the Worley--Sagan mixed insertion. Starting from a shifted tableau $P$ and a letter $\alpha$, we produce another tableau by adding $\alpha$ to $P$. 
In this process, we use two different rules for bumping as in the doubly mixed insertion for unshifted tableaux, the \textit{ordinary mixed bumping} and the \textit{dual mixed bumping}. We also use the terms \textit{row nature} and \textit{column nature} for letters as defined in Definition~\ref{OMDM}. 
After choosing the initial rule suitably as explained in the definition below, we iterate these bumping procedures until the insertion terminates by filling an empty cell, but we switch between the two rules every time a diagonal bumping tries to bump out a letter of row nature.

\begin{dfn}[Worley--Sagan mixed insertion]\label{WSM}
Let $P$ be a shifted tableau with $|P|$ having no repeated entries, and $\alpha\in\p\cup\ol{\p}$ a letter with $|\alpha|\notin|P|$.
The Worley--Sagan mixed insertion is an operation that produces a new tableau from $P$ and $\alpha$, by applying both the \textit{ordinary mixed bumping} and the \textit{dual mixed bumping} together with the \textit{rule change} between the two, following the procedure below.

The procedure starts by setting the rule to be the \uline{ordinary} or \uline{dual} mixed bumping, according to whether $\alpha$ is unbarred or barred, so that $\alpha$ has the \uline{row nature} in the chosen rule. Then, just as if $\alpha$ came from a row above the first row, we search the first \uline{row} of $P$ from left to right for an entry larger than $\alpha$, and letting $\alpha'$ be the first encountered such entry and $c'$ the cell containing it, we form a tableau $P'$ from $P$ by replacing the entry of the cell $c'$ by $\alpha$; namely by letting $\alpha$ bump out its old entry $\alpha'$. If $\alpha'=\infty$, the procedure terminates here with $P'$ the result of the insertion.

Otherwise we proceed, keeping the current rule, ordinary or dual, and if $\alpha'$ has the \uline{row nature} in the current rule, we search the next \uline{row} of $P'$ below $c'$ from left to right; while if $\alpha'$ has the \uline{column nature} in the current rule, we search the next \uline{column} of $P'$ to the right of $c'$ from top to bottom. Letting $\alpha''$ be the first encountered such entry in this search and $c''$ the cell containing it, we form a tableau $P''$ from $P'$ by replacing the entry of the cell $c''$ by $\alpha'$. If $\alpha''=\infty$, the procedure terminates here with $P''$ the result of the insertion.

Otherwise we still keep the current rule and proceed to produce $P'''$, $P''''$, and so on, until either the procedure terminates as above, or the bumping occurs at a diagonal cell and the letter being bumped out has the \uline{row nature}. Placing it anywhere in the $\mathcal{O}$-part of the next row would create a problem as explained in Definition~\ref{WS}. When this occurs, we call for the \textit{rule change} and switch to the other rule: if the current rule is \uline{ordinary} we switch to the \uline{dual}, while if the current rule is \uline{dual} we switch to the \uline{ordinary}, so that the letter about to be bumped will have the \uline{column nature} in the new rule.

Now we proceed to modify the tableau under the new rule. After modification of the tableau under the new rule continues for a certain number of steps, a diagonal letter of row nature may be bumped again; in such a situation we switch the rule once more at that point and proceed. The rule change between the two rules may occur in this manner as many times as necessary before the procedure terminates.

We denote the result of this insertion process by $P\wsm \alpha$.

Given a colored permutation $\pi=\alpha_1\dots \alpha_n$, as in Definition~\ref{WS}, we construct two sequences of shifted tableaux, $P_0, P_1, \dots, P_n$ and $Q_0, Q_1, \dots,Q_n$ which lead to $P_\text{WSM}(\pi)$ and $Q_\text{WSM}(\pi)$, respectively. 
Let $P_0=\emptyset$ and $Q_0=\emptyset$. For each integer $k$ with $1\leq k\leq n$, a tableaux pair $(P_k, Q_k)$ is constructed inductively based on $(P_{k-1}, Q_{k-1})$ through the following process: 
Let $P_k=P_{k-1}\wsm \alpha_k$, and let $c_k$ be the cell added to $P_{k-1}$ by inserting $\alpha_k$. That is, $c_k=D'(P_k)\setminus D'(P_{k-1})$. 
Define $\tilde{k}\in\p\cup\ol{\p}$ to be
\begin{align*}
\tilde{k}=\begin{cases} k & \text{if the insertion of $\alpha_k$ into $P_{k-1}$ ends up with an ordinary mixed bumping, }\\
\ol{k}&
\text{if the insertion of $\alpha_k$ into $P_{k-1}$ ends up with a dual mixed bumping. } \\
\end{cases}
\end{align*}
Let $Q_k$ be the tableau obtained from $Q_{k-1}$ by adding the cell $c_k$ with entry 
\begin{align*}
Q_k(c_k)=\begin{cases} \tilde{k} & \text{if $c_k$ is off-diagonal, }\\
k&
\text{if $c_k$ is diagonal. } \\
\end{cases}
\end{align*}
We call $P_n$ the \textit{insertion tableau} of $\pi$ by the Worley--Sagan mixed insertion and write it as $P_\text{WSM}(\pi)$. Similarly, $Q_n$ is called the \textit{recording tableau} of $\pi$ by the Worley--Sagan mixed insertion and written as $Q_\text{WSM}(\pi)$. 
\end{dfn}
\begin{ex}\label{wsmeg}
For $\pi=\ol{7}5\ol{2}43\ol{6}\,\ol{1}\in B_n$, the Worley--Sagan mixed insertion proceeds as follows: 
\begin{align*}
&&(P_0, Q_0)&=(\emptyset, \emptyset) \\
\emptyset\wsm\,\ol{7}&=\begin{ytableau}\ol{7}\end{ytableau}& (P_1, Q_1)&=\left(\,
\begin{ytableau}\ol{7}\end{ytableau}\,, 
\begin{ytableau}1\end{ytableau}\,
\right)\\
\begin{ytableau}\blue{\ol{7}}\end{ytableau}\wsm\,5&=\begin{ytableau}5&\ol{7}\end{ytableau}& (P_2, Q_2)&=\left(\,
\begin{ytableau}5&\ol{7}\end{ytableau}\,, 
\begin{ytableau}1&2\end{ytableau}\,
\right)\\
\begin{ytableau}\blue{5}&\red{\ol{7}}\end{ytableau}\wsm\,\ol{2}&=\begin{ytableau}\ol{2}&5\\\none&\ol{7}\end{ytableau}& (P_3, Q_3)&=\left(\,
\begin{ytableau}\ol{2}&5\\\none&\ol{7}\end{ytableau}\,, 
\begin{ytableau}1&2\\\none&3\end{ytableau}\,
\right)\\
\begin{ytableau}\ol{2}&\red{5}\\\none&\blue{\ol{7}}\end{ytableau}\wsm\,4&=\begin{ytableau}\ol{2}&4&\ol{7}\\\none&5\end{ytableau}& (P_4, Q_4)&=\left(\,
\begin{ytableau}\ol{2}&4&\ol{7}\\\none&5\end{ytableau}\,, 
\begin{ytableau}1&2&4\\\none&3\end{ytableau}\,
\right)\\
\begin{ytableau}\ol{2}&\red{4}&\red{\ol{7}}\\\none&\blue{5}\end{ytableau}\wsm\,3&=\begin{ytableau}\ol{2}&3&5\\\none&4&\ol{7}\end{ytableau}& (P_5, Q_5)&=\left(\,
\begin{ytableau}\ol{2}&3&5\\\none&4&\ol{7}\end{ytableau}\,, 
\begin{ytableau}1&2&4\\\none&3&\ol{5}\end{ytableau}\,
\right)\\
\begin{ytableau}\ol{2}&3&5\\\none&4&\ol{7}\end{ytableau}\wsm\,\ol{6}&=\begin{ytableau}\ol{2}&3&5&\ol{6}\\\none&4&\ol{7}\end{ytableau}& (P_6, Q_6)&=\left(\,
\begin{ytableau}\ol{2}&3&5&\ol{6}\\\none&4&\ol{7}\end{ytableau}\,, 
\begin{ytableau}1&2&4&\ol{6}\\\none&3&\ol{5}\end{ytableau}\,
\right)\\
\begin{ytableau}\blue{\ol{2}}&\red{3}&\blue{5}&\red{\ol{6}}\\\none&\blue{4}&\red{\ol{7}}\end{ytableau}\wsm\,\ol{1}&=\begin{ytableau}\ol{1}&\ol{2}&4&5\\\none&3&\ol{6}\\\none&\none&\ol{7}\end{ytableau}& (P_7, Q_7)&=\left(\,
\begin{ytableau}\ol{1}&\ol{2}&4&5\\\none&3&\ol{6}\\\none&\none&\ol{7}\end{ytableau}\,, 
\begin{ytableau}1&2&4&\ol{6}\\\none&3&\ol{5}\\\none&\none&7\end{ytableau}\,
\right)\\
\end{align*}

Thus 
\begin{align}
(P_\text{WSM}(\pi), Q_\text{WSM}(\pi))=\left(\,
\begin{ytableau}\ol{1}&\ol{2}&4&5\\\none&3&\ol{6}\\\none&\none&\ol{7}\end{ytableau}\,, 
\begin{ytableau}1&2&4&\ol{6}\\\none&3&\ol{5}\\\none&\none&7\end{ytableau}\,
\right). \label{wsmeq}
\end{align}
Here, the red or blue letters are used in the same sense as in Examples~\ref{dmeg} and \ref{wseg}, and they are bumped in the order from the smallest to the largest among them. The largest letter among the red or blue letters indicates whether each insertion terminates with an ordinary mixed bumping or a dual mixed bumping. 
If it is unbarred red or barred blue, the insertion terminates with an ordinary mixed bumping. 
On the other hand, if it is unbarred blue or barred red, the insertion terminates with a dual mixed bumping. This is the information used in the construction of $Q_\text{WSM}(\pi)$.
\end{ex}

The Worley--Sagan mixed insertion gives a Robinson--Schensted-type correspondence between the set of colored permutations and the set of certain pairs of shifted tableaux. 
For a positive integer $n$ and $\lambda\in\mathcal{P}_d(n)$, the set of shifted tableaux obtained from standard shifted $\p$-tableaux of shape $\lambda$ by adding bars to some of its letters is denoted by $\text{SSHTab}''(\lambda)$. Adding no bars is allowed, so $\text{SSHTab}(\lambda)\subset\text{SSHTab}'(\lambda)\subset\text{SSHTab}''(\lambda)$ holds. 

\begin{thm}\label{WSMRS}
Let $n$ be a positive integer. The map which sends $\pi\in B_n$ to a pair of same-shape shifted tableaux with $n$ cells 
$$
\begin{array}{ccc} B_n & \stackrel{}{\longrightarrow} & \displaystyle \coprod_{\lambda\in\mathcal{P}_d(n)}\text{SSHTab}''(\lambda)\times\text{SSHTab}'(\lambda) \\
\rotatebox{90}{$\in$} & & \rotatebox{90}{$\in$} \\
\pi & \longmapsto & \left(P_\text{WSM}(\pi),\  Q_\text{WSM}(\pi)\right)
\end{array}
$$
is bijective. We use the notation $\pi\xlr{\!\!\!\!\!\text{WSM}\!\!\!\!\!}(P, Q)$ to indicate that $\pi$ corresponds to a pair of tableaux $(P, Q)$ in this correspondence. 
\end{thm}
\begin{proof}
We denote the map described above by $\phi$. 
We prove the bijectivity of $\phi$ by constructing an inverse map $\psi$ of $\phi$. 
Let us explain the definition of $\psi$.

We construct $\psi(P, Q)$ for arbitrary $(P, Q)\in\text{SSHTab}''(\lambda)\times\text{SSHTab}'(\lambda)$, where $\lambda$ is an arbitrary distinct partition of $n$. Let $c_k$ be the cell in $\mathcal{O}$ such that $|Q(c_k)|=k$ ($1\leq k\leq n$).

Set $P_n = P$. Apply the following procedure by varying $k=n, n-1, \dots, 1$ in this order to obtain $P_{n-1}, P_{n-2}, \dots, P_0 (=\emptyset)$ along with letters $\alpha_n, \alpha_{n-1}, \dots, \alpha_1$, and set $\psi(P, Q)=\alpha_1\alpha_2\dots\alpha_n$.

To obtain $P_{k-1}$ and $\alpha_k$ from $P_k$, we perform the following procedure in which we may switch between the two rules which we call the ordinary (ordinary reverse mixed bumping) and the dual (dual reverse mixed bumping). Similarly to the two rules used in the definition of the map $\phi$, in the ordinary reverse mixed bumping the unbarred [resp. barred] letters will have the row [resp. column] nature, whereas in the dual reverse mixed bumping the unbarred [resp. barred] letters will have the column [resp. row] nature.

For simplicity, set $\dot{P}=P_k, \dot{c}=c_k$, and $\dot{\alpha}=P_k(c_k)=\dot{P}(\dot{c})$. Let $\dot{P}'$ be the tableau obtained from $\dot{P}$ by removing its cell $\dot{c}$.

We choose the ordinary [resp. dual] as the first rule if $Q(c_k) = k$ [resp. $Q(c_k) = \ol{k}$]. There are certain exceptional situations to be explained soon, but we look at the typical situtation first. If $\dot{\alpha}$ has the row nature, search the previous row of $\dot{P}'$ above $\dot{c}$ from right to left for an entry smaller than $\dot{\alpha}$; if $\dot{\alpha}$ has the column nature, search the previous column immediately to the left of $\dot{c}$ from bottom to top for an entry smaller than $\dot{\alpha}$. Letting $\dot{\alpha}'$ be the first encountered such entry in this search and $\dot{c}'$ the cell containing it, modify the tableau $\dot{P}'$ into $\dot{P}''$ by replacing the entry of the cell $\dot{c}'$ by $\dot{\alpha}$, regarding the cell's old entry $\dot{\alpha}'$ as ``bumped out'' by $\dot{\alpha}$. Repeat the same procedure with $\dot{P}''$ and $\dot{\alpha}'$ instead of $\dot{P}'$ and $\dot{\alpha}$; namely follow the current rule and search for the cell $\dot{c}''$ of $\dot{P}''$ in the previous row or column, depending on the nature of the letter $\dot{\alpha}'$ in the current rule, to find where the letter $\dot{\alpha}'$ should go and bump out the cell's old entry $\dot{\alpha}''$ to form another tableau $\dot{P}'''$, and so on.

One of the exceptional situations involves the termination of the procedure: if a letter is bumped out of a cell in the first row and the bumped letter has the row nature, then the procedure terminates here, saving the current tableau under the name $P_{k-1}$ and define $\alpha_k$ to be the letter bumped out from the first row.

The other is concerned with rule switching: If a letter is bumped out of a diagonal cell and has the column nature, then sending it anywhere in the $\mathcal{O}$-part of the previous column would violate the tableau condition concerning the comparison of the transported letter with the right-hand neighbor of its new position. Furthermore, if the diagonal cell is $(1, 1)$, this is a situation we want to avoid since reverse bumping from the first column is not an inverse of any operation defined as part of the procedure $\phi$. Whenever this happens, we switch to the other rule so that the bumped letter has the row nature, search the previous row for the place where it should go, and proceed.

We repeat this until the procedure terminates in the sense explained above.

It is straightforward to check that $\phi\circ\psi$ and $\psi\circ\phi$ are the identity maps.
\end{proof}

\begin{rmk}
\begin{enumerate}[label=\arabic*., labelsep=0.5em]
  \item The name Worley--Sagan mixed insertion is used since it inherits its idea of rule change from the Worley--Sagan insertion, and its idea of changing the direction of progression depending on colors of letters from Haiman's mixed insertion.
  \item Sagan extended his Worley--Sagan insertion and gave what he calls the shifted Knuth correspondence {\cite[Theorem~8.1]{sa}}. By a suitable restriction, it also gives rise to a bijection
\[B_n\xrightarrow{\sim}\coprod_{\lambda\in\mathcal{P}_d(n)}\text{SSHTab}''(\lambda)\times\text{SSHTab}'(\lambda)\]
(up to some notational coordination). This map applied to the element $\pi\in B_n$ in Example~\ref{wsmeg}, say, gives the pair
\[\left(\,
\begin{ytableau}\ol{1}&\ol{2}&3&5&\ol{6}\\\none&4&\ol{7}\end{ytableau}\,, 
\begin{ytableau}1&\ol{2}&\ol{3}&\ol{5}&\ol{7}\\\none&4&6\end{ytableau}\,\right).\]
Our pair $\left(P_\text{WSM}(\pi),\  Q_\text{WSM}(\pi)\right)$ as shown in $\eqref{wsmeq}$ is quite different from this, even in the shape of the tableaux.
\end{enumerate}
\end{rmk}

Next, we define an extension of Haiman's mixed insertion, which we call the generalized mixed insertion. Starting with a shifted tableau $P$ and a letter $\alpha$, we obtain a new tableau by inserting $\alpha$ into $P$. 

Similarly to the Worley--Sagan mixed insertion, the generalized mixed insertion consists of a succession of ordinary mixed bumpings or dual mixed bumpings, which is defined in Definition~\ref{WSM}. 
Unlike the Worley--Sagan mixed insertion, this insertion sets the rule to either ordinary or dual mixed bumping so that $\alpha$ is given the row nature, and sticks to the same rule all the way until an empty cell is filled. However, when bumping occurs at a diagonal cell, certain issues may arise. In such cases, the \textit{color change} process is invoked, as explained in Definition~\ref{GM}, while still using the same rule. This process ensures that whenever a diagonal letter is bumped, it always moves into the column immediately to the right of its original position, and that the diagonal entries are unbarred.

\begin{dfn}[generalized mixed insertion]\label{GM}
Let $P$ be a shifted tableau where all diagonal entries are unbarred and $|P|$ has no repeated entries, and $\alpha\in\p\cup\ol{\p}$ a letter with $|\alpha|\notin|P|$.

The generalized mixed insertion is an operation that produces a new tableau from $P$ and $\alpha$, by applying either the \textit{ordinary mixed bumping} or the \textit{dual mixed bumping} together with the \textit{color change} of letters, following the explanation below. As in Definition~\ref{OMDM}, the ordinary mixed bumping gives unbarred letters the \uline{row nature} and barred letters the \uline{column nature}, whereas the dual mixed bumping gives unbarred letters the \uline{column nature} and barred letters the \uline{row nature}.

The procedure starts by setting the rule to be the \uline{ordinary} or \uline{dual} mixed bumping, according to whether $\alpha$ is unbarred or barred, so that $\alpha$ has the \uline{row nature} in the chosen rule. Then, just as if $\alpha$ came from a row above the first row, we search the first \uline{row} of $P$ from left to right for an entry larger than $\alpha$, and letting $\alpha'$ be the first encountered such entry and $c'$ the cell containing it, we form a tableau $P'$ from $P$ by replacing its entry of the cell $c'$ by $\alpha$; namely by letting $\alpha$ bump out its old entry $\alpha'$. If $\alpha'=\infty$, the procedure terminates here with $P'$ the result of the insertion.

Otherwise we proceed, keeping the rule, and if $\alpha'$ has the \uline{row nature} in the chosen rule, we search the next \uline{row} of $P'$ below $c'$ from left to right; while if $\alpha'$ has the \uline{column nature} in the chosen rule, we search the next \uline{column} of $P'$ to the right of $c'$ from top to bottom. Letting $\alpha''$ be the first encountered such entry in this search and $c''$ the cell containing it, we form a tableau $P''$ from $P'$ by replacing its entry of the cell $c''$ by $\alpha'$. If $\alpha''=\infty$, the procedure terminates here with $P''$ the result of the insertion.

Otherwise we still keep the rule and proceed to produce $P'''$, $P''''$, and so on, until either the procedure terminates as above, or a bumping occurs at a diagonal cell $e$, say a letter $\beta$ bumping out the entry $\gamma$ of $e$.
Note that a search along a column never reaches a diagonal cell for a reason similar to that explained in Definition~\ref{WS}, so that $\beta$ must be a letter having the \uline{row nature}; note also that every diagonal entry of the current tableau is unbarred, so that $\gamma$ is \uline{unbarred}. If the chosen rule is \uline{ordinary}, this means that $\beta$ is \uline{unbarred} and so placing it in a diagonal cell is fine, and we let it bump $\gamma$. But $\gamma$ has the \uline{row nature} in this case, so sending it to anywhere in the $\mathcal{O}$-part of the next row would create an object violating the tableau condition in the sense explained in Definition~\ref{WS}. 
In this case we change the color of the bumped letter $\gamma$ (unbarred) into $\ol{\gamma}$ (barred) so that it has the \uline{column nature} in the chosen rule (which is \uline{ordinary}), and conduct a search for its place in the next column to the right. On the other hand, if the chosen rule is \uline{dual}, then $\beta$ must be barred to have the \uline{row nature}, so placing it in a diagonal cell would violate the unbarredness of diagonal entries. In this case we change the color of the letter $\beta$ (barred) to $\ol{\beta}$ (unbarred), and let it bump the entry $\gamma$ of $e$, which has the \uline{column nature} in the chosen rule (which is \uline{dual}), so we proceed to search for its place in the next column to the right. This latter color change is also applied when $\gamma=\infty$, in which case the procedure terminates by filling $e$ with $\ol{\beta}$ (unbarred).

Now we proceed to modify the tableau, still under the same rule. Every time the bumping at a diagonal cell occurs, we change either the color of the letter being bumped out of the diagonal cell or the color of the letter being brought into the diagonal cell, depending on whether the chosen rule is ordinary or dual, and proceed until the procedure terminates. All through the procedure the tableau condition and the unbarredness of the diagonal entries are preserved, and hence anothor generalized mixed insertion can be applied.

We denote the result of this insertion process by $P\gm \alpha$.

Given a colored permutation $\pi=\alpha_1\dots \alpha_n$, as in Definition~\ref{WS}, we construct two sequences of shifted tableaux, $P_0, P_1, \dots, P_n$ and $Q_0, Q_1, \dots,Q_n$, which lead to $P_\text{GM}(\pi)$ and $Q_\text{GM}(\pi)$, respectively. 
Let $P_0=\emptyset$ and $Q_0=\emptyset$. For each integer $k$ with $1\leq k\leq n$, a tableaux pair $(P_k, Q_k)$ is constructed inductively based on $(P_{k-1}, Q_{k-1})$ through the following process: 
Let $P_k=P_{k-1}\gm \alpha_k$, and let $c_k$ be the cell added to $P_{k-1}$ by inserting $\alpha_k$. That is, $c_k=D'(P_k)\setminus D'(P_{k-1})$. 
Let $Q_k$ be the tableau obtained by adding to $Q_{k-1}$ the cell $c_k$ whose entry is determined by the rule: 
\begin{align*}
Q_k(c_k)=\begin{cases} k & \text{if $\alpha_k$ is unbarred, }\\
\ol{k}&
\text{if $\alpha_k$ is barred. } \\
\end{cases}
\end{align*}
We call $P_n$ the \textit{insertion tableau} of $\pi$ by the generalized mixed insertion and write it as $P_\text{GM}(\pi)$. Similarly, $Q_n$ is called the \textit{recording tableau} of $\pi$ by the generalized mixed insertion and written as $Q_\text{GM}(\pi)$.

\end{dfn}
\begin{ex}\label{gmeg}
For $\sigma=\ol{7}\,\ol{3}542\ol{6}\,\ol{1}\in B_n$, the generalized mixed insertion proceeds as follows: 
\begin{align*}
&&(P_0, Q_0)&=(\emptyset, \emptyset) \\
\emptyset\gm\ol{7}&=\begin{ytableau}7\end{ytableau}& (P_1, Q_1)&=\left(\,
\begin{ytableau}7\end{ytableau}\,, 
\begin{ytableau}\ol{1}\end{ytableau}\,
\right)\\
\begin{ytableau}\blue{7}\end{ytableau}\gm\ol{3}&=\begin{ytableau}3&7\end{ytableau}& (P_2, Q_2)&=\left(\,
\begin{ytableau}3&7\end{ytableau}\,, 
\begin{ytableau}\ol{1}&\ol{2}\end{ytableau}\,
\right)\\
\begin{ytableau}3&\red{7}\end{ytableau}\gm5&=\begin{ytableau}3&5\\\none&7\end{ytableau}& (P_3, Q_3)&=\left(\,
\begin{ytableau}3&5\\\none&7\end{ytableau}\,, 
\begin{ytableau}\ol{1}&\ol{2}\\\none&3\end{ytableau}\,
\right)\\
\begin{ytableau}3&\red{5}\\\none&\blue{7}\end{ytableau}\gm4&=\begin{ytableau}3&4&\ol{7}\\\none&5\end{ytableau}& (P_4, Q_4)&=\left(\,
\begin{ytableau}3&4&\ol{7}\\\none&5\end{ytableau}\,, 
\begin{ytableau}\ol{1}&\ol{2}&4\\\none&3\end{ytableau}\,
\right)\\
\begin{ytableau}\blue{3}&\red{4}&\blue{\ol{7}}\\\none&\blue{5}\end{ytableau}\gm2&=\begin{ytableau}2&\ol{3}&\ol{5}&\ol{7}\\\none&4\end{ytableau}& (P_5, Q_5)&=\left(\,
\begin{ytableau}2&\ol{3}&\ol{5}&\ol{7}\\\none&4\end{ytableau}\,, 
\begin{ytableau}\ol{1}&\ol{2}&4&5\\\none&3\end{ytableau}\,
\right)\\
\begin{ytableau}2&\ol{3}&\ol{5}&\red{\ol{7}}\\\none&4\end{ytableau}\gm\ol{6}&=\begin{ytableau}2&\ol{3}&\ol{5}&\ol{6}\\\none&4&\ol{7}\end{ytableau}& (P_6, Q_6)&=\left(\,
\begin{ytableau}2&\ol{3}&\ol{5}&\ol{6}\\\none&4&\ol{7}\end{ytableau}\,, 
\begin{ytableau}\ol{1}&\ol{2}&4&5\\\none&3&\ol{6}\end{ytableau}\,
\right)\\
\begin{ytableau}\blue{2}&\red{\ol{3}}&\red{\ol{5}}&\ol{6}\\\none&\blue{4}&\red{\ol{7}}\end{ytableau}\gm\ol{1}&=\begin{ytableau}1&2&4&\ol{6}\\\none&3&\ol{5}\\\none&\none&7\end{ytableau}& (P_7, Q_7)&=\left(\,
\begin{ytableau}1&2&4&\ol{6}\\\none&3&\ol{5}\\\none&\none&7\end{ytableau}\,, 
\begin{ytableau}\ol{1}&\ol{2}&4&5\\\none&3&\ol{6}\\\none&\none&\ol{7}\end{ytableau}\,
\right)\\
\end{align*}

Thus 
\begin{align}
(P_\text{GM}(\sigma), Q_\text{GM}(\sigma))=\left(\,
\begin{ytableau}1&2&4&\ol{6}\\\none&3&\ol{5}\\\none&\none&7\end{ytableau}\,, 
\begin{ytableau}\ol{1}&\ol{2}&4&5\\\none&3&\ol{6}\\\none&\none&\ol{7}\end{ytableau}\,
\right). \label{gmeq}
\end{align}
Here, the red or blue letters are used in the same sense as in Examples~\ref{dmeg} and \ref{hmeg}, and they are bumped in the order from the smallest to the largest among them.
Note that in the case of ordinary mixed bumping, bars are attached to unbarred letters bumped out of the diagonal cells, and in the case of dual mixed bumping, bars are removed from barred letters entering the diagonal cells.
\end{ex}

The generalized mixed insertion also gives a Robinson--Schensted-type correspondence. We should note that the order of the set of tableaux is swapped in this correspondence compared to the correspondence in Theorem~\ref{WSMRS}.
\begin{thm}\label{GMRS}
Let $n$ be a positive integer. The map which sends $\pi\in B_n$ to a pair of same-shape shifted tableaux with $n$ cells
$$
\begin{array}{ccc} B_n & \stackrel{}{\longrightarrow} & \displaystyle \coprod_{\lambda\in\mathcal{P}_d(n)}\text{SSHTab}'(\lambda)\times\text{SSHTab}''(\lambda) \\
\rotatebox{90}{$\in$} & & \rotatebox{90}{$\in$} \\
\pi & \longmapsto & \left(P_\text{GM}(\pi),\  Q_\text{GM}(\pi)\right)
\end{array}
$$
is bijective. We use the notation $\pi\xlr{\!\!\!\!\!\text{GM}\!\!\!\!\!}(P, Q)$ to indicate that $\pi$ corresponds to a pair of tableaux $(P, Q)$ in this correspondence. 
\end{thm}
\begin{proof}

The outline of the proof is similar to that of Theorem~\ref{WSMRS}. 
We denote the map described above by $\phi$, and prove the bijectivity of $\phi$ by constructing an inverse map $\psi$ of $\phi$. 
Let us explain the definition of $\psi$. 

We construct $\psi(P, Q)$ for arbitrary $(P, Q)\in\text{SSHTab}'(\lambda)\times\text{SSHTab}''(\lambda)$, where $\lambda$ is an arbitrary distinct partition of $n$. Let $c_k$ be the cell in $\mathcal{O}$ such that $|Q(c_k)|=k$ ($1\leq k\leq n$). 

Set $P_n=P$. Apply the following procedure by varying $k=n, n-1, \dots, 1$ in this order to obtain $P_{n-1}, P_{n-2}, \dots, P_0(=\emptyset)$ along with letters $\alpha_n, \alpha_{n-1}, \dots, \alpha_1$, and set $\psi(P, Q)=\alpha_1\alpha_2\dots\alpha_n$. 

To obtain $P_{k-1}$ and $\alpha_k$ from $P_k$, we perform the following procedure in which we follow one of the two rules which we call the ordinary reverse mixed bumping and the dual reverse mixed bumping, as in the proof of Theorem~\ref{WSMRS}. Unlike in the proof of Theorem~\ref{WSMRS}, however, we never switch between the two rules in the course of the procedure obtaining $P_{k-1}$ and $\alpha_k$ from $P_k$; instead we change the colors of letters in certain exceptional situations, just as in the definition of the generalized mixed insertion. 

For simplicity, set $\dot{P}=P_k, \dot{c}=c_k$, and $\dot{\alpha}=P_k(c_k)=\dot{P}(\dot{c})$. Let $\dot{P}'$ be the tableau obtained from $\dot{P}$ by removing its cell $\dot{c}$. 

We choose the ordinary [resp. dual] reverse mixed bumping as the rule if $Q(c_k) = k$ [resp. $Q(c_k) = \ol{k}$], and keep it until the procedure terminates. If $\dot{\alpha}$ has the row nature, search the previous row of $\dot{P}'$ above $\dot{c}$ from right to left for an entry smaller than $\dot{\alpha}$; if $\dot{\alpha}$ has the column nature, search the previous column immediately to the left of $\dot{c}$ from bottom to top for an entry smaller than $\dot{\alpha}$. Letting $\dot{\alpha}'$ be the first encountered such entry in this search and $\dot{c}'$ the cell containing it, modify the tableau $\dot{P}'$ into $\dot{P}''$ by replacing the entry of the cell $\dot{c}'$ by $\dot{\alpha}$, regarding the cell's old entry $\dot{\alpha}'$ as ``bumped out'' by $\dot{\alpha}$. From $\dot{P}''$ and $\dot{\alpha}'$ we obtain $\dot{P}''', \dot{P}'''', \dots$ and $\dot{\alpha}'', \dot{\alpha}''', \dots$ in the same way as in the proof of Theorem~\ref{WSMRS}, searching at each step the previous row or the previous column according to the nature in the rule of the letter being relocated, except for the treatment of diagonal cells explained below. The procedure terminates when a letter having the row nature is bumped out of a cell in the first row; at this point we name the current tableau $P_{k-1}$ and define $\alpha_k$ to be the letter bumped out from the first row.

In this procedure, bumping at diagonal cells is handled by
color change as follows.
Suppose the search for the place for a bumped letter $\dot{\gamma}$ ends
at a diagonal cell $\dot{e}$; namely, its entry $\dot{\beta}$ is the first
encountered letter smaller than $\dot{\gamma}$ in this search.
Since $\dot{e}$ is a diagonal cell, $\dot{\beta}$ is unbarred.
Moreover, if this were the search of a row, it would end in a cell either
directly above the original place of $\dot{\gamma}$ or somewhere more to
the right, which could not be a diagonal cell.
Hence $\dot{\gamma}$ must have the column nature.
If the chosen rule is ordinary, $\dot{\gamma}$ must be barred, so we
change its color to an unbarred letter $\overline{\dot{\gamma}}$ and let
it bump the entry $\dot{\beta}$ of $\dot{e}$, which has the row nature.
So, without changing the color of $\dot{\beta}$, we proceed to search the
row immediately above $\dot{e}$ for its new place
(or if $\dot{e}=(1,1)$, the procedure terminates).
If the chosen rule is dual, $\dot{\gamma}$ must be unbarred in order to
have the column nature. So, without changing its color, we let it bump the
entry $\dot{\beta}$ of $\dot{e}$.
However, $\dot{\beta}$ is unbarred and hence has the column nature, and placing $\dot{\beta}$ anywhere in the
$\mathcal{O}$-part of the column immediately to the left of $\dot{e}$
would make it larger than the entry immediately to its right,
violating the tableau condition.
If $\dot{e}=(1,1)$, reverse bumping from the first column would not
be the inverse of any operation defined as part of the procedure $\phi$.
So we change the color of $\dot{\beta}$ (unbarred) to
$\overline{\dot{\beta}}$ (barred) to give it the row nature, and proceed to
search the row immediately above $\dot{e}$ for its new place
(or if $\dot{e}=(1,1)$, the procedure terminates).
In this manner the unbarredness of the diagonal letters and the tableau
conditions are preserved.
Hence this procedure can be applied repeatedly with $k$ varying from
$n$ to $1$.

It is straightforward to check that $\phi\circ\psi$ and $\psi\circ\phi$ are the identity maps.
\end{proof}

\begin{rmk}
This insertion inherits a feature of Haiman's mixed insertion rather than the Worley--Sagan insertion in that it invokes the color change process when diagonal bumping occurs; hence, it is called the generalized mixed insertion.
\end{rmk}

\section{Relationship between the Worley--Sagan Mixed Insertion and the Doubly Mixed Insertion}

In this section, we will prove that the Worley--Sagan mixed insertion can be embedded into the doubly mixed insertion by doubling tableaux and words. 
 
We define two types of doubling operators for shifted tableaux. The first one is for insertion tableaux, written $\bbackslash$, and the second one is for recording tableaux, written $\bp$. Roughly speaking, for a shifted tableaux $P$, $\bbackslash P$ is defined by attaching the transpose of $P$ to the down left corner of $P$. 
$\bp P$ is a tableau defined by modifying some colors of letters written in $\bbackslash P$. 

We redefine $\mathcal{Q}$ to be $\mathbb{P}\times(\mathbb{P}\cup\{0\})$ and adopt this definition henceforth. This is because, when we attach the transpose of a shifted tableau, we want to shift the transpose to the left by one column so that it has no overlap with the original.
We regard $\mathcal{O}=\{\,(i, j)\in\mathcal{Q} \mid 1\leq i\leq j\,\}$ as a subset of $\mathcal{Q}$. 
We say a cell is \textit{above diagonal} if it is an element of $\mathcal{O}$ and \textit{below diagonal} if it is an element of $\mathcal{Q}\setminus\mathcal{O}$. 
Figure~\ref{fig3} illustrates $\mathcal{Q}$ placed in the $xy$-plane with the new coordinates, and the above-diagonal cells are indicated by gray squares.

Define a map $\omega : \mathcal{Q}\to\mathcal{Q}$ by
\begin{align*}
\omega(c)=\begin{cases}(j, i-1)& \text{if $c\in\mathcal{O}$, }\\
(j+1, i)& \text{if $c\in\mathcal{Q}\setminus\mathcal{O}$. }
\end{cases}
\end{align*}
for $c=(i, j)\in\mathcal{Q}$. 
$\omega(c)$ is the left-neighbor cell of the transpose of $c$ across the diagonal line $\Delta=\{\,(i, i)\in\mathcal{Q} \mid i\in\mathbb{P}\,\}$. 
Note that this line does not begin from $(1, 0)$, the left top corner of $\mathcal{Q}$. Note that our $\omega$ is the restriction of the map $\sigma$ defined by Haiman \cite[Definition~5.2]{h} to $\mathcal{Q}$.
Figure~\ref{fig4} shows the location of $\omega(a)$ and $\omega(b)$ for two specific cells $a, b$. 
As can be easily seen in Figure~\ref{fig4}, $\mathcal{O}$ and $\mathcal{Q}\setminus\mathcal{O}$ maps to each other by $\omega$, and $\omega$ is an involution. 

\begin{dfn}
For a map $T : \mathcal{A}\to\p\cup\ol{\p}$ defined on a subset $\mathcal{A}$ of $\mathcal{Q}$, 
we denote maps 
\begin{align*}
\ol{(\cdot)}\circ T&: \mathcal{A}\to\p\cup\ol{\p}, \\
T\circ\omega&: \omega(\mathcal{A})\to\p\cup\ol{\p}, \\
\ol{(\cdot)}\circ T\circ\omega&: \omega(\mathcal{A})\to\p\cup\ol{\p}
\end{align*}
 by $\ol{T}, \omega(T)$ and $\ol{\omega}(T)$, respectively. 
Note that $\ol{\omega}(T)=\omega(\ol{T})=\ol{\omega(T)}$ holds. 
\end{dfn}

\begin{figure}[htbp]
\centering
\begin{minipage}{0.45\linewidth}
\begin{tikzpicture}[x=0.8cm, y=0.8cm, yscale=-1]\label{fig3}
\draw[step=0.8cm, help lines, color=gray!30, dashed] (-0.9,-0.9) grid (8.25,8.25);
\foreach \x in {1, ..., 7} \fill[black!30] (\x+0.5,\x-0.5) rectangle (7.75,\x+0.5);
\draw[->,ultra thick] (0,-1)--(0,8.5) node[below]{$x$};
\draw[->,ultra thick] (-1,0)--(8.5,0) node[right]{$y$};
\foreach \x in {1, ..., 8} \draw (0, \x) node[left]{$\x$};
\foreach \x in {0, ..., 7} \draw (\x+1,0) node[above]{$\x$};
\foreach \x in {1.5, ..., 7.5} \draw[line width=1.5pt] (0.5,\x)--(7.75,\x);
\foreach \x in {0.5, ..., 6.5} \draw[line width=1.5pt] (\x,7.75)--(\x, 0.5);
\draw[line width=1.5pt](0.5,0.5)--(7.75,0.5);
\draw[line width=1.5pt](7.5,0.5)--(7.5,7.75);
\draw (0, 0) node[above left] {O}; 
\foreach \x in {1, ..., 7} \draw (8.25, \x)node{$\dots$};
\foreach \x in {1, ..., 7} \draw (\x, 8)node{$\vdots$};
\draw(8,8)node{$\ddots$};
\end{tikzpicture}
\caption{\label{fig3}}
\end{minipage}
\hspace{30pt}
\begin{minipage}{0.45\linewidth}
\begin{tikzpicture}[x=0.8cm, y=0.8cm, yscale=-1]\label{fig4}
\draw[step=0.8cm, help lines, color=gray!30, dashed] (-0.9,-0.9) grid (8.25,8.25);
\foreach \x in {1, ..., 7} \fill[black!30] (\x+0.5,\x+0.5) rectangle (8,\x-0.5);
\foreach \x in {1, ..., 7} \fill[black!10] (\x-0.5,\x-0.5) rectangle (\x+0.5,8);
\fill[black!10] (7.5,7.5) rectangle (8,8);
\draw[->,ultra thick] (0,-1)--(0,8.5) node[below]{$x$};
\draw[->,ultra thick] (-1,0)--(8.5,0) node[right]{$y$};
\foreach \x in {1, ..., 8} \draw (0, \x) node[left]{$\x$};
\foreach \x in {0, ..., 7} \draw (\x+1,0) node[above]{$\x$};
\foreach \x in {1.5, ..., 6.5} \draw[line width=1.5pt] (\x,\x)--(\x+1,\x);
\foreach \x in {0.5, ..., 6.5} \draw[line width=1.5pt] (\x+1,\x+1)--(\x+1,\x);
\draw[line width=1.5pt](0.5,0.5)--(8,0.5);
\draw (0, 0) node[above left] {O}; 
\draw[line width=1.5pt](7.5,7.5)--(8,7.5);
\draw[line width=1.5pt](0.5,0.5)--(0.5,8);
\node (a) at (7,3) {};
\node (b) at (3,6) {};
\draw[->, bend angle=50, thick] (a) to[bend left] (b);
\draw[->, bend angle=50, thick] (b) to[bend left] (a);
\draw[thin] (5.5,1.5)--(5.5,2.5)(5.5,2.5)--(5.5,2.5)(5.5,2.5)--(6.5,2.5)(6.5,2.5)--(6.5,1.5)(6.5,1.5)--(5.5,1.5);
\draw[thin] (1.5,4.5)--(1.5,5.5)(1.5,5.5)--(1.5,5.5)(1.5,5.5)--(2.5,5.5)(2.5,5.5)--(2.5,4.5)(2.5,4.5)--(1.5,4.5);
\draw (6, 2) node[font=\small]{$a$};
\draw (2, 5) node[font=\small]{$\omega(a)$};
\draw[thin] (6.5,5.5)--(7.5,5.5)(7.5,5.5)--(7.5,6.5);
\draw[thin] (5.5,5.5)--(5.5,6.5)(5.5,6.5)--(6.5,6.5);
\draw (7, 6) node[font=\small]{$b$};
\draw (6, 6) node[font=\small]{$\omega(b)$};
\draw (4.5, 3) node[font=\Large]{$\omega$};
\end{tikzpicture}
\caption{\label{fig4}}
\end{minipage}
\caption*{}
\end{figure}

\begin{dfn}
\label{squareop}
(1) Let $P$ be a shifted tableau such that each positive integer appears at most once in $|P|$. Then $P\cup\omega(P)$ is an unshifted tableau, which we denote by $\bbackslash P$. 

(2) Let $P$ be a shifted tableau such that all of its diagonal letters are unbarred, and each positive integer appears at most once in $|P|$. In this case, $P\cup\ol{\omega}(P)$ is also an unshifted tableau, which we denote by $\bp P$.  
\end{dfn}

\begin{rmk}
The idea behind Definition~\ref{squareop} stems from Haiman~\cite{h}, whose symbol $\bbackslash$ corresponds, in our notation, to $\bp$.
\end{rmk}
\medskip
\begin{ex}
For $P_1=\begin{ytableau}\ol{1}&2&6&\ol{7}\\\none&4&8\\\none&\none&\ol{9}\end{ytableau}$ and
$P_2=\begin{ytableau}2&3&\ol{6}\\\none&5&\ol{7}\\\none&\none&8\end{ytableau}$\,, 
\begin{align*}
\bbackslash P_1=\begin{ytableau}\ol{1}&\ol{1}&2&6&\ol{7}\\2&4&4&8\\6&8&\ol{9}&\ol{9}\\\ol{7}\end{ytableau},\, 
\bp P_2=\begin{ytableau}\ol{2}&2&3&\ol{6}\\\ol{3}&\ol{5}&5&\ol{7}\\6&7&\ol{8}&8\end{ytableau}\,. 
\end{align*}
\end{ex}
\begin{dfn}\label{wsmx}
(1) Let $T$ be an unshifted tableau, $\beta\in\p\cup\ol{\p}$ be a letter, and $X$ be a row or column in $\mathcal{Q}$. We define $T\x\beta$ as follows:

Choose the rule, either the ordinary or dual mixed insertion, that gives $\beta$ the row nature if $X$ is a row and the column nature if $X$ is a column.
Then as in the ordinary or dual mixed insertion, search $X$ from left to right
for an entry greater than $\beta$ if $X$ is a row, and from top to
bottom for an entry greater than or equal to $\beta$ if $X$ is a column.
Letting $\beta'$ be the first encountered such entry and $c$ the cell containing it, modify $T$
into $T'$ by changing the entry of $c$ into $\beta$. Unlike in the case where $X$ is the first row or column, or where $\beta$ comes from the previous row or column in the process of an ordinary or dual mixed insertion, $T'$ may not be an unshifted tableau, but if it is indeed an unshifted tableau, we say that $T\x\beta$ is well defined, and it will denote the result of continuing the search-and-bump
procedure with the chosen rule unchanged until the termination condition of the ordinary or dual mixed insertion is attained. The conformity with the tableau condition will persist all the way to the termination of the process.

(2) Let $P$ be a shifted tableau with $|P|$ having no repeated entries, $\alpha\in\p\cup\ol{\p}$ a letter with $|\alpha|\notin|P|$, and $X$ a row or column of $\mathcal{O}$. We define $P\wsmx\alpha$ as follows:

Choose the initial rule, either the ordinary or dual mixed bumping, that gives $\alpha$ the row nature if $X$ is a row and the column nature if $X$ is a column. Then as in the Worley--Sagan mixed insertion, search $X$ (left to right if $X$ is a row and top to bottom if $X$ is a column) for an entry larger than $\alpha$, and letting $\alpha'$ be the first encountered such entry and $c$ the cell containing it, modify $P$ into $P'$ by changing the entry of $c$ into $\alpha$. Unlike in the Worley--Sagan mixed insertion, $P'$ may not be a shifted tableau. If $P'$ is a shifted tableau, we continue bumping as in the Worley--Sagan mixed insertion, including the rule change whenever appropriate, until the procedure terminates as an empty cell is filled, and we define $P\wsmx\alpha$ to be the tableau resulting from this procedure; in this case we say that $P\wsmx\alpha$ is well defined. So long as $P'$ obtained in the first step satisfies the tableau condition, the conformity with the tableau condition persists until the procedure terminates.
\end{dfn}
\begin{rmk}\label{aboveleft}
In Definition~\ref{wsmx}, to determine whether $T'$ or $P'$ is a tableau, it suffices to compare the new entry in $c$ with its neighbor above it if $X$ is a row, and with its neighbor to the left of it if $X$ is a column.
\end{rmk}
\begin{ex}
Definition~\ref{wsmx} gives a notation for the result of an ordinary or dual mixed insertion using the result of some of its bumpings.
For instance, in the calculation
\[\begin{ytableau}1 & \ol{2} \\ 2 & 3\end{ytableau}\dm \ol{1},\]
which is the process in Example~\ref{dmeg} to obtain $P_5$ from $P_4$, the first bumping occurs at cell $(1, 1)$ where the newly inserted letter $\ol{1}$ bumps the old entry $1$. The letter has the column nature in the dual rule, and so is sent to the second column, hence we have 
\[\begin{ytableau}1 & \ol{2} \\ 2 & 3\end{ytableau}\dm \ol{1}=\begin{ytableau}1 & \ol{2} \\ 2 & 3\end{ytableau}\xleftarrow[R_1]{}\ol{1}=\begin{ytableau}\ol{1} & \ol{2} \\ 2 & 3\end{ytableau}\xleftarrow[C_2]{}1,\]
where the symbols $R_1$ and $C_2$ denote the first row and the second column, respectively. Equalities may continue: 
\[=\begin{ytableau}\ol{1} & 1 \\ 2 & 3\end{ytableau}\xleftarrow[R_2]{} \ol{2}=\begin{ytableau}\ol{1} & 1 \\ \ol{2} & 3\end{ytableau}\xleftarrow[C_2]{}2=\begin{ytableau}\ol{1} & 1 \\ \ol{2} & 2\end{ytableau}\xleftarrow[C_3]{}3 = \begin{ytableau}\ol{1} & 1 & 3 \\ \ol{2} & 2\end{ytableau}\,.\]
\end{ex}
In the following, any row or column of $\mathcal{O}$ is also considered as a row or column of $\mathcal{Q}$.

\begin{prop}\label{6m}
Let $P$ be a shifted tableau with $|P|$ having no repeated entries, and $\alpha\in\p\cup\ol{\p}$ a letter with $|\alpha|\notin|P|$. Suppose that $X$ is a row or column, that $P\wsmx\alpha$ is well defined, and that, if $X$ is a column, $\alpha$ does not enter a diagonal cell in $P\wsmx\alpha$. Then
\begin{align}
\bbackslash\left(P\wsmx\alpha\right)&=\bbackslash P\ox\alpha\x\alpha \label{6.1}\\
&=\bbackslash P\x\alpha\ox\alpha \label{6.2}
\end{align}
holds.
Here, $\bbackslash P\ox\alpha=(\bbackslash P)\ox\alpha$, $\bbackslash P\ox\alpha\x\alpha=\left((\bbackslash P)\ox\alpha\right)\x\alpha$, and so on.
\end{prop}

\begin{rmk}
The statement of Proposition~\ref{6m} includes the assertion that 
$\bbackslash P\ox\alpha$, $\bbackslash P\ox\alpha\x\alpha$, 
$\bbackslash P\x\alpha$, and $\bbackslash P\x\alpha\ox\alpha$ 
are well defined.
See Remark~\ref{aboveleft} for the well-definedness of these tableaux.
\end{rmk}

\begin{proof}
We prove this proposition by induction on the number $N$ of times diagonal letters are bumped during the insertion $\wsmx\alpha$ applied to $P$. Note that if an empty diagonal cell gets filled, we count this as a bumping of a diagonal letter. (If we interpret empty cells as being filled with $\infty$, then $\infty$ is considered the bumped letter in such a case. )
First, consider the base case $N=0$. We extract this as a lemma for repeated use in later arguments.

\begin{lem}\label{6l} Under the assumption of Proposition~\ref{6m}, let $\alpha_0=\alpha$, and suppose that in the insertion $\wsmx\alpha$ applied to $P$, $\alpha_k$ bumps $\alpha_{k+1}$ located in the cell $c_{k+1}$ ($0\leq k\leq n-1$). Assume that $\alpha_n=\infty$, i.e., the insertion terminates when the empty cell $c_n$ is filled with $\alpha_{n-1}$. 
Define $P_0=P$ and $X_0=X$. For each integer $k$ with $1\leq k\leq n$, let $P_k$ denote the tableau immediately after replacing $\alpha_k$ in $c_k$ with $\alpha_{k-1}$ during the insertion process. 
Let $X_k$ be the row or column into which $\alpha_k$ is inserted after being bumped from $c_k$. 
Suppose that all of $c_1, \dots, c_n$ are off-diagonal. 

With abbreviations $A_k, B_k$ of sentences as defined below, the following assertions \textup{(1)}--\textup{(4)} hold:
\begin{enumerate}[label=\textup{(\arabic*)}, labelsep=0.6em]
\item[(1)]\, 
In $\ox\alpha$ into $\bbackslash P$, $A_k$ holds. 
\item[(2)]\,  
In $\x\alpha$ into $\bbackslash P\ox\alpha$, $B_k$ holds. 
\item[(3)]\, 
In $\x\alpha$ into $\bbackslash P$, $B_k$ holds. 
\item[(4)]\, 
In $\ox\alpha$ into $\bbackslash P\x\alpha$, $A_k$ holds.
\end{enumerate}

Abbreviations: 
\begin{description}
 \item[$A_k$]: 
 $\begin{cases}
  \alpha_0 \text{ bumps } \alpha_1 \text{ in } \omega(c_1), & \quad (k=1), \\
  \alpha_{k-1} \text{ bumped from } \omega(c_{k-1}) \text{ enters } \omega(X_{k-1}) \text{ and bumps } \alpha_k \text{ in } \omega(c_k), & \quad (2\leq k\leq n).
 \end{cases}$
 \item[$B_k$]: 
 $\begin{cases}
  \alpha_0 \text{ bumps } \alpha_1 \text{ in } c_1, & \quad (k=1), \\
  \alpha_{k-1} \text{ bumped from } c_{k-1} \text{ enters } X_{k-1} \text{ and bumps } \alpha_k \text{ in } c_k, & \quad (2\leq k\leq n).
 \end{cases}$
\end{description}
\end{lem}

\begin{proof}[\textrm{Proof of Lemma~\ref{6l}}]
(1) We assume that $A_1, \dots, A_{k-1}$ holds, and we will show $A_k$. In case where $k=1$, we assume nothing. Let $T_0=\bbackslash P$ and $T_{k-1}$ be the tableau right after the replacement of $\alpha_{k-1}$ in $\omega(c_{k-1})$ with $\alpha_{k-2}$ in the insertion $\ox\alpha$ applied to $\bbackslash P$ if $k\geq 2$. 

If $X_{k-1}$ is a column, the arrangement of letters in the below diagonal part of $\omega(X_{k-1})$ of $T_{k-1}$ is the same as in $X_{k-1}$ of $P_{k-1}$. Hence, $A_k$ holds. 
If $X_{k-1}$ is a row, the above diagonal part of $\omega(X_{k-1})$ of $T_{k-1}$ is filled with letters smaller than $T_{k-1}(\omega(\Delta\cap X_{k-1}))=T_{k-1}(\Delta\cap X_{k-1})=T(\Delta\cap X_{k-1})=P(\Delta\cap X_{k-1})$. Because $c_k$ is not diagonal, $P(\Delta\cap X_{k-1})<\alpha_{k-1}$. Hence $\alpha_{k-1}$ does not bump above diagonal cells. Moreover, the arrangement of letters in the below diagonal part of $\omega(X_{k-1})$ is the same as in  $X_{k-1}$ of $P_{k-1}$. Therefore, $\alpha_{k-1}$ enters $\omega(X_{k-1})$ and bumps $\alpha_k$ in $\omega(c_k)$. 

(2) We assume that $B_1, \dots, B_{k-1}$ holds, and we will show $B_k$. If $k=1$, we assume nothing. 
Let $U_0=\bbackslash P\ox\alpha$ and $U_{k-1}$ be the tableau right after the replacement of $\alpha_{k-1}$ in $c_{k-1}$ with $\alpha_{k-2}$ in the insertion $\x\alpha$ applied to $\bbackslash P\ox\alpha$ if $k\geq 2$. 

If $X_{k-1}$ is a column, the arrangement of letters in the above diagonal part of $X_{k-1}$ of $U_{k-1}$ is the same as in $X_{k-1}$ of $P_{k-1}$. Hence, $B_k$ holds. 
If $X_{k-1}$ is a row, the below diagonal part of $X_{k-1}$ of $U_{k-1}$ is filled with letters smaller than or equal to $U_{k-1}(\Delta\cap X_{k-1})=P(\Delta\cap X_{k-1})$. Because $c_k$ is not diagonal, $P(\Delta\cap X_{k-1})<\alpha_{k-1}$. Hence $\alpha_{k-1}$ does not bump below diagonal cells. Moreover, the arrangement of letters in the above diagonal part of $X_{k-1}$ is the same as in  $X_{k-1}$ of $P_{k-1}$. Therefore, $\alpha_{k-1}$ enters $X_{k-1}$ and bumps $\alpha_k$ in $c_k$. 

(3) The proof of (3) is the same as (2) except for a slight modification. We should modify the definition of $U_{k-1}$ in the proof of (2) to the tableau right after the replacement of $\alpha_{k-1}$ in $c_{k-1}$ with $\alpha_{k-2}$ in the insertion $\x\alpha$ applied to $\bbackslash P$. This modified proof works for (3). 

(4) The proof of (4) is the same as (1) except for a slight modification. We should modify the definition of $T_{k-1}$ in the proof of (1) to the tableau right after the replacement of $\alpha_{k-1}$ in $\omega(c_{k-1})$ with $\alpha_{k-2}$ in the insertion $\ox\alpha$ applied to $\bbackslash P\x\alpha$. This modified proof works for (4). 
\end{proof}
\medskip
Let us return to the proof of Proposition~\ref{6m}. The proof is divided into parts \textup{I}--\textup{III}.

\bigskip

\noindent
\textbf{Part \textup{I}.} \textit{Proof of \eqref{6.1} and \eqref{6.2} in the case $N=0$.}

By Lemma~\ref{6l}\textup{(1)}, we have
\[
\bbackslash P\ox\alpha=P\cup\omega
\left(P\wsmx\alpha\right). 
\]
Moreover, by Lemma~\ref{6l}(2), we have 
\[
\bbackslash P\ox\alpha\x\alpha=\left(P\wsmx\alpha\right)\cup
\omega
\left(P\wsmx\alpha\right), 
\]
which is equivalent to \eqref{6.1}. 
Similarly, by Lemma~\ref{6l}(3), we have 
\[
\bbackslash P\x\alpha=\left(P\wsmx\alpha\right)\cup\omega(P), 
\]
and by Lemma~\ref{6l}(4), we have 
\[
\bbackslash P\x\alpha\ox\alpha=\left(P\wsmx\alpha\right)\cup
\omega
\left(P\wsmx\alpha\right), 
\]
which is equivalent to \eqref{6.2}. 
This completes the proof for the case $N=0$. 
\hfill $\diamondsuit$

\bigskip

Now, we have shown Proposition~\ref{6m} in the case $N=0$. 
Suppose that the diagonal letters are bumped $N\geq 1$ times during the insertion $\wsmx\alpha$ applied to $P$. 
Let $\alpha_0=\alpha$. Suppose $\alpha_k$ bumps $\alpha_{k+1}$ in $c_{k+1}$ during this insertion for $0\leq k\leq n-1$, where $c_1, \dots, c_{n-1}$ are off-diagonal and $c_n$ is diagonal. 
Let $X_k$ be the row or column $\alpha_k$ is inserted after being bumped from $c_k$ for $1\leq k\leq n$. 
Let $Y$ be the column containing $c_n$. Note that, since $c_n$ is diagonal, $X_{n-1}$ must be a row, and $X_n$ must be a column. Hence $Y\neq X_{n-1}$. 
For a cell $c=(i, j)\in\mathcal{O}$, define 
\begin{align*}
S_\mathcal{O}(c)=\{\,(i', j')\in\mathcal{O} \mid i'\leq i, j'\leq j\,\}. 
\end{align*}
Then, set 
\[
D_0 = \bigcup_{r=1}^{n-1} S_{\mathcal{O}}(c_r), 
\qquad
D_1 = D_0 \cup S_{\mathcal{O}}(c_n),
\qquad
D_2 = \mathcal{O}\setminus D_1 .
\]

\bigskip

\noindent
\textbf{Part \textup{II}.} \textit{Proof of \eqref{6.1} in the general case.}

\begin{figure}
\centering
\begin{minipage}{0.45\linewidth}
\begin{tikzpicture}[x=1.2cm, y=1.2cm, yscale=-1]
\fill[black!10] (2.5,0.5) rectangle (3.5,6.5);
\fill[black!30] (3.5,0.5) rectangle (4.5,6.5);
\fill[blue!10, opacity=0.5] (0.5,3.5) rectangle (6.5,4.5);
\foreach \x in {1.5, ..., 6.5} \draw[line width=1.5pt] (\x-1,\x)--(\x,\x);
\foreach \x in {0.5, ..., 5.5} \draw[line width=1.5pt] (\x,\x+1)--(\x,\x);
\draw[line width=1.5pt](2.5,2.5)--(3.5,2.5);
\draw[line width=1.5pt](3.5,2.5)--(3.5,3.5);
\draw[line width=1.5pt](2.5,3.5)--(2.5,4.5);
\draw[line width=1.5pt](2.5,4.5)--(4.5,4.5);
\draw[line width=1.5pt](3.5,3.5)--(4.5,3.5)--(4.5,4.5);
\draw[line width=1.5pt](1.5,5.5)--(1.5,6.5)--(2.5,6.5)--(2.5,5.5)--(1.5,5.5);
\draw(2.7,2.7)node[font=\scriptsize]{$d$};
\draw(3.7,3.7)node[font=\scriptsize]{$c_n$};
\draw(2.85,3.65)node[font=\scriptsize]{$\omega(c_n)$};
\draw(1.95,5.65)node[font=\scriptsize]{$\omega(c_{n-1})$};
\draw(3,4)node[font=\normalsize]{$\alpha_n$};
\draw(4,4)node[font=\normalsize]{$\alpha_n$};
\draw(2,6)node[font=\normalsize]{$\alpha_{n-2}$};
\draw(3,6)node[font=\scriptsize]{$\omega(X_{n-1})$};
\draw(6,4)node[font=\scriptsize]{$X_{n-1}$};
\draw(4,6)node[font=\scriptsize]{$Y$};
\draw(3,1.5)node[font=\normalsize]{$\downarrow$};
\draw(3,1)node[font=\normalsize]{$\alpha_{n-1}$};
\draw(4,0.5)node[font=\normalsize]{\phantom{(Case 2)}};
\end{tikzpicture}
\caption{\label{tn-1}$T_{n-1}$ and $\alpha_{n-1}$ bumped from $\omega(c_{n-1})$. }
\end{minipage}
\hspace{10pt}
\begin{minipage}{0.45\linewidth}
\begin{tikzpicture}[x=1.2cm, y=1.2cm, yscale=-1]
\fill[black!50] (4.5,0.5) rectangle (5.5,6.5);
\fill[black!30] (3.5,0.5) rectangle (4.5,6.5);
\fill[blue!10, opacity=0.5] (0.5,3.5) rectangle (6.5,4.5);
\fill[blue!30, opacity=0.5] (0.5,4.5) rectangle (6.5,5.5);
\foreach \x in {1.5, ..., 6.5} \draw[line width=1.5pt] (\x-1,\x)--(\x,\x);
\foreach \x in {0.5, ..., 5.5} \draw[line width=1.5pt] (\x,\x+1)--(\x,\x);
\draw[line width=1.5pt](2.5,2.5)--(3.5,2.5);
\draw[line width=1.5pt](3.5,2.5)--(3.5,3.5);
\draw[line width=1.5pt](2.5,3.5)--(2.5,4.5);
\draw[line width=1.5pt](2.5,4.5)--(4.5,4.5);
\draw[line width=1.5pt](3.5,3.5)--(4.5,3.5)--(4.5,4.5);
\draw(2.7,2.7)node[font=\scriptsize]{$d$};
\draw(3.7,3.7)node[font=\scriptsize]{$c_n$};
\draw(2.85,3.65)node[font=\scriptsize]{$\omega(c_n)$};
\draw(3,4)node[font=\normalsize]{$\alpha_{n-1}$};
\draw(4,4)node[font=\normalsize]{$\alpha_n$};
\draw(5,6)node[font=\scriptsize]{$X_n$};
\draw(6,4)node[font=\scriptsize]{$X_{n-1}$};
\draw(6,5)node[font=\scriptsize]{$\omega(X_n)$};
\draw(4,6)node[font=\scriptsize]{$Y$};
\draw(4,1.5)node[font=\normalsize]{$\downarrow$};
\draw(4,1)node[font=\normalsize]{$\alpha_n$};
\draw(4,0.5)node[font=\normalsize]{(Case 2)};
\draw(1,5)node[font=\normalsize]{$\alpha_n$};
\draw(1.5,5)node[font=\normalsize]{$\rightarrow$};
\draw(1,4.5)node[font=\normalsize]{(Case 1)};
\end{tikzpicture}
\caption{\label{tn}$T_{n}$ and $\alpha_n$ bumped from $\omega(c_n)$. }
\end{minipage}
\caption*{}
\end{figure}
 
We consider $\ox\alpha$ applied to $\bbackslash P$. 
When applying $\ox\alpha$ to $\bbackslash (P|_{D_0\setminus\{c_{n-1}\}})$, 
the statement $A_k$ $(1\leq k\leq n-1)$ holds by Lemma~\ref{6l}(1). 
In $\ox\alpha$ applied to $\bbackslash P$, the letter moves in the same way as $\ox\alpha$ applied to $\bbackslash (P|_{D_0\setminus\{c_{n-1}\}})$ until $\alpha_{n-1}$ in $\omega(c_{n-1})$ is bumped. 
Let $T_{n-1}$ be the tableau right after the replacement of $\alpha_{n-1}$ in $\omega(c_{n-1})$. See Figure~\ref{tn-1}. 
Note that, since $T_{n-1}$ differs from $\bbackslash P$ only in the letters placed in $\omega(c_1), \dots, \omega(c_{n-1})$, all of which are off-diagonal, we have $T_{n-1}(c)=P(c)$ for all diagonal cells $c$. 
After being bumped from $\omega(c_{n-1})$, $\alpha_{n-1}$ enters $\omega(X_{n-1})$. 
Since both the cell $d=\Delta\cap\omega(X_{n-1})$ and the off-diagonal cell $c_{n-1}$ lie in the row immediately above $X_{n-1}$, $T_{n-1}(d)=P(d)<P(c_{n-1})=\alpha_{n-1}$. 
Hence $\alpha_{n-1}$ bumps $\alpha_n$ in $\omega(c_n)$. 

Let $T_n$ be the tableau right after the replacement of $\alpha_n$ in $\omega(c_n)$. See Figure~\ref{tn}. 
After being bumped from $\omega(c_n)$, there are two possible cases for the behavior of $\alpha_n$. 
\begin{itemize}
\item Case 1: (i) $\ox\alpha$ is an ordinary mixed bumping and $\alpha_n$ is unbarred, or (ii) $\ox\alpha$ is a dual mixed bumping and $\alpha_n$ is barred. In this case, $\alpha_n$ is inserted into the row immediately below $X_{n-1}$, which is none other than $\omega(X_n)$, as can be seen in Figure~\ref{tn}.

\item Case 2: (i) $\ox\alpha$ is an ordinary mixed bumping and $\alpha_n$ is barred, or (ii) $\ox\alpha$ is a dual mixed bumping and $\alpha_n$ is unbarred. In this case, $\alpha_n$ is inserted into the column $Y$ and bumps $\alpha_n$ in $c_n$. This occurs because $T_n(c_n)=\alpha_n$, and 
the letters located strictly above $c_n$ in $Y$ are smaller than $T_n(c_n)$. After being bumped from $c_n$, $\alpha_n$ is inserted into the column immediately right to $Y$, which is none other than $X_n$, as can be seen in Figure~\ref{tn}.
\end{itemize}

Therefore, 
\begin{align}\label{wsm1-1}
\bbackslash P\ox\alpha=P|_{D_1}\cup
\omega\left(\left.\left(P\wsmx\alpha\right)\right|_{D_1}\right)\cup
\left(
\left.\left(\bbackslash P\right)\right|_{D_2\cup \omega(D_2)}\xleftarrow[Z]{}\alpha_n
\right). 
\end{align}
Here,
\begin{align}\label{z}
Z = 
\begin{cases}
\omega(X_n) & \text{in Case 1}, \\
X_n & \text{in Case 2}.
\end{cases}
\end{align}

\begin{figure}
\centering
\begin{minipage}{0.45\linewidth}
\begin{tikzpicture}[x=1.2cm, y=1.2cm, yscale=-1]
\fill[blue!10, opacity=0.5] (0.5,3.5) rectangle (6.5,4.5);
\foreach \x in {1.5, ..., 6.5} \draw[line width=1.5pt] (\x-1,\x)--(\x,\x);
\foreach \x in {0.5, ..., 5.5} \draw[line width=1.5pt] (\x,\x+1)--(\x,\x);
\draw[line width=1.5pt](2.5,3.5)--(2.5,4.5)--(4.5,4.5);
\draw[line width=1.5pt](5.5,2.5)--(6.5,2.5)--(6.5,3.5)--(5.5,3.5)--(5.5,2.5);
\draw[line width=1.5pt](3.5,3.5)--(4.5,3.5)--(4.5,4.5);
\draw(3.7,3.7)node[font=\scriptsize]{$c_n$};
\draw(2.85,3.65)node[font=\scriptsize]{$\omega(c_n)$};
\draw(5.8,2.7)node[font=\scriptsize]{$c_{n-1}$};
\draw(3,4)node[font=\normalsize]{$\alpha_{n-1}$};
\draw(4,4)node[font=\normalsize]{$\alpha_n$};
\draw(6,3)node[font=\normalsize]{$\alpha_{n-2}$};
\draw(6,4)node[font=\scriptsize]{$X_{n-1}$};
\draw(1.6,4)node[font=\normalsize]{$\rightarrow$};
\draw(1,4)node[font=\normalsize]{$\alpha_{n-1}$};
\draw(4,0.5)node[font=\normalsize]{\phantom{(Case 2)}};
\end{tikzpicture}
\caption{\label{un-1}$U_{n-1}$ and $\alpha_{n-1}$ bumped from $c_{n-1}$. }
\end{minipage}
\hspace{10pt}
\begin{minipage}{0.45\linewidth}
\begin{tikzpicture}[x=1.2cm, y=1.2cm, yscale=-1]
\fill[black!50] (4.5,0.5) rectangle (5.5,6.5);
\fill[blue!30, opacity=0.5] (0.5,4.5) rectangle (6.5,5.5);
\foreach \x in {1.5, ..., 6.5} \draw[line width=1.5pt] (\x-1,\x)--(\x,\x);
\foreach \x in {0.5, ..., 5.5} \draw[line width=1.5pt] (\x,\x+1)--(\x,\x);
\draw[line width=1.5pt](2.5,3.5)--(2.5,4.5);
\draw[line width=1.5pt](2.5,4.5)--(4.5,4.5);
\draw[line width=1.5pt](3.5,3.5)--(4.5,3.5)--(4.5,4.5);
\draw(3.7,3.7)node[font=\scriptsize]{$c_n$};
\draw(2.85,3.65)node[font=\scriptsize]{$\omega(c_n)$};
\draw(3,4)node[font=\normalsize]{$\alpha_{n-1}$};
\draw(4,4)node[font=\normalsize]{$\alpha_{n-1}$};
\draw(5,6)node[font=\scriptsize]{$X_n$};
\draw(6,5)node[font=\scriptsize]{$\omega(X_n)$};
\draw(5,1.5)node[font=\normalsize]{$\downarrow$};
\draw(5,1)node[font=\normalsize]{$\alpha_n$};
\draw(5,0.5)node[font=\normalsize]{(Case 1)};
\draw(1,5)node[font=\normalsize]{$\alpha_n$};
\draw(1.5,5)node[font=\normalsize]{$\rightarrow$};
\draw(1,4.5)node[font=\normalsize]{(Case 2)};
\end{tikzpicture}
\caption{\label{un}$U_{n}$ and $\alpha_n$ bumped from $c_n$. }
\end{minipage}
\caption*{}
\end{figure}

Next, consider applying $\x\alpha$ to $\bbackslash P\ox\alpha$. When applying $\x\alpha$ to the tableau obtained by restricting $\bbackslash P \ox \alpha$ to the unshifted diagram $D_0 \cup \omega(D_0)\setminus\{c_{n-1}, \omega(c_{n-1})\}$, 
the statements $B_k$ $(1\leq k\leq n-1)$ hold by \eqref{wsm1-1} and Lemma~\ref{6l}(2). 
When applying $\x\alpha$ to $\bbackslash P\ox\alpha$, the letters move in the same way as in this insertion until $\alpha_{n-1}$ in $c_{n-1}$ is bumped. 
Let $U_{n-1}$ be the tableau right after the replacement of $\alpha_{n-1}$ in $c_{n-1}$. See Figure~\ref{un-1}. 
After being bumped from $c_{n-1}$, $\alpha_{n-1}$ is inserted into $X_{n-1}$ and 
bumps $\alpha_n$ in $c_n$, because $U_{n-1}(\omega(c_n))=\left(\bbackslash P\ox\alpha\right)(\omega(c_n))=\alpha_{n-1}$. 

Let $U_{n}$ be the tableau right after the replacement of $\alpha_n$ in $c_n$. See Figure~\ref{un}. 
$\alpha_n$, after being bumped from $c_n$, enters $X_n$ in Case 1 and $\omega(X_n)$ in Case 2. 
Since the insertion we are considering is $\x\alpha$, note that after $\alpha_n$ is bumped, whether it enters a row or column is the opposite of the case illustrated in Figure~\ref{tn}, which corresponds to $\ox\alpha$. 

From the above, 
\begin{align}\label{wsm1-2}
\bbackslash P\ox\alpha\x\alpha=\bbackslash \left(\left.\left(P\wsmx\alpha\right)\right|_{D_1}\right)\cup
\left(
\left.\left(\bbackslash P\right)\right|_{D_2\cup \omega(D_2)}\xleftarrow[Z]{}\alpha_n\xleftarrow[\!W\!]{}\alpha_n
\right). 
\end{align}
Here, 
\begin{align}\label{w}
W= 
\begin{cases}
X_n & \text{in Case 1}, \\
\omega(X_n) & \text{in Case 2}.
\end{cases}
\end{align}

In $\xleftarrow[\!\!X_n\!\!]{\!\!\!\!\!\text{WSM}\!\!\!\!\!}\alpha_n$ applied to $P|_{D_2}$, 
the letters move in the same way as in $P\wsmx\alpha$ after the bumping of $\alpha_n$ in $c_n$. 
Therefore, when we apply $\xleftarrow[\!\!X_n\!\!]{\!\!\!\!\!\text{WSM}\!\!\!\!\!}\alpha_n$ to $P|_{D_2}$, the diagonal letters are bumped exactly $N-1$ times. 
Hence, we can apply the induction hypothesis to $P|_{D_2}\xleftarrow[\!\!X_n\!\!]{\!\!\!\!\!\text{WSM}\!\!\!\!\!}\alpha_n$. 
By $\left.\left(\bbackslash P\right)\right|_{D_2\cup \omega(D_2)}=\bbackslash(P|_{D_2})$, 
\begin{align*}
\left.\left(\bbackslash P\right)\right|_{D_2\cup \omega(D_2)}\xleftarrow[Z]{}\alpha_n\xleftarrow[\!W\!]{}\alpha_n=\bbackslash(P|_{D_2})\xleftarrow[Z]{}\alpha_n\xleftarrow[W]{}\alpha_n. 
\end{align*}
By \eqref{z} and \eqref{w}, together with the induction hypothesis, this tableau is equal to $\bbackslash\left(P|_{D_2}\xleftarrow[\!\!X_n\!\!]{\!\!\!\!\!\text{WSM}\!\!\!\!\!}\alpha_n\right)$. 
From this fact, together with $P|_{D_2}\xleftarrow[\!\!X_n\!\!]{\!\!\!\!\!\text{WSM}\!\!\!\!\!}\alpha_n=\left.\left(P\wsmx\alpha\right)\right|_{D_2}$ and \eqref{wsm1-2}, we obtain
\begin{align}\label{wsm1-3}
\left.\left(\bbackslash P\ox\alpha\x\alpha\right)\right|_{D_2\cup \omega(D_2)}&=
\bbackslash\left(\left.\left(P\wsmx\alpha\right)\right|_{D_2}\right). 
\end{align}
\eqref{wsm1-2} and \eqref{wsm1-3} complete the proof of \eqref{6.1}. 
\hfill $\diamondsuit$

\bigskip

\noindent
\textbf{Part \textup{III}.} \textit{Proof of \eqref{6.2} in the general case.}

\begin{figure}
\centering
\begin{minipage}{0.45\linewidth}
\begin{tikzpicture}[x=1.2cm, y=1.2cm, yscale=-1]
\fill[blue!10, opacity=0.5] (0.5,3.5) rectangle (6.5,4.5);
\foreach \x in {1.5, ..., 6.5} \draw[line width=1.5pt] (\x-1,\x)--(\x,\x);
\foreach \x in {0.5, ..., 5.5} \draw[line width=1.5pt] (\x,\x+1)--(\x,\x);
\draw[line width=1.5pt](2.5,3.5)--(2.5,4.5);
\draw[line width=1.5pt](2.5,4.5)--(4.5,4.5);
\draw[line width=1.5pt](3.5,3.5)--(4.5,3.5)--(4.5,4.5);
\draw[line width=1.5pt](5.5,2.5)--(6.5,2.5)--(6.5,3.5)--(5.5,3.5)--(5.5,2.5);
\draw[line width=1.5pt](3.5,3.5)--(3.5,2.5)--(4.5,2.5)--(4.5,3.5); 
\draw[line width=1.5pt](2.5,3.5)--(1.5,3.5)--(1.5,4.5)--(2.5,4.5); 
\draw(3.7,3.7)node[font=\scriptsize]{$c_n$};
\draw(2.85,3.65)node[font=\scriptsize]{$\omega(c_n)$};
\draw(4,2.7)node[font=\scriptsize]{$(p\!-\!1, p)$};
\draw(2,3.7)node[font=\scriptsize]{$(p, p\!-\!2)$};
\draw(5.8,2.7)node[font=\scriptsize]{$c_{n-1}$};
\draw(3,4)node[font=\normalsize]{$\alpha_n$};
\draw(4,4)node[font=\normalsize]{$\alpha_n$};
\draw(6,4)node[font=\scriptsize]{$X_{n-1}$};
\draw(6,3)node[font=\normalsize]{$\alpha_{n-2}$};
\draw(1.2,4)node[font=\normalsize]{$\rightarrow$};
\draw(0.6,4)node[font=\normalsize]{$\alpha_{n-1}$};
\draw(4,0.5)node[font=\normalsize]{\phantom{(Case 2)}};
\end{tikzpicture}
\caption{\label{vn-1}$V_{n-1}$ and $\alpha_{n-1}$ bumped from $c_{n-1}$. }
\end{minipage}
\hspace{10pt}
\begin{minipage}{0.45\linewidth}
\begin{tikzpicture}[x=1.2cm, y=1.2cm, yscale=-1]
\fill[black!30] (3.5,0.5) rectangle (4.5,6.5);
\fill[black!50] (4.5,0.5) rectangle (5.5,6.5);
\fill[blue!30, opacity=0.5] (0.5,4.5) rectangle (6.5,5.5);
\foreach \x in {1.5, ..., 6.5} \draw[line width=1.5pt] (\x-1,\x)--(\x,\x);
\foreach \x in {0.5, ..., 5.5} \draw[line width=1.5pt] (\x,\x+1)--(\x,\x);
\draw[line width=1.5pt](2.5,3.5)--(2.5,4.5);
\draw[line width=1.5pt](2.5,4.5)--(4.5,4.5);
\draw[line width=1.5pt](3.5,3.5)--(4.5,3.5)--(4.5,4.5);
\draw(3.7,3.7)node[font=\scriptsize]{$c_n$};
\draw(2.85,3.65)node[font=\scriptsize]{$\omega(c_n)$};
\draw(3,4)node[font=\normalsize]{$\alpha_{n-1}$};
\draw(4,4)node[font=\normalsize]{$\alpha_n$};
\draw(6,5)node[font=\scriptsize]{$\omega(X_n)$};
\draw(4,6)node[font=\scriptsize]{$Y$};
\draw(5,6)node[font=\scriptsize]{$X_n$};
\draw(1.3,5)node[font=\normalsize]{$\rightarrow$};
\draw(0.7,5)node[font=\normalsize]{$\alpha_n$};
\draw(4,1.5)node[font=\normalsize]{$\downarrow$};
\draw(4,1)node[font=\normalsize]{$\alpha_n$};
\draw(4,0.5)node[font=\normalsize]{(Case 1)};
\draw(0.5,4.5)node[font=\normalsize]{(Case 2)};
\draw(4,0.5)node[font=\normalsize]{\phantom{(Case 2)}};
\end{tikzpicture}
\caption{\label{vn}$V_{n}$ and $\alpha_n$ bumped from $\omega(c_n)$. }
\end{minipage}
\caption*{}
\end{figure}
We consider $\x\alpha$ applied to $\bbackslash P$. 
When applying $\x\alpha$ to $\bbackslash\ (P|_{D_0\setminus\{c_{n-1}\}})$, 
the statememts $B_k$ $(1\leq k\leq n-1)$ hold by Lemma~\ref{6l}(3). 
In the application of $\x\alpha$ to $\bbackslash P$, the letters move in the same way as this insertion until $\alpha_{n-1}$ in $c_{n-1}$ is bumped. Let $V_{n-1}$ be the tableau right after the replacement of $\alpha_{n-1}$ in $c_{n-1}$. Let $(p, p)$ be the coordinate of $c_n$. See Figure~\ref{vn-1}. Note that the letters in the cells of $V_{n-1}$ other than $c_1, \dots, c_{n-1}$ remain unchanged from $\bbackslash P$ under the insertion. 

After being bumped from $c_{n-1}$, $\alpha_{n-1}$ enters $X_{n-1}$. 
We have $\omega(p, p-2)=(p-1, p)$, and $c_{n-1}$ is located in the off-diagonal part of the $(p-1)$-th row. Thus, $V_{n-1}(p, p-2)=(\bbackslash P)(p, p-2)=P(p-1, p)\leq P(c_{n-1})=\alpha_{n-1}$. Hence, $\alpha_{n-1}$ bumps $\alpha_n$ in $\omega(c_n)$. 

Let $V_n$ be the tableau right after the replacement of $\alpha_n$ in $\omega(c_n)$. 
After being bumped from $\omega(c_n)$, $\alpha_n$ enters $Y$ in Case 1 and $\omega(X_n)$ in Case 2. 
Since we are currently dealing with the insertion $\x\alpha$,  the roles of rows and columns are interchanged compared to the case in Figure~\ref{tn}, illustrating a similar situation for $\ox\alpha$ with regard to the destination of $\alpha_n$. 
In Case 1, $\alpha_n$ bumps the letter $\alpha_n$ located in $c_n$ because $V_n(c_n)=\alpha_n$, and 
the letters located strictly above $c_n$ in $Y$ are smaller than $V_n(c_n)$. 
After being bumped from $c_n$, $\alpha_n$ enters $X_n$. 

Therefore, 
\begin{align}\label{wsm2-1}
\bbackslash P\x\alpha=
\left.\left(P\wsmx\alpha\right)\right|_{D_1\setminus\{c_n\}}\cup
\omega(P|_{D_1\setminus\{c_n\}})&\cup c_n(\alpha_n)\cup \omega(c_n)(\alpha_{n-1})\notag\\&\cup
\left(
\left.\left(\bbackslash P\right)\right|_{D_2\cup \omega(D_2)}\xleftarrow[\!W\!]{}\alpha_n
\right). 
\end{align}
Here, \(c_n(\alpha_n)\) denotes the cell \(c_n\) with entry \(\alpha_n\), \(\omega(c_n)(\alpha_{n-1})\) denotes the cell \(\omega(c_n)\) with entry \(\alpha_{n-1}\), and
\begin{align}\label{ww}
W= 
\begin{cases}
X_n & \text{in Case 1}, \\
\omega(X_n) & \text{in Case 2}. 
\end{cases}
\end{align}
This definition of $W$ is the same as the one in \eqref{w}. 

\begin{figure}
\centering
\begin{minipage}{0.45\linewidth}
\begin{tikzpicture}[x=1.2cm, y=1.2cm, yscale=-1]
\fill[black!10] (2.5,0.5) rectangle (3.5,6.5);
\fill[black!30] (3.5,0.5) rectangle (4.5,6.5);
\foreach \x in {1.5, ..., 6.5} \draw[line width=1.5pt] (\x-1,\x)--(\x,\x);
\foreach \x in {0.5, ..., 5.5} \draw[line width=1.5pt] (\x,\x+1)--(\x,\x);
\draw[line width=1.5pt](2.5,3.5)--(2.5,4.5)--(3.5,4.5);
\draw[line width=1.5pt](3.5,3.5)--(4.5,3.5)--(4.5,4.5);
\draw[line width=1.5pt](3.5,3.5)--(3.5,2.5)--(4.5,2.5)--(4.5,3.5);
\draw[line width=1.5pt](1.5,5.5)--(1.5,6.5)--(2.5,6.5)--(2.5,5.5)--(1.5,5.5);
\draw[line width=1.5pt](5.5,2.5)--(6.5,2.5)--(6.5,3.5)--(5.5,3.5)--(5.5,2.5);
\draw(2.85,3.65)node[font=\scriptsize]{$\omega(c_n)$};
\draw(1.95,5.65)node[font=\scriptsize]{$\omega(c_{n-1})$};
\draw(3.7,3.7)node[font=\scriptsize]{$c_n$};
\draw(3,4)node[font=\normalsize]{$\alpha_{n-1}$};
\draw(2,6)node[font=\normalsize]{$\alpha_{n-2}$};
\draw(3,1.5)node[font=\normalsize]{$\downarrow$};
\draw(3,1)node[font=\normalsize]{$\alpha_{n-1}$};
\draw(3,6)node[font=\scriptsize]{$\omega(X_{n-1})$};
\draw(4,2.7)node[font=\scriptsize]{$(p\!-\!1, p)$};
\draw(5.8,2.7)node[font=\scriptsize]{$c_{n-1}$};
\draw(6,3)node[font=\normalsize]{$\alpha_{n-2}$};
\draw(4,4)node[font=\normalsize]{$\alpha_n$};
\draw(4,6)node[font=\scriptsize]{$Y$};
\draw(4,0.5)node[font=\normalsize]{\phantom{(Case 2)}};
\end{tikzpicture}
\caption{\label{wn-1}$W_{n-1}$ and $\alpha_{n-1}$ bumped from $\omega(c_{n-1})$. }
\end{minipage}
\hspace{10pt}
\begin{minipage}{0.45\linewidth}
\begin{tikzpicture}[x=1.2cm, y=1.2cm, yscale=-1]
\fill[black!50] (4.5,0.5) rectangle (5.5,6.5);
\fill[blue!30, opacity=0.5] (0.5,4.5) rectangle (6.5,5.5);
\foreach \x in {1.5, ..., 6.5} \draw[line width=1.5pt] (\x-1,\x)--(\x,\x);
\foreach \x in {0.5, ..., 5.5} \draw[line width=1.5pt] (\x,\x+1)--(\x,\x);
\draw[line width=1.5pt](2.5,3.5)--(2.5,4.5);
\draw[line width=1.5pt](2.5,4.5)--(4.5,4.5);
\draw[line width=1.5pt](3.5,3.5)--(4.5,3.5)--(4.5,4.5);
\draw(3.7,3.7)node[font=\scriptsize]{$c_n$};
\draw(2.85,3.65)node[font=\scriptsize]{$\omega(c_n)$};
\draw(3,4)node[font=\normalsize]{$\alpha_{n-1}$};
\draw(4,4)node[font=\normalsize]{$\alpha_{n-1}$};
\draw(5,6)node[font=\scriptsize]{$X_n$};
\draw(6,5)node[font=\scriptsize]{$\omega(X_n)$};
\draw(5,1.5)node[font=\normalsize]{$\downarrow$};
\draw(5,1)node[font=\normalsize]{$\alpha_n$};
\draw(5,0.5)node[font=\normalsize]{(Case 2)};
\draw(1,5)node[font=\normalsize]{$\alpha_n$};
\draw(1.5,5)node[font=\normalsize]{$\rightarrow$};
\draw(1,4.5)node[font=\normalsize]{(Case 1)};
\end{tikzpicture}
\caption{\label{wn}$W_{n}$ and $\alpha_n$ bumped from $c_n$. }
\end{minipage}
\caption*{}
\end{figure}

Next, consider $\ox\alpha$ applied to $\bbackslash P\x\alpha$. 
When applying $\ox\alpha$ to the tableau obtained by restricting $\bbackslash P\x\alpha$ to the shifted diagram $D_0\cup\omega(D_0)\setminus\{c_{n-1}, \omega(c_{n-1})\}$, 
the statements $A_k$ $(1\leq k\leq n-1)$ hold by Lemma~\ref{6l}(4). 
When applying $\ox\alpha$ to the entire tableau $\bbackslash P\x\alpha$, the letters move in the same way as in this insertion until $\alpha_{n-1}$ in $\omega(c_{n-1})$ is bumped. 
Let $W_{n-1}$ be the tableau immediately after $\alpha_{n-1}$ is bumped from $\omega(c_{n-1})$. See Figure~\ref{wn-1}. Note that, because $\alpha_{n-1}$ is inserted into a row in $\wsmx\alpha$ before the diagonal letter $\alpha_n$ in $c_n$ is bumped, $\alpha_{n-1}$ is inserted into a column in $\ox\alpha$. 
After being bumped from $\omega(c_{n-1})$, $\alpha_{n-1}$ enters $\omega(X_{n-1})$ and then bumps $\alpha_{n-1}$ in $\omega(c_n)$, since the cells in $\omega(X_{n-1})$ strictly above $\omega(c_n)$ are filled with letters smaller than $\alpha_{n-1}$. Since $\alpha_{n-1}$ bumps another $\alpha_{n-1}$, $W_{n-1}$ remains unchanged after going through this bumping process. 
Subsequently, the letter $\alpha_{n-1}$ bumped from $\omega(c_n)$ enters $Y$. 
Since $c_{n-1}$ is an off-diagonal cell in the $(p-1)$-th row, we have
$W_{n-1}(p-1, p)\leq W_{n-1}(c_{n-1})=\alpha_{n-2}<\alpha_{n-1}$. 
Hence, $\alpha_{n-1}$ bumps $\alpha_n$ in $c_n$. 

Let $W_n$ be the tableau immediately after the replacement of $\alpha_n$ in $c_n$. See Figure~\ref{wn}. 
After being bumped from $c_n$, $\alpha_n$ enters $\omega(X_n)$ in Case 1 and $X_n$ in Case 2. 
Note that, since we are dealing with the insertion $\ox\alpha$, the behavior of $\alpha_n$ after being bumped is the same as the case illustrated in Figure~\ref{tn}. 

From the above, 
\begin{align}\label{wsm2-2}
\bbackslash P\x\alpha\ox\alpha=\bbackslash \left(\left.\left(P\wsmx\alpha\right)\right|_{D_1}\right)\cup
\left(
\left.\left(\bbackslash P\right)\right|_{D_2\cup \omega(D_2)}\xleftarrow[\!W\!]{}\alpha_n\xleftarrow[Z]{}\alpha_n
\right). 
\end{align}
Here, 
\begin{align}\label{zz}
Z= 
\begin{cases}
\omega(X_n) & \text{in Case 1}, \\
X_n & \text{in Case 2}. 
\end{cases}
\end{align}
This definition of $Z$ is the same as the one in \eqref{z}. 

The rest of the proof proceeds exactly as in the corresponding part of the proof of \eqref{6.1}. 
The reader should replace $\eqref{z}$, $\eqref{wsm1-2}$, and $\eqref{w}$ by $\eqref{ww}$, $\eqref{wsm2-2}$, and $\eqref{zz}$, respectively. 
This completes the proof of \eqref{6.2}.
\end{proof}

We now define a transformation on words, which maps words of length $n$ to biwords of length $2n$.
This is the word analogue of the maps $\bbackslash$ and $\bp$ in the tableau setting, which sends a tableau with $n$ cells to a tableau with $2n$ cells.
\begin{dfn}\label{d1}
Define an injective map $d_1 : B_n\to DC_{2n}$ as follows: 

Let $\pi=\alpha_1\dots\alpha_n\in B_n$. 
For each $k$ satisfying $1\leq k\leq n$, define column vectors $u_{2k-1}$ and $u_{2k}$ by
\begin{align*}
u_{2k-1}=
\begin{pmatrix}
\ol{k} \\
\alpha_k  \\
\end{pmatrix}, 
u_{2k}=
\begin{pmatrix}
k \\
\alpha_k  \\
\end{pmatrix}. 
\end{align*}

We define $d_1(\pi)$ to be the doubly colored biword of length $2n$, whose columns are $u_1, \dots, u_{2n}$ arranged in this order. 
\end{dfn}

The Worley--Sagan mixed insertion can be embedded into the doubly mixed insertion using the maps
$\bbackslash$\,, $\bp$, and $d_1$ as follows: 

\begin{thm}\label{WSMDM}
For $\pi\in B_n$, we have
\begin{align*}
\bbackslash P_\text{WSM}(\pi)&=P_\text{M*}(d_1(\pi)), \\
\bp Q_\text{WSM}(\pi)&=Q_\text{M*}(d_1(\pi)). 
\end{align*}

\end{thm}
\begin{proof}
Suppose $\pi=\alpha_1\dots\alpha_n$. Let $R_1$ denote the first row of $\mathcal{Q}$. 
For each integer $k$ with $1\leq k\leq n$, define $P_k=P_\text{WSM}(\alpha_1\dots\alpha_k)$ and set 
\begin{align*}
(X_k, Y_k)=\begin{cases}
(\omega(R_1), R_1) & \text{if $\alpha_k$ is unbarred, }\\
(R_1, \omega(R_1)) & \text{if $\alpha_k$ is barred. }
\end{cases}
\end{align*}
Then, by the definition of the doubly mixed insertion, we have
\begin{align*}
P_\text{M*}(d_1(\pi))=\emptyset\xleftarrow[\!X_1\!]{}\alpha_1\xleftarrow[\!Y_1\!]{}\alpha_1\cdots\xleftarrow[\!X_n\!]{}\alpha_n\xleftarrow[\!Y_n\!]{}\alpha_n. 
\end{align*}
By Proposition~\ref{6m}, we have
\begin{align}\label{xy}
\emptyset\xleftarrow[\!X_1\!]{}\alpha_1\xleftarrow[\!Y_1\!]{}\alpha_1\cdots\xleftarrow[\!X_k\!]{}\alpha_k\xleftarrow[\!Y_k\!]{}\alpha_k=\bbackslash P_k. 
\end{align}
In particular, $P_\text{M*}(d_1(\pi))=\bbackslash P_\text{WSM}(\pi)$.

For each integer $k$ with $1\leq k\leq n$, let $\{c_k\}=D(P_k)\setminus D(P_{k-1})$ and $\beta_k=P_k(c_k)$. Then we have $\{c_k, \omega(c_k)\}=D(\bbackslash P_k)\setminus D(\bbackslash P_{k-1})$. From \eqref{xy}, it follows that, when we apply $\xleftarrow[\!X_k\!]{}\alpha_k\xleftarrow[\!Y_k\!]{}\alpha_k$ to $\bbackslash P_{k-1}$, one of the two insertions $\xleftarrow[\!X_k\!]{}\alpha_k$ or $\xleftarrow[\!Y_k\!]{}\alpha_k$ fills the cell $c_k$ with $\beta_k$, while the other fills the cell $\omega(c_k)$ with $\beta_k$. 
\begin{itemize}
\item Case 1: $\xleftarrow[\!X_k\!]{}\alpha_k$ fills $c_k$. In this case, since $\xleftarrow[\!X_k\!]{}\alpha_k$ is a dual mixed insertion, $\beta_k$ is inserted into the column containing $c_k$ if it is unbarred, and into the row containing $c_k$ if it is barred. By \eqref{xy}, in the Worley--Sagan mixed insertion of $\alpha_1\dots\alpha_k$, $\beta_k$ follows the same trajectory, so that $c_k$ is filled via dual mixed bumping. Hence, $c_k$ is filled with $\ol{k}$ in $Q_\text{WSM}(\pi)$. 
Meanwhile, in $Q_\text{M*}(d_1(\pi))$, $c_k$ is filled with $\ol{k}$ and $\omega(c_k)$ is filled with $k$.

\item Case 2: $\xleftarrow[\!Y_k\!]{}\alpha_k$ fills $c_k$. In this case, since $\xleftarrow[\!Y_k\!]{}\alpha_k$ is an ordinary mixed insertion, the roles of rows and columns are interchanged compared to Case 1. As a result, $c_k$ is filled with $k$ in $Q_\text{WSM}(\pi)$, while, in $Q_\text{M*}(d_1(\pi))$, $c_k$ is filled with $k$ and $\omega(c_k)$ is filled with $\ol {k}$. 
\end {itemize}
This yields $Q_\text{M*}(d_1(\pi))=\bp Q_\text{WSM}(\pi)$. 
\end{proof}

\section{Relationship between the Generalized Mixed Insertion and the Doubly Mixed Insertion}
This section is the analogue of \S 6 for the generalized mixed insertion. 
We will prove that the generalized mixed insertion can be embedded into the doubly mixed insertion by doubling tableaux and words. As before we use the doubling operators $\bbackslash$ and $\bp$; in this section $\bbackslash$ acts on the recording tableaux, whereas $\bp$ acts on the insertion tableaux. The argument proceeds in parallel with that of \S6. 

\begin{dfn}
Let $P$ be a shifted tableau such that all diagonal letters are unbarred
and each positive integer occurs at most once in $|P|$.
Define $\w{P}$ to be the tableau obtained from $P$ by replacing every off-diagonal letter $\alpha$ with its color-changed counterpart $\overline{\alpha}$. For an unshifted tableau $T$ that can be written as $T=\bp P$ for some shifted tableau satisfying the preceding conditions, set $\w{T}=\bp\,\w{P}$. We also write $P\sphat\,$ in place of $\w{P}$ when the expression for $P$ is long. 
\end{dfn}
\begin{ex}
For $P=\begin{ytableau}1&2&5&7\\\none&3&9\\\none&\none&10\end{ytableau}$ and
$T=\bp P=\begin{ytableau}\ol{1}&1&2&5&7\\\ol{2}&\ol{3}&3&9\\\ol{5}&\ol{9}&\ol{10}&10\\\ol{7}\end{ytableau}\,$, we have
\begin{align*}
\w{P}=\begin{ytableau}1&\ol{2}&\ol{5}&\ol{7}\\\none&3&\ol{9}\\\none&\none&10\end{ytableau}
\end{align*}
and
\begin{align*}
\w{T}=\begin{ytableau}\ol{1}&1&\ol{2}&\ol{5}&\ol{7}\\2&\ol{3}&3&\ol{9}\\5&9&\ol{10}&10\\7\end{ytableau}\,. 
\end{align*}

\end{ex}
The following definition is the analogue of Definition~\ref{wsmx} in the setting of the generalized mixed insertion: 
\begin{dfn}\label{gmx}
Let $P$ be a shifted tableau with $|P|$ having no repeated entries and with every diagonal entry unbarred, $\alpha\in\p\cup\ol{\p}$ be a letter with $|\alpha|\notin|P|$, and $X$ be a row or column of $\mathcal{O}$. We define $P\gmx\alpha$ as follows:

Choose the rule, either the ordinary or dual mixed bumping, that gives $\alpha$ the row nature if $X$ is a row and the column nature if $X$ is a column, and keep it until the procedure terminates. Then as in the generalized mixed insertion, search $X$ (left to right if $X$ is a row and top to bottom if $X$ is a column) for an entry larger than $\alpha$, and letting $\alpha'$ be the first encountered such entry and $c$ the cell containing it, modify $P$ into $P'$ by changing the entry of $c$ into $\alpha$; in case $c$ is diagonal, the color change already applies at this step, so that the entry of $c$ becomes the unbarred one of $\alpha$ and $\ol{\alpha}$, and the letter bumped out of $c$ becomes the one of $\alpha'$ and $\ol{\alpha'}$ having the column nature in the rule. Unlike in the generalized mixed insertion, $P'$ may not be a shifted tableau. If $P'$ is a shifted tableau, we continue bumping as in the generalized mixed insertion, keeping the rule and calling for the color change whenever appropriate, until the procedure terminates as an empty cell is filled, and we define $P\gmx\alpha$ to be the tableau resulting from this procedure; in this case we say that $P\gmx\alpha$ is well defined. So long as $P'$ obtained in the first step satisfies the tableau condition, the conformity with the tableau condition persists until the procedure terminates, and every diagonal entry of the resulting tableau is again unbarred.
\end{dfn}

\begin{rmk}\label{aboveleft2}
In Definition~\ref{gmx}, to determine whether $P'$ is a tableau, it suffices to compare the new entry in $c$ with its neighbor above it if $X$ is a row, and with its neighbor to the left of it if $X$ is a column, as in Remark~\ref{aboveleft}.
\end{rmk}

The next proposition is the counterpart of Proposition~\ref{6m} in \S6. 
\begin{prop}\label{7m}
Let $P$ be a shifted tableau such that all diagonal letters are unbarred
and each positive integer occurs at most once in $|P|$. Let $\alpha\in\p\cup\ol{\p}$ be a letter satisfying $|\alpha|\notin|P|$.
Suppose that $X$ is a row or column, that $P\gmx\alpha$ is well defined, and that, if $X$ is a column, $\alpha$ does not enter a diagonal cell in $P\gmx\alpha$. Then we have
\begin{align}
\bp\left(P\gmx\alpha\right)&=\bp P\ox\ol{\alpha}\x\alpha\label{7.1}\\
&=
\left(\bp \w{P}\ox\alpha\x\ol{\alpha}\right)\sphat. \label{7.2}
\end{align}
Here, $\bp P\ox\ol{\alpha}=(\bp P)\ox\ol{\alpha}$,\, $\bp P\ox\ol{\alpha}\x\alpha=\left((\bp P)\ox\ol{\alpha}\right)\x\alpha$, and so on.
\end{prop}
\begin{rmk}
The statement of Proposition~\ref{7m} includes the assertion that $\bp P\ox\ol{\alpha}$, $\bbackslash P\ox\ol{\alpha}\x\alpha$, $\bp \w{P}\ox\alpha$, and $\bp \w{P}\ox\alpha\x\ol{\alpha}$ are well defined.
See Remarks~\ref{aboveleft} and~\ref{aboveleft2} for the well-definedness of these tableaux.
\end{rmk}

\begin{proof}
As in the proof of Proposition~\ref{6m}, we proceed by induction on $N$, 
the number of times a diagonal letter is bumped during the insertion $\gmx$ into $P$. 
Note that if an empty diagonal cell is filled, we count this as a bumping of a diagonal letter. 
We begin with the base case $N=0$, which we isolate as a lemma, as before. 
\begin{lem}\label{7l}
Under the hypotheses of Proposition~\ref{7m}, set $\alpha_0=\alpha$. In the insertion $P\gmx\alpha$, suppose that for $0\le k\le n-1$ the letter 
$\alpha_k$ bumps $\alpha_{k+1}$ in the cell $c_{k+1}$. 
We put $\alpha_n=\infty$, i.e., we assume that the insertion terminates when the empty cell $c_n$ is filled with $\alpha_{n-1}$. 
We assume $\alpha_n=\infty$, that is, we assume that this insertion terminates when the empty cell $c_n$ gets filled with $\alpha_{n-1}$. 
Set $P_0=P$ and let $X$ be either a row or column with $X_0=X$. 
For $k\ge1$, let $P_k$ denote the tableau obtained immediately after $\alpha_k$ in $c_k$ is replaced by $\alpha_{k-1}$ during the insertion $P\gmx\alpha$. 
For $1\le k\le n$, let $X_k$ be the row or column that $\alpha_k$ enters after being bumped out of $c_k$.
Suppose that all of $c_1, \dots, c_n$ are off-diagonal. 

With abbreviations $A_k, B_k, \ol{A}_k, \ol{B}_k$ of sentences as defined below, the following assertions (1)--(4) hold: 
\begin{enumerate}[label=\textup{(\arabic*)}, labelsep=0.6em]
  \item In the insertion $\ox\,\overline{\alpha}$ into $\bp P$, $A_k$ holds.
  \item In the insertion $\x\,\alpha$ into $\bp P\ox\overline{\alpha}$, $B_k$ holds.
  \item In the insertion $\ox\,\alpha$ into $\bp \w{P}$, $\overline{A}_k$ holds.
  \item In the insertion $\x\,\overline{\alpha}$ into $\bp \w{P}\ox\alpha$, $\overline{B}_k$ holds.
\end{enumerate}
Abbreviations:
\begin{align*}
A_k &:
\begin{cases}
  \overline{\alpha_0} \text{ bumps } \overline{\alpha_1} \text{ in } \omega(c_1).
  & (k=1)
  \\[1ex]
  \begin{tabular}{@{}l@{}}
    After being bumped out of $\omega(c_{k-1})$, \\
    $\overline{\alpha_{k-1}}$ enters $\omega(X_{k-1})$ and bumps $\overline{\alpha_k}$ in $\omega(c_k)$.
  \end{tabular}
  & (2\le k\le n)
\end{cases} \\[3ex]
B_k &:
\begin{cases}
  \alpha_0 \text{ bumps } \alpha_1 \text{ in } c_1.
  & (k=1)
  \\[1ex]
  \begin{tabular}{@{}l@{}}
    After being bumped out of $c_{k-1}$, \\
    $\alpha_{k-1}$ enters $X_{k-1}$ and bumps $\alpha_k$ in $c_k$.
  \end{tabular}
  & (2\le k\le n)
\end{cases} \\[3ex]
\overline{A}_k &:
\begin{cases}
  \alpha_0 \text{ bumps } \alpha_1 \text{ in } \omega(c_1).
  & (k=1)
  \\[1ex]
  \begin{tabular}{@{}l@{}}
    After being bumped out of $\omega(c_{k-1})$, \\
    $\alpha_{k-1}$ enters $\omega(X_{k-1})$ and bumps $\alpha_k$ in $\omega(c_k)$.
  \end{tabular}
  & (2\le k\le n)
\end{cases} \\[3ex]
\overline{B}_k &:
\begin{cases}
  \overline{\alpha_0} \text{ bumps } \overline{\alpha_1} \text{ in } c_1.
  & (k=1)
  \\[1ex]
  \begin{tabular}{@{}l@{}}
    After being bumped out of $c_{k-1}$, \\
    $\overline{\alpha_{k-1}}$ enters $X_{k-1}$ and bumps $\overline{\alpha_k}$ in $c_k$.
  \end{tabular}
  & (2\le k\le n)
\end{cases}
\end{align*}
\end{lem}

\begin{proof}[\textrm{Proof of Lemma~\ref{7l}}]
(1) We proceed by induction on $k$. Suppose that all of $A_1, \dots, A_{k-1}$ hold. In case where $k=1$, we suppose nothing. Let $T_0=\bp P$ and $T_{k-1}$ be the tableau which appears right after the replacement of $\ol{\alpha_{k-1}}$ in $\omega(c_{k-1})$ with $\ol{\alpha_{k-2}}$ in the process of $\ox\ol{\alpha}$ applied to $\bp P $ if $k\ge 2$. $\ol{\alpha_{k-1}}$ bumped out of $\omega(c_{k-1})$ enters $\omega(X_k)$. 

If $X_{k-1}$ is a column, it is easy to see $\ol{\alpha_{k-1}}$ bumps $\ol{\alpha_k}$ in $\omega(c_k)$. This is because the arrangement of letters in $X_{k-1}$ of $P_{k-1}$ with colors switched is the same as the arrangement of letters in the below diagonal part of $\omega(X_{k-1})$ of $T_{k-1}$. If $X_{k-1}$ is a row, $\ol{\alpha_{k-1}}$ bumps $\ol{\alpha_k}$ in $\omega(c_k)$ as expected. To see this, we have $T_{k-1}(\Delta\cap\omega(X_{k-1}))=P(\Delta\cap\omega(X_{k-1}))<P(c_{k-1})=\alpha_{k-1}$. The first inequality holds because $\Delta\cap\omega(X_{k-1})$ and a off-diagonal cell $c_{k-1}$ are located in the same row of $P$ (the one above $X_{k-1}$). By this, 
$\ol{\alpha_{k-1}}$ never bumps the above diagonal cells in $\omega(X_{k-1})$ of $T_{k-1}$. 
Because the arrangement of letters in the below diagonal part of $\omega(X_{k-1})$ of $T_{k-1}$ is the same as the arrangement of letters in $X_{k-1}$ of $P_{k-1}$ with colors switched, $\ol{\alpha_{k-1}}$ bumps $\ol{\alpha_k}$ in $\omega(c_k)$. 

(2) As in the proof of (1), we use induction on $k$. We assume that all of $B_1, \dots, B_{k-1}$ hold. In case where $k=1$, we suppose nothing. Let $U_0=\bp P\ox\ol{\alpha}$ and $U_{k-1}$ be the tableau which appears right after the replacement of $\alpha_{k-1}$ in $c_{k-1}$ with $\alpha_{k-2}$ in the process of 
$\x\alpha$ applied to $\bp P\ox\ol{\alpha}$ if $k \ge 2$. $\alpha_{k-1}$ bumped out of $c_{k-1}$ enters $X_{k-1}$. 

If $X_{k-1}$ is a column, $\alpha_{k-1}$ bumps $\alpha_k$ in $c_k$ because the arrangement of letters in the above diagonal part of $X_{k-1}$ of $U_{k-1}$ is the same as the arrangement of letters in $X_{k-1}$ of $P_{k-1}$. Let us think about the case where $X_{k-1}$ is a row. By (1) and $B_1, \dots, B_{k-1}$, $U_{k-1}(\Delta\cap X_{k-1})=P_{k-1}(\Delta\cap X_{k-1})<\alpha_{k-1}$. The last inequality follows from the fact that $\alpha_{k-1}$ does not bump the letter in $\Delta\cap X_{k-1}$ in $\gmx\alpha$ applied to $P$. Hence, $\alpha_{k-1}$ never bumps the below diagonal cells in $X_{k-1}$. The arrangement of letters in the above diagonal part of $X_{k-1}$ of $U_{k-1}$ is the same as the arrangement of letters in $X_{k-1}$ of $P_{k-1}$. So, $\alpha_{k-1}$ bumps $\alpha_k$ in $c_k$. 

(3) We can show our statements in the same way as (1). The only difference from the proof of (1) is that $c_k$ is filled with $\ol{\alpha_k}$ and $\omega(c_k)$ is filled with $\alpha_k$. 

(4) As is the case with (3), our proof is similar to the proof of (2). Howerver, as before, we should note that $c_k$ is filled with $\ol{\alpha_k}$ and $\omega(c_k)$ is filled with $\alpha_k$ in this case. 
\end{proof}

We now return to the proof of Proposition~\ref{7m}. The proof is divided into parts \textup{I}--\textup{III}.

\bigskip

\noindent
\textbf{Part \textup{I}.} \textit{Proof of \eqref{7.1} and \eqref{7.2} in the case $N=0$.}

By Lemma~\ref{7l}\textup{(1)},
\[
  \bp P \ox \overline{\alpha}=
  P \cup \overline{\omega}\left(P\gmx\alpha\right).
\]
Applying Lemma~\ref{7l}\textup{(2)} then gives
\[
  \bp P \ox \overline{\alpha} \x \alpha=
  \left(P\gmx\alpha\right)\cup \overline{\omega}\left(P\gmx\alpha\right),
\]
which is equivalent to \eqref{7.1}. 
Similarly, by Lemma~\ref{7l}\textup{(3)},
\[\bp \w{P}\ox\alpha=\w{P}\cup\ol{\omega}\left(\left(P\gmx\alpha\right)\sphat\ \ \right), \]
and by Lemma~\ref{7l}\textup{(2)}, 
\[\bp \w{P}\ox\alpha\x\ol{\alpha}=\left(P\gmx\alpha\right)\sphat\cup \ol{\omega}\left(\left(P\gmx\alpha\right)\sphat\ \ \right). \]
This is equivalent to \eqref{7.2}. 
\hfill $\diamondsuit$

\bigskip

Next, assume that the number of diagonal bumpings in the insertion $P\gmx\alpha$ is $N\ge 1$. 
Set $\alpha_0=\alpha$, and suppose that, for $0\leq k\leq n-1$, the letter $\alpha_k$ bumps $\alpha_{k+1}$ in the cell $c_{k+1}$ during this insertion. Assume that $c_1,\dots,c_{n-1}$ lie off the diagonal, while $c_n$ is diagonal. For $1\le k\le n$, let $X_k$ denote the row or column that $\alpha_k$ enters after being bumped from $c_k$. 

Write $c_k=(i_k,j_k)$ for $1\le k\le n$, and define
\[
D_0 = \bigcup_{r=1}^{n-1} S_{\mathcal{O}}(c_r), 
\qquad
D_1 = D_0 \cup S_{\mathcal{O}}(c_n),
\qquad
D_2 = \mathcal{O}\setminus D_1 .
\]
Recall that, for $c=(i,j)$, we defined
\[
S_{\mathcal{O}}(c)\;=\;\{\, (i',j')\in \mathcal{O}\mid i'\le i,\; j'\le j \,\}
\]
in the proof of Proposition~\ref{6m}.

Note that $X_{n-1}$ is a row. Indeed, if $X_{n-1}$ were a column, then the cell immediately to the right of $c_{n-1}$ would lie in $X_{n-1}$ and its entry would be greater than $\alpha_{n-1}$. Hence $\alpha_{n-1}$ would bump a letter in a cell of $X_{n-1}$ strictly above the diagonal, contradicting the assumption that $c_n$ is diagonal. 
$\alpha_n$ is unbarred, since diagonal letters of $P$ are unbarred by hypothesis. 

\bigskip

\noindent
\textbf{Part \textup{II}.} \textit{Proof of \eqref{7.1} in the general case.}

\begin{figure}
\centering
\begin{minipage}{0.45\linewidth}
\begin{tikzpicture}[x=1.2cm, y=1.2cm, yscale=-1]
\fill[black!10] (2.5,0.5) rectangle (3.5,6.5);
\fill[blue!10, opacity=0.5] (0.5,3.5) rectangle (6.5,4.5);
\foreach \x in {1.5, ..., 6.5} \draw[line width=1.5pt] (\x-1,\x)--(\x,\x);
\foreach \x in {0.5, ..., 5.5} \draw[line width=1.5pt] (\x,\x+1)--(\x,\x);
\draw[line width=1.5pt](2.5,2.5)--(3.5,2.5);
\draw[line width=1.5pt](3.5,2.5)--(3.5,3.5);
\draw[line width=1.5pt](2.5,3.5)--(2.5,4.5);
\draw[line width=1.5pt](2.5,4.5)--(4.5,4.5);
\draw[line width=1.5pt](3.5,3.5)--(4.5,3.5)--(4.5,4.5);
\draw[line width=1.5pt](1.5,5.5)--(1.5,6.5)--(2.5,6.5)--(2.5,5.5)--(1.5,5.5);
\draw[line width=1.5pt](5.5,2.5)--(5.5,3.5)--(6.5,3.5)--(6.5,2.5)--(5.5,2.5);
\draw(2.7,2.7)node[font=\scriptsize]{};
\draw(3.7,3.7)node[font=\scriptsize]{$c_n$};
\draw(2.7,2.7)node[font=\scriptsize]{$d$};
\draw(5.8,2.7)node[font=\scriptsize]{$c_{n-1}$};
\draw(2.85,3.65)node[font=\scriptsize]{$\omega(c_n)$};
\draw(1.95,5.65)node[font=\scriptsize]{$\omega(c_{n-1})$};
\draw(3,4)node[font=\normalsize]{$\ol{\alpha_n}$};
\draw(6,3)node[font=\normalsize]{$\alpha_{n-1}$};
\draw(2,6)node[font=\normalsize]{$\ol{\alpha_{n-2}}$};
\draw(3,6)node[font=\scriptsize]{$\omega(X_{n-1})$};
\draw(6,4)node[font=\scriptsize]{$X_{n-1}$};
\draw(3,1.5)node[font=\normalsize]{$\downarrow$};
\draw(3,1)node[font=\normalsize]{$\ol{\alpha_{n-1}}$};
\draw(4,0.5)node[font=\normalsize]{\phantom{(Case 2)}};
\end{tikzpicture}
\caption{\label{7tn-1}$T_{n-1}$ and $\ol{\alpha_{n-1}}$ bumped from $\omega(c_{n-1})$.}
\end{minipage}
\hspace{10pt}
\begin{minipage}{0.45\linewidth}
\begin{tikzpicture}[x=1.2cm, y=1.2cm, yscale=-1]
\fill[black!50] (4.5,0.5) rectangle (5.5,6.5);
\fill[black!30] (3.5,0.5) rectangle (4.5,6.5);
\fill[blue!10, opacity=0.5] (0.5,3.5) rectangle (6.5,4.5);
\fill[blue!30, opacity=0.5] (0.5,4.5) rectangle (6.5,5.5);
\foreach \x in {1.5, ..., 6.5} \draw[line width=1.5pt] (\x-1,\x)--(\x,\x);
\foreach \x in {0.5, ..., 5.5} \draw[line width=1.5pt] (\x,\x+1)--(\x,\x);

\draw[line width=1.5pt](2.5,3.5)--(2.5,4.5);
\draw[line width=1.5pt](2.5,4.5)--(4.5,4.5);
\draw[line width=1.5pt](3.5,3.5)--(4.5,3.5)--(4.5,4.5);
\draw(2.7,2.7)node[font=\scriptsize]{};
\draw(3.7,3.7)node[font=\scriptsize]{$c_n$};
\draw(2.85,3.65)node[font=\scriptsize]{$\omega(c_n)$};
\draw(3,4)node[font=\normalsize]{$\ol{\alpha_{n-1}}$};
\draw(4,4)node[font=\normalsize]{$\alpha_n$};
\draw(5,6)node[font=\scriptsize]{$X_n$};
\draw(6,4)node[font=\scriptsize]{$X_{n-1}$};
\draw(6,5)node[font=\scriptsize]{$\omega(X_n)$};
\draw(4,6)node[font=\scriptsize]{$Y$};
\draw(4,1.5)node[font=\normalsize]{$\downarrow$};
\draw(4,1)node[font=\normalsize]{$\ol{\alpha_n}$};
\draw(4,0.5)node[font=\normalsize]{(Case 2)};
\draw(1,5)node[font=\normalsize]{$\ol{\alpha_n}$};
\draw(1.5,5)node[font=\normalsize]{$\rightarrow$};
\draw(1,4.5)node[font=\normalsize]{(Case 1)};
\end{tikzpicture}
\caption{\label{7tn}$T_{n}$ and $\ol{\alpha_n}$ bumped from $\omega(c_n)$. }
\end{minipage}
\caption*{}
\end{figure}

First we prove \eqref{7.1}. When $\ox\ol{\alpha}$ is applied to $\bp P$, 
$A_k$ holds for $1\le k\le n-1$. This follows from Lemma~\ref{7l}(1), specialized to the insertion $\ox\ol{\alpha}$ applied to $\bp \left(P|_{D_0\setminus\{c_{n-1}\}}\right)$. 
After being bumped from $\omega(c_{n-1})$, $\overline{\alpha_{n-1}}$ bumps 
$\overline{\alpha_n}$ in $\omega(c_n)$. To see this, let $T_{n-1}$ denote the tableau
immediately after the bumping of $\overline{\alpha_{n-1}}$ in $\omega(c_{n-1})$, and set $d=\Delta\cap\omega(X_{n-1})$. See Figure~\ref{7tn-1}. 
Then $\overline{\alpha_{n-1}}$ enters $\omega(X_{n-1})$, and
$T_{n-1}(d)= P(d)$ since no above-diagonal letters have been bumped at this stage. Because $X_{n-1}$ is a row,
both $d$ and the off-diagonal cell $c_{n-1}$ lie in the same row
(the one immediately above $X_{n-1}$). Hence $T_{n-1}(d)=P(d)<P(c_{n-1})=\alpha_{n-1}$. 
Since $T_{n-1}(d)=P(d)$ is unbarred, it is smaller than $\overline{\alpha_{n-1}}$. 
Therefore $\overline{\alpha_{n-1}}$ does not bump any above-diagonal cell of $\omega(X_{n-1})$;
instead, it bumps $\overline{\alpha_n}$ in $\omega(c_n)$. The letter $\overline{\alpha_n}$
then proceeds to the next row or column according to the insertion type of $P\gmx\alpha$.
We treat the two cases separately. 

Let $T_n$ denote the tableau obtained immediately after $\overline{\alpha_n}$ is bumped from the cell $\omega(c_n)$. See Figure~\ref{7tn}. 
Note that $\alpha_n$ is unbarred since it is a diagonal letter of $P$. 
\begin{itemize}
\item Case 1: $\ox\ol{\alpha}$ is a dual mixed bumping. 
In this case, $\ol{\alpha_n}$ bumped out of $\omega(c_n)$ enters the next row, $\omega(X_n)$. 
Note that the letters in $D_0$ remain unchanged for the remainder of the insertion $\ox\,\ol{\alpha}$. 
Indeed, at the moment $\ol{\alpha_n}$ in $\omega(c_n)$ is bumped, no letter in $D_0$ exceeds $\ol{\alpha_n}$, so no subsequent bump can occur in $D_0$.

\item Case 2: $\ox\ol{\alpha}$ is an ordinary mixed bumping. 
In this case, after being bumped from the cell $\omega(c_n)$, $\overline{\alpha_n}$ enters the column immediately to the right of $\omega(X_{n-1})$; denote this column by $Y$. Then $\Delta\cap Y=\{c_n\}$ and $T_n(c_n)=\alpha_n$. Since $T_n\big|_{\mathcal{O}}=P$ and each positive integer occurs at most once in $|P|$, the cell of $T_n$ directly above $c_n$ contains a letter strictly less than $\overline{\alpha_n}$. Hence $\overline{\alpha_n}$ bumps $\alpha_n$ in $c_n$, and the letter $\alpha_n$ then enters the next row $\omega(X_n)$. As in Case~1, the letters in $D_0$ remain unchanged for the remainder of the insertion $\ox\,\overline{\alpha}$.

\end{itemize}

Summarizing, we have 
\begin{align}\label{gm1-1}
\bp P\ox\ol{\alpha}=P|_{D_1\setminus\{c_n\}}
\cup c_n\left(\ol{\beta}\right)\cup\ol{\omega}\left(\left.\left(P\gmx\alpha\right)\right|_{D_1}\right)\cup\left((\bp P)|_{D_2\cup\omega(D_2)}\xleftarrow[\!\!\!\!\!\omega(X_n)\!\!\!\!\!]{}\,\beta\right). 
\end{align}
Here, $c_n\left(\ol{\beta}\right)$ denotes the cell $c_n$ with entry $\ol{\beta}$, and
\begin{align}\label{betabeta}
\beta=
\begin{cases}
\ol{\alpha_n} & \text{in Case 1}, \\
\alpha_n & \text{in Case 2}.
\end{cases}
\end{align}

\begin{figure}
\centering
\begin{minipage}{0.45\linewidth}
\begin{tikzpicture}[x=1.2cm, y=1.2cm, yscale=-1]
\fill[blue!10, opacity=0.5] (0.5,3.5) rectangle (6.5,4.5);
\foreach \x in {1.5, ..., 6.5} \draw[line width=1.5pt] (\x-1,\x)--(\x,\x);
\foreach \x in {0.5, ..., 5.5} \draw[line width=1.5pt] (\x,\x+1)--(\x,\x);
\draw[line width=1.5pt](2.5,3.5)--(2.5,4.5)--(4.5,4.5);
\draw[line width=1.5pt](5.5,2.5)--(6.5,2.5)--(6.5,3.5)--(5.5,3.5)--(5.5,2.5);
\draw[line width=1.5pt](3.5,3.5)--(4.5,3.5)--(4.5,4.5);
\draw(3.7,3.7)node[font=\scriptsize]{$c_n$};
\draw(2.85,3.65)node[font=\scriptsize]{$\omega(c_n)$};
\draw(5.8,2.7)node[font=\scriptsize]{$c_{n-1}$};
\draw(3,4)node[font=\normalsize]{$\ol{\alpha_{n-1}}$};
\draw(4,4)node[font=\normalsize]{$\overline{\beta}$};
\draw(6,3)node[font=\normalsize]{$\alpha_{n-2}$};
\draw(6,4)node[font=\scriptsize]{$X_{n-1}$};
\draw(1.6,4)node[font=\normalsize]{$\rightarrow$};
\draw(1,4)node[font=\normalsize]{$\alpha_{n-1}$};
\draw(4,0.5)node[font=\normalsize]{\phantom{(Case 2)}};
\end{tikzpicture}
\caption{\label{7un-1}$U_{n-1}$ and $\alpha_{n-1}$ bumped from $c_{n-1}$. }
\end{minipage}
\hspace{10pt}
\begin{minipage}{0.45\linewidth}
\begin{tikzpicture}[x=1.2cm, y=1.2cm, yscale=-1]
\fill[black!50] (4.5,0.5) rectangle (5.5,6.5);
\foreach \x in {1.5, ..., 6.5} \draw[line width=1.5pt] (\x-1,\x)--(\x,\x);
\foreach \x in {0.5, ..., 5.5} \draw[line width=1.5pt] (\x,\x+1)--(\x,\x);
\draw[line width=1.5pt](2.5,3.5)--(2.5,4.5);
\draw[line width=1.5pt](2.5,4.5)--(4.5,4.5);
\draw[line width=1.5pt](3.5,3.5)--(4.5,3.5)--(4.5,4.5);
\draw(3.7,3.7)node[font=\scriptsize]{$c_n$};
\draw(2.85,3.65)node[font=\scriptsize]{$\omega(c_n)$};
\draw(3,4)node[font=\normalsize]{$\ol{\alpha_{n-1}}$};
\draw(4,4)node[font=\normalsize]{$\alpha_{n-1}$};
\draw(5,6)node[font=\scriptsize]{$X_n$};
\draw(5,1.5)node[font=\normalsize]{$\downarrow$};
\draw(5,1)node[font=\normalsize]{$\ol{\beta}$};
\end{tikzpicture}
\caption{\label{7un}$U_{n}$ and $\ol{\beta}$ bumped from $c_n$. }
\end{minipage}
\caption*{}
\end{figure}

Next, we apply $\x\alpha$ to $\bp P\ox\overline{\alpha}$. As observed above, the letters in $D_0$ remain unchanged throughout the insertion $\ox\,\overline{\alpha}$ into $\bp P$; hence, by \eqref{gm1-1}, 
\begin{align}\label{gm1-2}
  \left.\left(\bp P\ox\overline{\alpha}\right)\right|_{D_0\cup\omega(D_0)\setminus\{c_{n-1}\}}
  = \bp\left(P\big|_{D_0\setminus\{c_{n-1}\}}\right)\ox\overline{\alpha}.
\end{align}
By \eqref{gm1-2} and Lemma~\ref{7l}\textup{(2)}, when $\x\,\alpha$ is applied to $\bp P\ox\overline{\alpha}$, $B_k$ holds for $0\le k\le n-1$. Let $U_{n-1}$ be the tableau right after the replacement of $\alpha_{n-1}$ in $c_{n-1}$. See Figure~\ref{7un-1}. 
After being bumped from $c_{n-1}$, $\alpha_{n-1}$ enters $X_{n-1}$. 
This is because the insertions $\x\alpha$ and $\gmx\alpha$ are of the same bumping type: if one is of ordinary mixed type, then so is the other; if one is of dual mixed type, then so is the other. 
By \eqref{gm1-1}, we have that $\omega(c_n)$ contains $\overline{\alpha_{n-1}}$, whereas $c_n$ contains $\overline{\beta}$ at this stage. Hence $\alpha_{n-1}$ bumps $\overline{\beta}$ in $c_n$. 

Let $U_n$ be the tableau right after the replacement of $\ol{\beta}$ in $c_n$. See Figure~\ref{7un}. 
Then, $\overline{\beta}$ enters $X_n$. This again follows because the insertions $\x\alpha$ and $\gmx\alpha$ are of the same bumping type. 

Thus we obtain
\begin{align}\label{gm1-3}
\bp P\ox\ol{\alpha}\x\alpha=\bp\left(\left.\left(P\gmx\alpha\right)\right|_{D_1}\right)\cup\left((\bp P)|_{D_2\cup\omega(D_2)}\xleftarrow[\!\!\!\!\!\omega(X_n)\!\!\!\!\!]{}\,\beta\xleftarrow[X_n]{}\ol{\beta}\right). 
\end{align}
By definition of generalized mixed insertion, we have
\begin{align}\label{gm1-4}
\left.\left(P\gmx\alpha\right)\right|_{D_2}=P|_{D_2}\xleftarrow[X_n]{\!\!\!\!\!\text{GM}\!\!\!\!\!}\ol{\beta}. 
\end{align}
This follows from \eqref{betabeta} and from the fact that $\gmx\alpha$ is a dual [resp. an ordinary] mixed bumping in Case~1 [resp. Case~2]. 
In $\xleftarrow[X_n]{\!\!\!\!\!\text{GM}\!\!\!\!\!}\ol{\beta}$ applied to $P|_{D_2}$, diagonal letters are bumped for $N-1$ times. So, by the induction hypothesis, we have 
\begin{align}\label{gm1-5}
\bp \left(P|_{D_2}\xleftarrow[X_n]{\!\!\!\!\!\text{GM}\!\!\!\!\!}\ol{\beta}\right)=\bp (P|_{D_2})\xleftarrow[\!\!\!\!\!\omega(X_n)\!\!\!\!\!]{}\,\beta\xleftarrow[X_n]{}\ol{\beta}. 
\end{align}
\eqref{gm1-4} and \eqref{gm1-5} yield
\begin{align}\label{gm1-6}
\bp \left(\left.\left(P\gmx\alpha\right)\right|_{D_2}\right)=\bp (P|_{D_2})\xleftarrow[\!\!\!\!\!\omega(X_n)\!\!\!\!\!]{}\,\beta\xleftarrow[X_n]{}\ol{\beta}. 
\end{align}
The right-hand side of \eqref{gm1-6} agrees with the second term in the union on the right-hand side of \eqref{gm1-3}. 
Therefore, \eqref{7.1} follows from \eqref{gm1-3} and \eqref{gm1-6}. 
\hfill $\diamondsuit$

\bigskip

\noindent
\textbf{Part \textup{III}.} \textit{Proof of \eqref{7.2} in the general case.}

\begin{figure}
\centering
\begin{minipage}{0.45\linewidth}
\begin{tikzpicture}[x=1.2cm, y=1.2cm, yscale=-1]
\fill[black!10] (2.5,0.5) rectangle (3.5,6.5);
\foreach \x in {1.5, ..., 6.5} \draw[line width=1.5pt] (\x-1,\x)--(\x,\x);
\foreach \x in {0.5, ..., 5.5} \draw[line width=1.5pt] (\x,\x+1)--(\x,\x);
\draw[line width=1.5pt](2.5,3.5)--(2.5,4.5);
\draw[line width=1.5pt](2.5,4.5)--(4.5,4.5);
\draw[line width=1.5pt](3.5,3.5)--(4.5,3.5)--(4.5,4.5);
\draw[line width=1.5pt](1.5,5.5)--(2.5,5.5)--(2.5,6.5)--(1.5,6.5)--(1.5,5.5);
\draw[line width=1.5pt](5.5,2.5)--(6.5,2.5)--(6.5,3.5)--(5.5,3.5)--(5.5,2.5);
\draw[line width=1.5pt](2.5,2.5)--(3.5,2.5)--(3.5,3.5); 
\draw(3.7,3.7)node[font=\scriptsize]{$c_n$};
\draw(2.7,2.7)node[font=\scriptsize]{$d$};
\draw(2.85,3.65)node[font=\scriptsize]{$\omega(c_n)$};
\draw(5.8,2.65)node[font=\scriptsize]{$c_{n-1}$};
\draw(4,2.7)node[font=\scriptsize]{};
\draw(2,3.7)node[font=\scriptsize]{};
\draw(1.95,5.7)node[font=\scriptsize]{$\omega(c_{n-1})$};
\draw(3,4)node[font=\normalsize]{$\ol{\alpha_n}$};
\draw(4,4)node[font=\normalsize]{$\alpha_n$};
\draw(6,3)node[font=\normalsize]{$\ol{\alpha_{n-1}}$};
\draw(3,6)node[font=\scriptsize]{$\omega(X_{n-1})$};
\draw(2,6)node[font=\normalsize]{$\alpha_{n-2}$};
\draw(3,1.5)node[font=\normalsize]{$\downarrow$};
\draw(3,1)node[font=\normalsize]{$\alpha_{n-1}$};
\draw(4,0.5)node[font=\normalsize]{\phantom{(Case 2)}};
\end{tikzpicture}
\caption{\label{7vn-1}$V_{n-1}$ and $\alpha_{n-1}$ bumped from $\omega(c_{n-1})$. }
\end{minipage}
\hspace{10pt}
\begin{minipage}{0.45\linewidth}
\begin{tikzpicture}[x=1.2cm, y=1.2cm, yscale=-1]
\fill[black!30] (3.5,0.5) rectangle (4.5,6.5);
\fill[blue!10, opacity=0.5] (0.5,3.5) rectangle (6.5,4.5);
\fill[blue!30, opacity=0.5] (0.5,4.5) rectangle (6.5,5.5);
\foreach \x in {1.5, ..., 6.5} \draw[line width=1.5pt] (\x-1,\x)--(\x,\x);
\foreach \x in {0.5, ..., 5.5} \draw[line width=1.5pt] (\x,\x+1)--(\x,\x);
\draw[line width=1.5pt](2.5,3.5)--(2.5,4.5);
\draw[line width=1.5pt](2.5,4.5)--(4.5,4.5);
\draw[line width=1.5pt](3.5,3.5)--(4.5,3.5)--(4.5,4.5);
\draw(3.7,3.7)node[font=\scriptsize]{$c_n$};
\draw(2.85,3.65)node[font=\scriptsize]{$\omega(c_n)$};
\draw(3,4)node[font=\normalsize]{$\alpha_{n-1}$};
\draw(4,4)node[font=\normalsize]{$\alpha_n$};
\draw(6,5)node[font=\scriptsize]{$\omega(X_n)$};
\draw(4,6)node[font=\scriptsize]{$Y$};
\draw(6,4)node[font=\scriptsize]{$X_{n-1}$};
\draw(1.3,5)node[font=\normalsize]{$\rightarrow$};
\draw(0.7,5)node[font=\normalsize]{$\ol{\alpha_n}$};
\draw(4,1.5)node[font=\normalsize]{$\downarrow$};
\draw(4,1)node[font=\normalsize]{$\ol{\alpha_n}$};
\draw(4,0.5)node[font=\normalsize]{(Case 1)};
\draw(0.5,4.5)node[font=\normalsize]{(Case 2)};
\draw(4,0.5)node[font=\normalsize]{\phantom{(Case 2)}};
\end{tikzpicture}
\caption{\label{7vn}$V_{n}$ and $\ol{\alpha_n}$ bumped from $\omega(c_n)$. }
\end{minipage}
\caption*{}
\end{figure}

Next, we will prove \eqref{7.2}. 
When $\ox\alpha$ is applied to $\bp\w{P}$, $\overline{A}_k$ holds for $1\le k\le n-1$.
This follows from Lemma~\ref{7l}\textup{(3)} applied to the insertion $\ox\alpha$ on
\[
  \bp\left(\left(P\big|_{D_0\setminus\{c_{n-1}\}}\right)\sphat\,\right)
  = \left.\left(\bp\w{P}\right)\right|_{(D_0\setminus\{c_{n-1}\})\cup \omega\left(D_0\setminus\{c_{n-1}\}\right)}.
\]
Let $V_{n-1}$ denote the tableau obtained immediately after $\alpha_{n-1}$ is bumped from the cell $\omega(c_{n-1})$, and set $d=\Delta\cap\omega(X_{n-1})$. See Figure~\ref{7vn-1}. 
After being bumped from $\omega(c_{n-1})$, $\alpha_{n-1}$ enters $\omega(X_{n-1})$ and bumps $\overline{\alpha_n}$ in $\omega(c_n)$. 
This follows from $V_{n-1}(d)<V_{n-1}(c_{n-1})=\ol{\alpha_{n-1}}$, since $d$ and the off-diagonal cell $c_{n-1}$ lie in the same row. 

Let $V_n$ denote the tableau obtained immediately after $\ol{\alpha_n}$ is bumped from $\omega(c_n)$. See Figure~\ref{7vn}. 
The letter $\ol{\alpha_n}$ bumped from $\omega(c_n)$ then enters the next row or column according to the insertion type of $\ox\alpha$. 
As in Part~II, we split into two cases, Case~1 and Case~2. Since the insertion $\ox\,\alpha$ is of a different bumping type from $\ox\,\overline{\alpha}$, the cases are as follows: 
\begin{itemize}
\item Case 1: $\ox\alpha$ uses ordinary mixed bumping. 
In this case, any bumped barred letter enters the next column. Hence, after being bumped from $\omega(c_n)$, $\overline{\alpha_n}$ enters the column immediately to the right of $\omega(X_{n-1})$; as before, denote this column by $Y$. Since $\Delta\cap Y=\left\{c_n\right\}$, for any above-diagonal cell $e$ in $Y$, we have $\left|V_n(e)\right|<\left|V_n(c_n)\right|=\alpha_n$. 
Therefore $\overline{\alpha_n}$ does not bump any above-diagonal cell of $Y$; instead, it bumps $\alpha_n$ in $c_n$. The letter $\alpha_n$ bumped from $c_n$ then enters the row next to $X_{n-1}$, namely $\omega(X_n)$.
\item Case 2: $\ox\alpha$ uses dual mixed bumping. In this case, any bumped barred letter enters the next row. Hence, after being bumped from $\omega(c_n)$, the letter $\overline{\alpha_n}$ enters the row next to $X_{n-1}$, namely $\omega(X_n)$. 
\end{itemize}

Therefore, 
\begin{align}\label{gm2-1}
\bp \w{P}\ox\alpha=\left.\w{P}\right|_{D_1\setminus\{c_n\}}\cup c_n(\beta)&\cup\ol{\omega}\left(\left.\left(P\gmx\alpha\right)\sphat\,\,\right|_{D_1\setminus\{c_n\}}\right)\cup \omega(c_n)(\alpha_{n-1})\notag\\&\cup\left(\left.\left(\bp\w{P}\right)\right|_{D_2\cup\omega(D_2)}\ \xleftarrow[\!\!\!\!\!\omega(X_n)\!\!\!\!\!]{}\,\ol{\beta}\right). 
\end{align}
Here, $c_n(\beta)$ denotes the cell $c_n$ with entry $\beta$, $\omega(c_n)(\alpha_{n-1})$ denotes the cell $\omega(c_n)$ with entry $\alpha_{n-1}$, and
\begin{align}\label{betabeta2}
\beta=
\begin{cases}
\ol{\alpha_n} & \text{in Case 1}, \\
\alpha_n & \text{in Case 2}.
\end{cases}
\end{align}
This definition of $\beta$ is the same as the one in \eqref{betabeta}. 
\begin{figure}
\centering
\begin{minipage}{0.45\linewidth}
\begin{tikzpicture}[x=1.2cm, y=1.2cm, yscale=-1]
\fill[blue!10, opacity=0.5] (0.5,3.5) rectangle (6.5,4.5);
\foreach \x in {1.5, ..., 6.5} \draw[line width=1.5pt] (\x-1,\x)--(\x,\x);
\foreach \x in {0.5, ..., 5.5} \draw[line width=1.5pt] (\x,\x+1)--(\x,\x);
\draw[line width=1.5pt](2.5,3.5)--(2.5,4.5)--(3.5,4.5);
\draw[line width=1.5pt](3.5,3.5)--(4.5,3.5)--(4.5,4.5);
\draw[line width=1.5pt](5.5,2.5)--(6.5,2.5)--(6.5,3.5)--(5.5,3.5)--(5.5,2.5);
\draw[line width=1.5pt](2.5,3.5)--(1.5,3.5)--(1.5,4.5)--(2.5,4.5);
\draw[line width=1.5pt](3.5,3.5)--(3.5,2.5)--(4.5,2.5)--(4.5,3.5);
\draw(2.85,3.7)node[font=\scriptsize]{$\omega(c_n)$};
\draw(1.8,3.7)node[font=\scriptsize]{$\omega(e)$};
\draw(3.7,3.7)node[font=\scriptsize]{$c_n$};
\draw(3.65,2.7)node[font=\scriptsize]{$e$};
\draw(3,4)node[font=\normalsize]{$\alpha_{n-1}$};
\draw(5.8,2.7)node[font=\scriptsize]{$c_{n-1}$};
\draw(6,3)node[font=\normalsize]{$\ol{\alpha_{n-2}}$};
\draw(4,4)node[font=\normalsize]{$\beta$};
\draw(6,4)node[font=\scriptsize]{$X_{n-1}$};
\draw(1,4)node[font=\normalsize]{$\ol{\alpha_{n-1}}$};
\draw(1.5,4)node[font=\normalsize]{$\rightarrow$};
\end{tikzpicture}
\caption{\label{7wn-1}$W_{n-1}$ and $\ol{\alpha_{n-1}}$ bumped from $c_{n-1}$. }
\end{minipage}
\hspace{10pt}
\begin{minipage}{0.45\linewidth}
\begin{tikzpicture}[x=1.2cm, y=1.2cm, yscale=-1]
\fill[black!30] (3.5,0.5) rectangle (4.5,6.5);
\fill[black!50] (4.5,0.5) rectangle (5.5,6.5);
\foreach \x in {1.5, ..., 6.5} \draw[line width=1.5pt] (\x-1,\x)--(\x,\x);
\foreach \x in {0.5, ..., 5.5} \draw[line width=1.5pt] (\x,\x+1)--(\x,\x);
\draw[line width=1.5pt](2.5,3.5)--(2.5,4.5);
\draw[line width=1.5pt](2.5,4.5)--(4.5,4.5);
\draw[line width=1.5pt](3.5,3.5)--(4.5,3.5)--(4.5,4.5);
\draw[line width=1.5pt](3.5,3.5)--(3.5,2.5)--(4.5,2.5)--(4.5,3.5);
\draw[line width=1.5pt](5.5,2.5)--(6.5,2.5)--(6.5,3.5)--(5.5,3.5)--(5.5,2.5);
\draw(3.7,3.7)node[font=\scriptsize]{$c_n$};
\draw(2.85,3.7)node[font=\scriptsize]{$\omega(c_n)$};
\draw(3,4)node[font=\normalsize]{$\ol{\alpha_{n-1}}$};
\draw(4,4)node[font=\normalsize]{$\beta$};
\draw(4,1.5)node[font=\normalsize]{$\downarrow$};
\draw(4,1)node[font=\normalsize]{$\alpha_{n-1}$};
\draw(4,6)node[font=\scriptsize]{$Y$};
\draw(3.65,2.7)node[font=\scriptsize]{$e$};
\draw(5.8,2.7)node[font=\scriptsize]{$c_{n-1}$};
\draw(6,3)node[font=\normalsize]{$\ol{\alpha_{n-2}}$};

\draw(5,6)node[font=\scriptsize]{$X_n$};
\end{tikzpicture}
\caption{\label{7wn}$W_{n}$ and $\alpha_{n-1}$ bumped from $\omega(c_n)$. }
\end{minipage}
\caption*{}
\end{figure}

Let us apply $\x\ol{\alpha}$ to $\bp \w{P}\ox\alpha$. For this insertion, $\ol{B_k}$ holds for all $1\le k\le n-1$.
Indeed, this follows by applying Lemma~\ref{7l}(4) to the insertion $\x\ol{\alpha}$ into
\[
\bp \left(\left(\left.P\right|_{D_0\setminus\{c_n\}}\right)\sphat\ \right)\ox\alpha=\left.\left(\bp \w{P}\ox\alpha\right)\right|_{D_0\cup\omega(D_0)\setminus\{c_n\}}.
\]
Let $W_{n-1}$ be the tableau right after the replacement of $\ol{\alpha_{n-1}}$ in $c_{n-1}$. See Figure~\ref{7wn-1}. 
After being bumped from $c_{n-1}$, $\ol{\alpha_{n-1}}$ enters $X_{n-1}$. 
This follows from two observations. On the one hand, $\x\ol{\alpha}$ and $\gmx\alpha$ use opposite bumping rules: if one uses ordinary mixed bumping, then the other uses dual mixed bumping, and conversely. On the other hand, $\alpha_{n-1}$ and $\ol{\alpha_{n-1}}$ have different colors.
By \eqref{gm2-1}, $W_{n-1}(\omega(c_n))=\alpha_{n-1}$. 
Let $e$ be the cell immediately above $c_n$. Then the cell immediately to the left of $\omega(c_n)$ is $\omega(e)$. By \eqref{gm2-1}, the entry $W_{n-1}(\omega(e))$ is equal to the letter placed in the cell $e$ of $P\gmx\alpha$. This letter is smaller than $\alpha_{n-2}$, the entry of $P\gmx\alpha$ in the cell $c_{n-1}$, since $e$ lies in the same row as $c_{n-1}$ and in the same column as, or to the left of, $c_{n-1}$.
Therefore, $\ol{\alpha_{n-1}}$ bumped from $c_{n-1}$ bumps $\alpha_{n-1}$ in $\omega(c_n)$.

Let $W_n$ be the tableau right after the replacement of $\alpha_{n-1}$ in $\omega(c_n)$. See Figure~\ref{7wn}. 
The letter $\alpha_{n-1}$ bumped from $\omega(c_n)$ enters $Y$, the column containing $c_n$ (as defined above). This again follows from the fact that $\x\ol{\alpha}$ and $\gmx\alpha$ use opposite bumping rules. Since, under $\gmx\alpha$, the letter $\alpha_{n-1}$ bumped from $c_{n-1}$ enters the next row $X_{n-1}$, it follows that, under $\x\ol{\alpha}$, the letter $\alpha_{n-1}$ bumped from $\omega(c_n)$ enters the next column $Y$. 
Since $e$ lies in the same row as $c_{n-1}$ and in the same column as, or to the left of, $c_{n-1}$, $W_n(e)\le W_n(c_{n-1})=\ol{\alpha_{n-2}}$. 
Hence, $\alpha_{n-1}$ bumps the entry $\beta$ in the cell $c_n$.
Note that $\ox\alpha$ and $\x\ol{\alpha}$ use the same bumping rules, namely, either ordinary mixed bumping or dual mixed bumping. Thus, $\ox\alpha$ uses ordinary mixed bumping in Case~1 and dual mixed bumping in Case~2. Together with the definition of $\beta$ in \eqref{betabeta2} and the fact that $\alpha_n$ is unbarred, as it lies on the diagonal, this implies that, in either case, the entry $\beta$ bumped from $c_n$ enters the next column $X_n$.

To sum up, we have 
\begin{align}\label{gm2-2}
\bp \w{P}\ox\alpha\x\ol{\alpha}=\left.\left(P\gmx\alpha\right)\sphat\ \right|_{D_1}&\cup\ol{\omega}\left(\left.\left(P\gmx\alpha\right)\sphat\ \right|_{D_1}\right)\notag\\&\cup\left(\left.\left(\bp \w{P}\right)\right|_{D_2\cup\omega(D_2)}\ \xleftarrow[\!\!\!\!\!\omega(X_n)\!\!\!\!\!]{}\,\ol{\beta}\xleftarrow[\!X_n\!]{}\beta\right). 
\end{align}

In Case~1, $\gmx\alpha$ uses dual mixed bumping, so $\alpha_n$ bumped from $c_n$ enters $X_n$ and remains unbarred. In Case~2, $\gmx\alpha$ uses ordinary mixed bumping, so $\alpha_n$ bumped from $c_n$ becomes barred before it enters $X_n$. Together with \eqref{betabeta2}, this implies that
\begin{align}\label{gm2-3}
\left.\left(P\gmx\alpha\right)\right|_{D_2}
=
P|_{D_2}\xleftarrow[\!X_n\!]{\!\!\!\!\!\text{GM}\!\!\!\!\!}\ol{\beta}.
\end{align}
In the insertion $\xleftarrow[\!X_n\!]{\!\!\!\!\!\text{GM}\!\!\!\!\!}\ol{\beta}$ applied to $P|_{D_2}$, the number of diagonal letters that are bumped is $N-1$. Therefore, by the induction hypothesis, we obtain
\begin{align}\label{gm2-4}
\bp\left(P|_{D_2}\xleftarrow[\!X_n\!]{\!\!\!\!\!\text{GM}\!\!\!\!\!}\ol{\beta}\right)
=
\left(\bp (P|_{D_2})\sphat\xleftarrow[\!\!\!\!\!\omega(X_n)\!\!\!\!\!]{}\,\ol{\beta}\xleftarrow[\!X_n\!]{}\beta\right)\sphat.
\end{align}
Using the easily verified identity
\[
\bp (P|_{D_2})\sphat=\left.\left(\bp \w{P}\right)\right|_{D_2\cup\omega(D_2)}
\]
together with \eqref{gm2-3} and \eqref{gm2-4}, we see that the last term on the right-hand side of \eqref{gm2-2} is equal to
\[
\bp\left(\left.\left(P\gmx\alpha\right)\sphat\ \right|_{D_2}\right).
\]
This shows that \eqref{gm2-2} is precisely \eqref{7.2}, as desired.
\end{proof}

Now we define a transformation of words. Similarly to $d_1$ defined in \S 6, this transformation sends words of length $n$ to biwords of length $2n$.

\begin{dfn}\label{d2}
Define an injective map $d_2 : B_n\to DC_{2n}$ as follows: 

Let $\pi=\alpha_1\dots\alpha_n\in B_n$. 
For each $k$ satisfying $1\leq k\leq n$, let $a_k=|\alpha_k|$. Define column vectors $v_{2k-1}$ and $v_{2k}$ by
\begin{align*}
v_{2k-1}=
\begin{pmatrix}
k \\
\ol{a_k}  \\
\end{pmatrix}, 
v_{2k}=
\begin{pmatrix}
k \\
a_k  \\
\end{pmatrix}
\end{align*}
if $\alpha_k=a_k$, and by 
\begin{align*}
v_{2k-1}=
\begin{pmatrix}
\ol{k} \\
a_k  \\
\end{pmatrix}, 
v_{2k}=
\begin{pmatrix}
\ol{k} \\
\ol{a_k} \\
\end{pmatrix}
\end{align*}
if $\alpha_k=\ol{a_k}$. 
We define $d_2(\pi)$ to be the doubly colored biword of length $2n$, whose columns are $v_1, \dots, v_{2n}$ arranged in this order. 
\end{dfn}

The generalized mixed insertion can be embedded into the doubly mixed insertion as follows: 
\begin{thm}\label{GMDM}
For $\pi\in B_n$, we have 
\begin{align*}
\bp P_\text{GM}(\pi)=P_\text{M*}(d_2(\pi)), \\
\bbackslash Q_\text{GM}(\pi)=Q_\text{M*}(d_2(\pi)).
\end{align*}
\end{thm}
\begin{proof}
Write $\pi=\alpha_1\dots\alpha_n$, and let $R_1$ be the first row of $\mathcal{Q}$. Define $P_0=\emptyset$, $P_k=P_\text{GM}(\alpha_1\dots\alpha_k)$, $T_0=\emptyset$, and $T_k=P_\text{M*}(d_2(\alpha_1\dots\alpha_k))$ for $1\le k\le n$.
Then, by the definition of the doubly mixed insertion, we have
\[T_k=\emptyset\xleftarrow[\!\!\!\!\!\omega(R_1)\!\!\!\!\!]{}\ol{\alpha_1}\xleftarrow[R_1]{}\alpha_1\cdots\xleftarrow[\!\!\!\!\!\omega(R_1)\!\!\!\!\!]{}\ol{\alpha_k}\xleftarrow[R_1]{}\alpha_k.\]
By Proposition~\ref{7m}, we have
\[\emptyset\xleftarrow[\!\!\!\!\!\omega(R_1)\!\!\!\!\!]{}\ol{\alpha_1}\xleftarrow[R_1]{}\alpha_1\cdots\xleftarrow[\!\!\!\!\!\omega(R_1)\!\!\!\!\!]{}\ol{\alpha_k}\xleftarrow[R_1]{}\alpha_k=\bp P_k.\]
Hence $\bp P_k=T_k$ holds for $0\leq k\leq n$. In particular, the first equation $\bp P_\text{GM}(\pi)=P_\text{M*}(d_2(\pi))$ holds.
 
In order to show the second equation, we define $c_k$ by $\{c_k\}=D(P_k)\setminus D(P_{k-1})$ for $1\leq k\leq n$. 
Because $\bp P_k=T_k$ holds for $0\leq k\leq n$, we can see $D(T_k)\setminus D(T_{k-1})=\{c_k, \omega(c_k)\}$ for $1\leq k\leq n$. 
By definition, $Q_\text{GM}(\pi)$ is the tableau obtained by filling, for each $1\leq k\leq n$, the cell $c_k$ with $k$ if $\alpha_k$ is unbarred, and with $\ol{k}$ if $\alpha_k$ is barred.
Similarly, $Q_\text{M*}(d_2(\pi))$ is the tableau obtained by filling, for each $1\leq k\leq n$, the cells $c_k$ and $\omega(c_k)$ with $k$ if $\alpha_k$ is unbarred, and with $\ol{k}$ if $\alpha_k$ is barred.
Hence, by the definition of $\bbackslash$\,, the second equation $\bbackslash Q_\text{GM}(\pi)=Q_\text{M*}(d_2(\pi))$ holds.
\end{proof}

\section{Duality}

In Theorem~\ref{WSMDM} and Theorem~\ref{GMDM}, we showed that the Worley--Sagan mixed insertion and the generalized mixed insertion can be embedded into the doubly mixed insertion. 
By this embeddings and the self-duality of the doubly mixed insertion (Theorem~\ref{DMd}), we can show that the duality between the Worley--Sagan mixed insertion and the generalized mixed insertion. 

First of all, we define the operation of taking inverse in $B_n$. 

\begin{dfn}
Let $\pi=\alpha_1\dots\alpha_n\in B_n$. 
The inverse of $\pi$ is defined to be $\beta_1\dots\beta_n\in B_n$ and written $\pi^{-1}$, where $\beta_k$ is defined by 
\begin{align*}
\beta_{a_k}=\begin{cases} k & \text{if $\alpha_k$ is unbarred, }\\
\ol{k}&
\text{if $\alpha_k$ is barred. } \\
\end{cases}
\end{align*}
for each $k$ ($1\leq k\leq n$). 
\end{dfn}
\begin{rmk}
If all $\alpha_k$ are unbarred,  
$\pi$ can be seen as an element of $S_n$, the symmetric group of degree $n$. 
In this case, $\pi^{-1}$ can also be seen as an element of $S_n$ and coincides with the inverse of $\pi$ taken as an element of $S_n$. That is, the definition of inverse defined above is an extension of the usual definition of inverse of the elements of $S_n$. 
\end{rmk}

With this definition, a duality between the Worley--Sagan mixed insertion and the generalized mixed insertion can be expressed as follows:  \begin{thm}\label{WSMGMd}
Let $\pi\in B_n$. If $\pi\xlr{\!\!\!\!\!\text{WSM}\!\!\!\!\!}(P, Q)$, then $\pi^{-1}\xlr{\!\!\!\!\!\text{GM}\!\!\!\!\!}(Q, P)$ holds. 
\end{thm}

\begin{proof}
\begin{figure}[!t]
\centering
\[\begin{tikzpicture}
\node{\begin{tikzcd}[row sep=huge, column sep=huge]
\coprod\limits_{\lambda\in\mathcal{P}_d(n)}\text{SSHTab}''(\lambda)\times\text{SSHTab}'(\lambda)
\arrow[r, "\text{sw}"] \arrow[ddd, bend right=60, swap, start anchor={[yshift=2ex]}, "\text{$\bbackslash\times\bp$}"]
&\coprod\limits_{\lambda\in\mathcal{P}_d(n)}\text{SSHTab}'(\lambda)\times\text{SSHTab}''(\lambda)
 \arrow[ddd, bend left=60, start anchor={[yshift=2ex]}, "\text{$\bp\times\bbackslash$}"]\\
B_n \arrow[r,"\text{inv}"] \arrow[u,"\text{WSM}", end anchor={[yshift=2ex]}] \arrow[d,"d_1", swap]
&B_n\arrow[d,"d_2"]
\arrow[u, "\text{GM}", swap, end anchor={[yshift=2ex]}] \\
DC_{2n}\arrow[r, "\text{inv}"] \arrow[d, "\text{M*}", swap]&DC_{2n} \arrow[d, "\text{M*}"]\\
\coprod\limits_{\mu\in\mathcal{P}(n)}\text{UNSHTab}(\mu)\times\text{UNSHTab}(\mu)\arrow[r, "\text{sw}"]& 
\coprod\limits_{\mu\in\mathcal{P}(n)}\text{UNSHTab}(\mu)\times\text{UNSHTab}(\mu) 
\end{tikzcd}};
\draw(0,2.2)node{(1)};
\draw(0,0)node{(2)};
\draw(0,-2.2)node{(3)};
\draw(-6,0.5)node{(4)};
\draw(6,0.5)node{(5)};
\end{tikzpicture}\]
\caption{\label{commutative}}
\end{figure}
Consider the commutative diagram in Figure~\ref{commutative}.
Here, $\text{WSM}$, $\text{GM}$ and $\text{M*}$ are the bijections that appeared in Theorem\ref{WSMRS}, Theorem~\ref{GMRS} and Theorem~\ref{DMRS}, respectively. $\text{inv}$ is the map that sends a colored permutation $\pi$ to its inverse $\pi^{-1}$, and $\text{sw}$ is the map that swaps the order of the tableau pairs, i.e.,  $\text{sw}(P, Q)=(Q, P)$. 
$\bbackslash\times\bp$ is the map that sends $(P, Q)$ in $\text{SSHTab}''(\lambda)\times\text{SSHTab}'(\lambda)$ for some $\lambda\in\mathcal{P}_d(n)$ to $(\,\bbackslash P, \bp Q)$ in 
$\text{UNSHTab}(\mu)\times\text{UNSHTab}(\mu)$, where $\mu$ is the partition that satisfies $D(\mu)=D'(\lambda)\cup\omega(D'(\lambda))$. 
The definition of $\bp\times\bbackslash$ is similar, only the order of operators is swapped. 

We know that the squares (3), (4) and (5) are commutative by Theorem~\ref{DMd}, Theorem~\ref{WSMDM} and Theorem~\ref{GMDM}, respectively. What we want to prove now is equivalent to the commutativity of the square (1). A simple calculation shows that the square (2) is also commutative. First of all, we will show this fact. 

The commutativity of (2) is equivalent to say that $d_1(\pi)^{-1}=d_2(\pi^{-1})$ holds for any $\pi\in B_n$. 
Let $\pi=\alpha_1\dots\alpha_n$ and $\pi^{-1}=\beta_1\dots\beta_n$, where $\alpha_k, \beta_k\in\p\cup\ol{\p}$ ($1\leq k\leq n$). For each integer $k$ with $1\leq k\leq n$, we use the $2$-dimensional column vectors $u_{2k-1}, u_{2k}$ from Definition~\ref{d1} and $v_{2k-1}, v_{2k}$ from Definition~\ref{d2} again. 
Let $a_k=|\alpha_k|$. 
$d_1(\pi)$ is a doubly colored biword obtained by sorting $u_1,\dots, u_{2n}$ which satisfies 
\[\left(u_{2k-1}, u_{2k}\right)=\begin{cases}\hspace{5pt}
\begin{pmatrix}
 \ol{k} & k \\
a_k & a_k 
\end{pmatrix}
& \text{if $\alpha_k$ is unbarred, } \\[4ex]
\hspace{5pt}
\begin{pmatrix}
 \ol{k} & k \\
\ol{a_k} & \ol{a_k} 
\end{pmatrix}&
\text{if $\alpha_k$ is barred. } \\
\end{cases}
\]
On the other hand, $d_2(\pi^{-1})$ is a doubly colored biword obtained by sorting $v_1, \dots, v_{2n}$ which satisfies 
\[\left(v_{2a_k-1}, v_{2a_k}\right)=\begin{cases}\hspace{5pt}
\begin{pmatrix}
 a_k & a_k\\
 \ol{k} & k 
\end{pmatrix}
& \text{if $\alpha_k$ is unbarred, }\\[4ex]
\hspace{5pt}
\begin{pmatrix}
 \ol{a_k} & \ol{a_k} \\
k & \ol{k}
\end{pmatrix}&
\text{if $\alpha_k$ is barred. }\\
\end{cases}
\]

If $\alpha_k$ is unbarred, the two line array obtained by swapping the first and second rows of $(u_{2k-1}, u_{2k})$ is equal to $(v_{2a-1}, v_{2a})$. 
If $\alpha_k$ is barred, the two line array obtained by swapping the first and second rows of $(u_{2k-1}, u_{2k})$ is equal to $(v_{2a}, v_{2a-1})$. Therefore, $d_1(\pi)^{-1}=d_2(\pi^{-1})$ holds. 

Now, we know that the squares (2), (3), (4) and (5) are all commutative in the diagram above. 
Furthermore, obviously the outermost square (let us call it (6)) is also commutative, meaning that $\text{sw}\circ(\,\bbackslash\times\bp)=(\,\bp\times\bbackslash\,)\circ\text{sw}$. 
So, for $\pi\in B_n$, 
\begin{align*}
(\,\bp\times\bbackslash\,)\circ\text{sw}\circ\text{WSM}(\pi)
&\overset{(6)}{=}\text{sw}\circ(\,\bbackslash\times\bp)\circ\text{WSM}(\pi)\\
&\overset{(4)}{=}\text{sw}\circ\text{M*}\circ d_1(\pi)\\
&\overset{(3)}{=}\text{M*}\circ\text{inv}\circ d_1(\pi)\\
&\overset{(2)}{=}\text{M*}\circ d_2\circ\text{inv}(\pi)\\
&\overset{(5)}{=}(\,\bp\times\bbackslash\,)\circ\text{GM}\circ\text{inv}(\pi). 
\end{align*} 
Since $\bp\times\bbackslash$ is injective, we can conclude that $\text{sw}\circ\text{WSM}(\pi)=\text{GM}\circ\text{inv}(\pi)$. 
\end{proof}
\begin{ex}
$\pi$ that appears in Example~\ref{wsmeg} and $\sigma$ that appears in Example~\ref{gmeg} are inverse of each other. Comparing \eqref{wsmeq} and \eqref{gmeq}, it can be seen that the order of tableaux is swapped, as Theorem~\ref{WSMGMd} asserts. 
\end{ex}

\section*{Acknowledgements}
I am deeply grateful to my advisor, Itaru Terada, for his advice on every detail and for his patient guidance despite my procrastination in writing this paper. 

\bibliographystyle{amsplain}
\bibliography{references}

@article{h,
  author  = {Haiman, M. D.},
  title   = {On mixed insertion, symmetry, and shifted {Young} tableaux},
  journal = {J. Combin. Theory Ser. A},
  volume  = {50},
  number  = {2},
  pages   = {196--225},
  year    = {1989},
  doi     = {10.1016/0097-3165(89)90015-0}
}

@article{i,
  author  = {Ishikawa, M.},
  title   = {Mixed {Robinson--Schensted} correspondence, {Fomin} version, and mixed {Knuth} correspondence for {$(A,B)$}-partially strict tableaux},
  journal = {J. Fac. Educ. Tottori Univ. (Nat. Sci.)},
  volume  = {44},
  pages   = {17--64},
  year    = {1995}
}

@book{m,
  author    = {Macdonald, I. G.},
  title     = {Symmetric functions and {Hall} polynomials},
  publisher = {Oxford University Press},
  edition   = {second},
  year      = {1995}
}

@mastersthesis{n,
  author = {Nakagiri, M.},
  title  = {An extension of the {Worley--Sagan} insertion and {Haiman}'s mixed insertion for shifted {Young} tableaux and the duality between them},
  school = {The University of Tokyo},
  year   = {2023},
  note   = {in Japanese}
}

@article{sa,
  author  = {Sagan, B. E.},
  title   = {Shifted tableaux, {Schur} {$Q$}-functions, and a conjecture of {R. Stanley}},
  journal = {J. Combin. Theory Ser. A},
  volume  = {45},
  number  = {1},
  pages   = {62--103},
  year    = {1987},
  doi     = {10.1016/0097-3165(87)90047-1}
}

@article{sch,
  author  = {Schensted, C.},
  title   = {Longest increasing and decreasing subsequences},
  journal = {Canad. J. Math.},
  volume  = {13},
  pages   = {179--191},
  year    = {1961},
  doi     = {10.4153/CJM-1961-015-3}
}

@article{sh-wh,
  author  = {Shimozono, M. and White, D. E.},
  title   = {A color-to-spin domino {Schensted} algorithm},
  journal = {Electron. J. Combin.},
  volume  = {8},
  number  = {1},
  pages   = {Research Paper 21, 50 pp.},
  year    = {2001},
  doi     = {10.37236/1565}
}

@phdthesis{wo,
  author = {Worley, D. R.},
  title  = {A theory of shifted {Young} tableaux},
  school = {Massachusetts Institute of Technology},
  year   = {1984}
}

\par
\begin{flushleft}
\normalsize
Graduate School of Mathematical Sciences\\
The University of Tokyo\\
3-8-1 Komaba, Meguro-ku\\
Tokyo 153-8914, Japan\\
\textit{E-mail address:} \texttt{mnakagiri0@gmail.com}
\end{flushleft}

\end{document}